\documentclass[reqno, a4paper, 10pt]{amsart}
\usepackage{amssymb,url, color, mathrsfs}
\usepackage[colorlinks=true, bookmarks=true, pdfstartview=FitH, pagebackref=true, linktocpage=true]{hyperref}
\usepackage[short,nodayofweek]{datetime}
\usepackage[pagewise,running,mathlines,displaymath, switch]{lineno}
\usepackage{times}
\usepackage{indentfirst, tabularx}
\usepackage[table]{xcolor}
\usepackage{float, hhline, colortbl}
\usepackage{graphicx}
\usepackage{arydshln} 
\usepackage{setspace}

\usepackage{cite}

\hypersetup{
pdftitle={Higher order Sobolev trace inequalities on balls revisited}, 
pdfauthor={Ngo, Nguyen, Phan},
colorlinks = true,
linkcolor = magenta,
citecolor = blue,
}

\theoremstyle{plain}
\newtheorem{theorem}{Theorem}[section]
\newtheorem{proposition}[theorem]{Proposition}
\newtheorem{lemma}[theorem]{Lemma}
\newtheorem{corollary}[theorem]{Corollary}

\theoremstyle{remark}

\def\lt{\left}
\def\rt{\right}

\def\bS{\mathbb S}
\def\bB{\mathbb B}
\def\R{\mathbf R}

\numberwithin{equation}{section}

\title{Higher order Sobolev trace inequalities on balls revisited}

\def\cfac#1{\ifmmode\setbox7\hbox{$\accent"5E#1$}\else\setbox7\hbox{\accent"5E#1}\penalty 10000\relax\fi\raise 1\ht7\hbox{\lower1.05ex\hbox to 1\wd7{\hss\accent"13\hss}}\penalty 10000\hskip-1\wd7\penalty 10000\box7 }

\author[Q. A. Ng\^o]{Qu\cfac oc Anh Ng\^o}

\address[Q.A. Ng\^o]{Department of Mathematics\\
College of Science, Vi\^{e}t Nam National University\\
H\`{a} N\^{o}i, Vi\^{e}t Nam.}

\email{\href{mailto: Q.A. Ng\^o <nqanh@vnu.edu.vn>}{nqanh@vnu.edu.vn}}
\email{\href{mailto: Q.A. Ng\^o <bookworm\_vn@yahoo.com>}{bookworm\_vn@yahoo.com}}

\author[V. H. Nguyen]{Van Hoang Nguyen}

\address[V.H. Nguyen]{Institute of Mathematics\\
Vietnam Academy of Science and Technology\\
Hanoi, Vietnam.}

\email{\href{mailto: V.H. Nguyen <vanhoang0610@yahoo.com >}{vanhoang0610@yahoo.com}}
\email{\href{mailto: V.H. Nguyen <nvhoang@math.ac.vn>}{nvhoang@math.ac.vn}}

\author[Q. H. Phan]{Quoc Hung Phan}
\address[Q. H. Phan]{Institute of Research and Development\\
Duy Tan University\\
Da Nang, Vietnam.}

\email{\href{mailto: P.Q. Hung <phanquochung@dtu.edu.vn>}{phanquochung@dtu.edu.vn}} 

\begin{document}



\date{\bf \today \ at \currenttime}

\renewcommand{\thefootnote}{\arabic{footnote}}
\setcounter{footnote}{0}

\begin{abstract}
Inspired by a recent sharp Sobolev trace inequality of order four on the balls $\bB^{n+1}$ found by Ache and Chang \cite{ac2015}, we propose a slightly different approach to reprove Ache--Chang's trace inequality. To illustrate this approach, we reprove the classical Sobolev trace inequality of order two on $\bB^{n+1}$ and provide sharp Sobolev trace inequalities of orders six and eight on $\bB^{n+1}$. As the limiting case of the Sobolev trace inequality, a Lebedev--Milin type inequality of order up to eight is also considered.
\end{abstract}

\maketitle


\section{Introduction}

The motivation of writing this paper traces back to a recent work due to Ache and Chang \cite{ac2015} concerning the sharp Sobolev trace inequality of order four on the unit ball $\bB^{n+1}$ in $\R^{n+1}$. As indicated in \cite{ac2015}, by the order of all inequalities mentioned in the present paper, we refer to the order of the operator involved in the derivation of these inequalities.  In the next few paragraphs, we briefly recall the theory of Sobolev trace inequalities to understand why this finding is significant. 

Of importance in analysis and conformal geometry are Sobolev and Sobolev trace inequalities either on Euclidean spaces or on Euclidean balls. These inequalities, in brief, provide compact embeddings between important functional spaces. For the classical Sobolev inequality (of order two), its version on $\R^n$ is given as follows
\begin{equation}\label{eq:SobolevSpaceOrder2}
\frac{\Gamma (\frac{n+2}2)}{\Gamma (\frac{n-2}2)}\omega_n^{2/n}\Big( \int_{\R^n} |u|^{\frac{2n}{n-2}} dz \Big)^{\frac{n-2}n} \leqslant \int_{\R^n} |\nabla u|^2 dz 
\end{equation}
for any smooth function $u$ with compact support. Here, and throughout this paper, $\omega_n$ is the volume of the unit sphere $\bS^n$, the boundary of the unit ball $\bB^{n+1}$, in $\R^{n+1}$, which is $2\pi^{(n+1)/2}/\Gamma ((n+1)/2)$, which is also $2^n \pi^{n/2} \Gamma (n/2)/\Gamma (n)$. It is well-known that the inequality \eqref{eq:SobolevSpaceOrder2} is crucial in the resolution of the Yamabe problem on closed manifolds. Not limited to the Yamabe problem, Inequality \eqref{eq:SobolevSpaceOrder2} is the fundamental tool and have a significant role in various problems in analysis and geometry. Since the vast subject of Sobolev inequalities can be easily found in the literature, we do not mention it here.

Inspired by the sharp Sobolev inequality \eqref{eq:SobolevSpaceOrder2} on $\R^n$, the following sharp Sobolev trace inequality on $\R^{n+1}_+$ is well-known
\begin{equation}\label{eq:SobolevTraceSpaceOrder2}
\frac{\Gamma (\frac{n+1}2)}{\Gamma (\frac{n-1}2)} \omega_n^{1/n} \Big( \int_{\R^n} |U(x,0)|^{\frac{2n}{n-1}} dx \Big)^{\frac{n-1}n} \leqslant \int_{\R^{n+1}_+} |\nabla U|^2 dxdy.
\end{equation}
Here in \eqref{eq:SobolevTraceSpaceOrder2} we denote by $(x,y) \in \R^n \times \R$ a point in $\R^{n+1}$ and by $(x,y) \in \R_+^{n+1}$ we mean $y>0$. To study \eqref{eq:SobolevTraceSpaceOrder2}, it is routine to study the quotient
\[
Q(U) = \frac
{ \int_{\R^{n+1}_+} |\nabla U|^2 dxdy}
{\big( \int_{\R^n} |U(x,0)|^{\frac{2n}{n-1}} dx \big)^{(n-1)/n} }
\]
and its Sobolev quotient
\begin{equation}\label{eq:SobolevTraceSpaceOrder2-Variation}
Q(\R_+^{n+1}) =\inf \big\{ Q(U) : U \in C_0^\infty (\overline{\R_+^{n+1}}), U \not \equiv 0 \big\}.
\end{equation}
It turns out that $Q(\R_+^{n+1}) =  \big( \Gamma (\frac{n+1}2) / \Gamma (\frac{n-1}2) \big) \omega_n^{1/n} $.

The existence of optimizers for \eqref{eq:SobolevTraceSpaceOrder2-Variation} was first studied by Lions \cite{Lions} by using the concentration-compactness principle. Later on, Escobar classified all optimizers for $Q(\R_+^{n+1})$ and computed explicitly the sharp constant $Q(\R_+^{n+1})$; see \cite[Theorem 1]{Escobar88}. To obtain such results, Escobar exploits the conformally equivalent property between $\R_+^{n+1}$ and $\bB^{n+1}$ to transfer the trace inequality \eqref{eq:SobolevTraceSpaceOrder2} on $\R_+^{n+1}$ to a suitable trace inequality on $\bB^{n+1}$, namely, the inequality \eqref{eq:TraceOrder2} in Theorem \ref{thmTraceOrder2} below. Then he studied the similar Sobolev quotient
\[
Q ( \mathbb B^{n+1}) = \inf_{ v \in C^1 (\overline {\mathbb B^1})}
\left\{ 
\frac
{\int_{\bB^{n+1}} |\nabla v|^2 dz + ((n-1)/2) \int_{\bS^n} |v|^2 d\omega}
{ \big(\int_{\bS^n}|v|^{\frac{2n}{n-1}} d\omega \big)^{(n-1)/n}}
\right\}
\]
and proved that 
\[
Q(\R_+^{n+1}) = Q ( \mathbb B^{n+1}).
\]
Finally, he showed that an optimizer for $Q ( \mathbb B^{n+1})$ exists and by Obata's method he was able to classify all optimizers.

In \cite{beckner1993}, Beckner took a completely different approach based on spherical harmonics and the dual-spectral form of the Hardy--Littlewood--Sobolev inequality on $\mathbb S^n$, which was used earlier in \cite{beckner1992}, to reprove \eqref{eq:TraceOrder2}; see \cite[Theorem 4]{beckner1993}. Combining Beckner and Escobar' result, the following sharp Sobolev trace inequality of order two is already known.

\begin{theorem}[Sobolev trace inequality of order two]\label{thmTraceOrder2}
Let $f\in C^\infty(\bS^n)$ with $n > 1$, suppose that $v$ is a smooth extension of $f$ to the unit ball $\bB^{n+1}$. Then we have the following sharp trace inequality
\begin{equation}\label{eq:TraceOrder2}
\frac{\Gamma(\frac{n+1}2)}{\Gamma(\frac{n-1}2)}  \omega_n^{1/n} \Big(\int_{\bS^n}|f|^{\frac{2n}{n-1}} d\omega \Big)^{\frac{n-1}{n}}\leqslant \int_{\bB^{n+1}} |\nabla v|^2 dz + a_n \int_{\bS^n} |f|^2 d\omega ,
\end{equation}
where $a_n = \Gamma ((n+1)/2)/\Gamma((n-1)/2) = (n-1)/2$. Moreover, equality in \eqref{eq:TraceOrder2} holds if, and only if, $v$ is a harmonic extension of a function of the form
\[
f_{z_0} (\xi) = c|1-\langle z_0, \xi \rangle |^{-(n-1)/2},
\]
where $c>0$ is a constant, $\xi \in \mathbb S^n$, and $z_0$ is some fixed point in the interior of $\mathbb B^{n+1}$.
\end{theorem}

We note that using his approach, Beckner also obtained a sharp form of the Sobolev inequality on $\mathbb S^n$, namely
\begin{equation}\label{eq:SobolevSphereOrder2}
\frac{\Gamma (\frac{n+2}2)}{\Gamma (\frac{n-2}2)} \Big(\int_{\bS^n}|f|^{\frac{2n}{n-2}} d\omega \Big)^{\frac{n-2}{n}}\leqslant  \int_{\bS^n} |\nabla f|^2 d\omega + \frac{n(n-2)}4 \int_{\bS^n} |f|^2  d\omega .
\end{equation}
As for \eqref{eq:SobolevSpaceOrder2}, Inequality \eqref{eq:SobolevSphereOrder2} also has some role in the study of the Yamabe problem on $\mathbb S^n$.

Apparently, for all inequalities \eqref{eq:SobolevSpaceOrder2}--\eqref{eq:TraceOrder2} mentioned above, the operators involved are either the Laplacian or the conformal Laplacian, both are of order two. In recent years, a number of works are devoted to understanding higher order operators such as the poly-Laplacian, the Paneitz operator, the GJMS operators. For example, the following Sobolev inequality for higher order fractional derivatives in $\R^n$
\begin{equation}\label{eq:SobolevSpaceOrderS/2}
\frac{\Gamma (\frac{n+2s}2)}{\Gamma (\frac{n-2s}2)} \omega_n^{2s/n} \Big(\int_{\R^n} |u|^{\frac{2n}{n-2s}} dz\Big)^{\frac{n-2s}n} \leqslant \int_{\R^n} |(-\Delta)^{s/2} u|^2 dz
\end{equation}
was explicitly stated in \cite{CT04}, for before that, but in an implicitly formin terms of fractional integrals, in \cite{Lieb1983}. Similarly, there is a sharp higher order Sobolev inequality on $\bS^n$ for a class of pseudo-differential operators $P_{2\gamma}$ defined for $\gamma \in (0, n/2)$ as follows
\begin{equation}\label{eq:SobolevSphereOrder2gamma}
\frac{\Gamma (\frac{n+2\gamma}2)}{\Gamma (\frac{n-2\gamma}2)}  \omega_n^{2\gamma/n} \Big(\int_{\bS^n} |f|^{\frac{2n}{n-2\gamma}} d \omega \Big)^{\frac{n-2\gamma}n} \leqslant \int_{\bS^n} f \,P_{2\gamma} f \, d\omega
\end{equation}
see \cite[Theorem 6]{beckner1993}. Here the operator $P_{2\gamma}$ is formally given by
\[
P_{2\gamma} = \frac{\Gamma (B + 1/2+\gamma)}{\Gamma (B + 1/2-\gamma)}
\]
with
\[
B = \sqrt{-\widetilde \Delta + \big(\frac{n-1}2 \big)^2}.
\]
Here, and as always, $\widetilde \Delta$ denotes the Laplacian on $\bS^n$ with respect to the standard metric $g_{\bS^n}$. In a special case when $\gamma =2$, we know that
\[
\begin{aligned}
P_4 =& \Big( \sqrt{-\widetilde \Delta + \big(\frac{n-1}2 \big)^2}  + \frac 32 \Big)\Big( \sqrt{-\widetilde \Delta + \big(\frac{n-1}2 \big)^2}  + \frac 12 \Big)\\
&\times \Big( \sqrt{-\widetilde \Delta + \big(\frac{n-1}2 \big)^2}  - \frac 12 \Big)\Big( \sqrt{-\widetilde \Delta + \big(\frac{n-1}2 \big)^2}  - \frac 32 \Big)\\
=&\Big( -\widetilde \Delta + \frac {n(n-2)}4\Big)\Big( -\widetilde \Delta + \frac {(n+2)(n-4)}4\Big),
\end{aligned}
\]
which, by \eqref{eq:SobolevSphereOrder2gamma}, implies that
\begin{equation}\label{eq:SobolevSphereOrder4}
\begin{aligned}
\frac{\Gamma (\frac{n+4}2)}{\Gamma (\frac{n-4}2)}  \omega_n^{4/n} &\Big(\int_{\bS^n} |f|^{\frac{2n}{n-4}} d \omega \Big)^{\frac{n-4}n}\\
& \leqslant  \int_{\bS^n} \Big( (\widetilde \Delta f)^2 + \frac{n^2-2n-4}2 |\widetilde \nabla f|^2 + \frac{\Gamma (\frac{n+3}2)}{\Gamma (\frac{n-4}2)} f^2 \Big) \, d\omega 
\end{aligned}
\end{equation}
with $n > 4$. From \eqref{eq:SobolevSphereOrder4}, it is natural to ask with $n$ being greater than four \textit{whether or not there is higher order Sobolev trace inequality on $\bB^{n+1}$}.

A recent result due to Ache and Chang provides an affirmative answer to the above question. To be more precise and in terms of our notation convention, the following theorem, among other things, indicating a fourth-order Sobolev trace inequality on $\bB^{n+1}$, was proved in \cite[Theorem A]{ac2015}.

\begin{theorem}[Sobolev trace inequality of order four]\label{thmTraceOrder4}
Let $f\in C^\infty(\bS^n)$ with $n>3$ and suppose $v$ is a smooth extension of $f$ to the unit ball $\bB^{n+1}$, which also satisfies the Neumann boundary condition
\begin{equation}\label{eq:4NeumannCondition}
\partial_ \nu v \big|_{\bS^n} = -\frac{n-3}2 f.
\end{equation}
Then we have the sharp trace inequality
\begin{equation}\label{eq:TraceOrder4}
\begin{aligned}
2 \frac{\Gamma(\frac{n+3}2)}{\Gamma(\frac{n-3}2)} \omega_n^{3/n} & \Big(\int_{\bS^n} |f|^{\frac{2n}{n-3}} d\omega \Big)^{\frac{n-3}n} \\
&\leqslant \int_{\bB^{n+1}} |\Delta v|^2 dz + 2\int_{\bS^n} |\widetilde{\nabla} f|^2 d\omega + b_n \int_{\bS^n} |f|^2 d\omega ,
\end{aligned}
\end{equation}
where $b_n = (n+1)(n-3)/2$ and $\widetilde{\nabla}$ denotes spherical gradient on $\bS^n$. Moreover, equality in \eqref{eq:TraceOrder4} holds if, and only if, $v$ is a biharmonic extension of a function of the form
\[
f_{z_0} (\xi) = c|1-\langle z_0, \xi \rangle |^{-(n-3)/2},
\]
where $c>0$ is a constant, $\xi \in \mathbb S^n$, $z_0$ is some point in the interior of $\mathbb B^{n+1}$, and $v$ fulfills the boundary condition \eqref{eq:4NeumannCondition}.
\end{theorem}

To prove \eqref{eq:TraceOrder4}, Ache and Chang use a nontraditional way in the sense that they first derive a similar inequality for some metric $g^*$ on $\bB^{n+1}$, which is in the conformal class of the Euclidean metric, then they derive \eqref{eq:TraceOrder4}  by making use of the conformal covariant properties of the four-order Paneitz operator with respect to $g^*$ and the bilaplacian operator with respect to the Euclidean metric.

The aim of the present paper is twofold. First we revisit Escobar's approach based on the conformally convariant property of \eqref{eq:SobolevTraceSpaceOrder2} to provide new proofs for \eqref{eq:TraceOrder2} and \eqref{eq:TraceOrder4}. We note that although \eqref{eq:TraceOrder2} is already known by Beckner's fundamental paper, however, the proof given by Beckner is based on spherical harmonics. Our approach for \eqref{eq:TraceOrder2} is based on Escobar's. However, unlike Escobar's method which transforms the trace inequality \eqref{eq:SobolevTraceSpaceOrder2} to the trace inequality \eqref{eq:TraceOrder2}, our method is in the opposite direction. To be more precise, we show that we can obtain \eqref{eq:TraceOrder2} from \eqref{eq:SobolevTraceSpaceOrder2} after a suitable change of functions. In other words, the inequalities \eqref{eq:TraceOrder2} and \eqref{eq:SobolevTraceSpaceOrder2} are dual by the conformal equivalence between $\R^{n+1}_+$ and $\mathbb B^{n+1}$; see Section \ref{sec-Order2}. 

It turns out that we can do more with Escobar's idea. By exploiting further the conformal equivalence between $\R_+^{n+1}$ and $\mathbb B^{n+1}$, we are successful in providing a new proof for \eqref{eq:TraceOrder4}; see Section \ref{sec-Order4}. As noticed above, we prove \eqref{eq:TraceOrder4} by following a similar way that Escobar did, however, in an opposite direction. To this purpose, we make use of the following higher order Sobolev trace inequality in $\R_+^{n+1}$
\begin{equation}\label{eq:SobolevTraceSpaceOrder4}
2 \frac{\Gamma(\frac{n+3}2)}{\Gamma(\frac{n-3}2)} \omega_n^{3/n} \Big(\int_{\R^n} |U(x,0)|^{\frac{2n}{n-3}} dx\Big)^{\frac{n-3}n} \leqslant \int_{\R^{n+1}_+} |\Delta U(x,y)|^2 dx dy
\end{equation}
for functions $U$ having $\partial_y U(x,0)=0$. Furthermore, equality in \eqref{eq:SobolevTraceSpaceOrder4} holds if, and only if, $U$ is a biharmonic extension of a function of the form
\[c \big( 1 +  |\xi - z_0|^2 \big)^{-(n-3)/2},\]
where $c$ is a constant, $\xi \in \R^n$, $z_0 \in \R^n$, and $U$ also fulfills the boundary condition $\partial_y U(x,0)=0$. We believe that \eqref{eq:SobolevTraceSpaceOrder4} is already known but we are unable to find a reference for it until recently J. Case nicely informed us that \eqref{eq:SobolevTraceSpaceOrder4} can be derived from a general result in \cite{Case2018}. Therefore, we shall discuss Case's general result and provide a new proof for \eqref{eq:SobolevTraceSpaceOrder4} in Appendix \ref{apd-SobolevTraceSpace4}. In the last part of Section \ref{sec-Order4}, we also demonstrate that by using Beckner type trace inequality in Theorem \ref{thm2BecknerTypeNoWeight}, we can also recover \eqref{eq:TraceOrder4}. Compared to Escobar's approach, the analysis in Beckner's approach is less involved.

We note that Neumann's boundary condition for functions satisfied by \eqref{eq:SobolevTraceSpaceOrder4} comes from similar boundary conditions for functions satisfied by \eqref{eq:TraceOrder4}. Without restricting to the upper half-space $\R_+^{n+1}$, the following trace inequality is known
\begin{equation}\label{eq:SobolevTraceSpace}
2c_\alpha \omega_n^{(2\alpha-1)/n} \Big(\int_{\R^n} |U(x,0)|^{\frac{2n}{n+1-2\alpha}} dx\Big)^{\frac{n+1-2\alpha }n} \leqslant \int_{\R^{n+1} } U(x,y) (-\Delta)^\alpha U(x,y) dx dy 
\end{equation}
with 
\begin{equation}\label{eq:CAlpha}
c_\alpha = \sqrt \pi \frac{\Gamma(\alpha)\Gamma( \frac{n-1}2 +\alpha)}{\Gamma(\frac{n+1}2 - \alpha) \Gamma(\alpha - \frac 12)};
\end{equation}
see \cite{EL12}. We note that $c_1$ and $c_2$ are exactly the sharp constants in \eqref{eq:SobolevTraceSpaceOrder2} and in \eqref{eq:SobolevTraceSpaceOrder4} respectively. We note that the extra coefficient $2$ on the left hand side of \eqref{eq:SobolevTraceSpace} appears because the integral on the right hand side is over $\R^{n+1}$. It is our hope that there are dual trace inequalities of order six and this is the content of the second part of the paper.

To derive a suitable trace inequality of order six on $\R_+^{n+1}$, we revisit \eqref{eq:SobolevTraceSpace} when $\alpha=3$ and by a simple calculation, we expect that the following equality should hold
\begin{equation}\label{eq:SobolevTraceSpace6Expected}
\frac 83 \frac{\Gamma(\frac{n+5}2)}{\Gamma(\frac{n-5}2)} \omega_n^{5/n} \Big(\int_{\R^n} |U(x,0)|^{\frac{2n}{n-5}} dx\Big)^{\frac{n-5}n} \leqslant \int_{\R_+^{n+1} } |\nabla \Delta U|^2 (x,y) dx dy 
\end{equation}
for suitable function $U$ sufficiently smooth up to the boundary and decaying fast enough at infinity. Inspired by \cite{casechang2013}, we look for trace inequalities of order six for functions $U$ satisfying certain Neumann's boundary conditions. We shall prove the following trace inequality on the half space.

\begin{theorem}[Sobolev trace inequality of order six on $\R_+^{n+1}$]\label{thmTraceSpaceOrder6}
Let $U\in W^{3,2}(\overline{\R_+^{n+1}})$ be satisfied the Neumann boundary condition
\begin{subequations}\label{eq:6SpaceNeumannCondition}
\begin{align}
\partial_y U(x,0)=0, \quad \partial^2_y U(x,0)=\lambda \Delta_x U(x,0).
\tag*{\eqref{eq:6SpaceNeumannCondition}$_\lambda$}
\end{align}
\end{subequations}
Then we have the sharp trace inequality
\begin{subequations}\label{eq:TraceSpaceOrder6}
\begin{align}
(3\lambda^2-2\lambda +3) \frac{\Gamma(\frac{n+5}2)}{\Gamma(\frac{n-5}2)} \omega_n^{5/n} \Big(\int_{\R^n} |U(x,0)|^{\frac{2n}{n-5}} dx\Big)^{\frac{n-5}n} \leqslant \int_{\R^{n+1}_+} |\nabla \Delta U(x,y)|^2 dx dy.
\tag*{\eqref{eq:TraceSpaceOrder6}$_\lambda$}
\end{align}
\end{subequations}
Moreover, equality in \eqref{eq:TraceSpaceOrder6}$_\lambda$ holds if, and only if, $U$ is a triharmonic extension of a function of the form
\[
c \big( 1 +  |x - x_0|^2 \big)^{-(n-5)/2},
\]
where $c>0$ is a constant, $x \in \R^n$, $x_0$ is some fixed point in $\R^n$, and $U$ fulfills the boundary condition \eqref{eq:6SpaceNeumannCondition}$_\lambda$.
\end{theorem}

It is easy to see that $3\lambda^2-2\lambda +3 \geqslant 8/3$ with equality if $\lambda = 1/3$. Hence the sharp constant in \eqref{eq:TraceSpaceOrder6}$_\lambda$ is usually greater than that of \eqref{eq:SobolevTraceSpace6Expected}. We are aware that in the literature the Neumann boundary condition of the form \eqref{eq:6SpaceNeumannCondition}$_\lambda$ has already been used, for example, in  a work by Chang and Yang \cite{rayyang2013}. Once we can establish \eqref{eq:TraceSpaceOrder6}$_\lambda$, we hope that we can establish a similar trace inequality on $\mathbb B^{n+1}$ by using the natural conformal mapping between $\R_+^{n+1}$ and $\mathbb B^{n+1}$. By way of establishing the following trace inequality on $\mathbb B^{n+1}$, we shall prove that this is indeed the case.

\begin{theorem}[Sobolev trace inequality of order six]\label{thmTraceOrder6}
Let $f\in C^\infty(\bS^n)$ with $n> 5$ and suppose $v$ is a smooth extension of $f$ in the unit ball $\bB^{n+1}$, which also satisfies the boundary conditions
\begin{equation}\label{eq:boundarycond6}
\partial_ \nu v \big|_{\bS^n} = -\frac{n-5}2 f,\qquad \partial_ \nu^2 v \big|_{\bS^n} = \frac13 \widetilde \Delta f + \frac{(n-5)(n-6)}6 f.
\end{equation}
Then the following inequality holds
\begin{equation}\label{eq:TraceOrder6}
\begin{aligned}
\frac 83 \frac{\Gamma(\frac{n+5}2)}{\Gamma(\frac{n-5}2)}& \omega_n^{5/n} \Big(\int_{\bS^n} |f|^{\frac{2n}{n-5}}d\omega \Big)^{\frac{n-5}n}\\
 \leqslant &\int_{\bB^{n+1}} |\nabla \Delta v|^2 dx 
 +c_n^{(1)}\int_{\bS^n} (\widetilde\Delta f)^2 d\omega
+c_n^{(2)}  \int_{\bS^n} |\widetilde{\nabla} f|^2 d\omega 
+ c_n^{(3)} \int_{\bS^n} |f|^2 d\omega 
\end{aligned}
\end{equation}
with 
\begin{equation}\label{eq:CoefficientC_n}
\left\{
\begin{split}
c_n^{(1)}=&8(n+3)/9, \\
c_n^{(2)}=&4(n^3+n^2-21n-9)/9, \\
c_n^{(3)}=&(n-5)(n-3)(n+3)(n^2+4n-9)/18.
\end{split}
\right.
\end{equation}
Moreover, equality in \eqref{eq:TraceOrder6} holds if, and only if, $v$ is a triharmonic extension of a function of the form
\[
f_{z_0} (\xi) = c|1-\langle z_0, \xi \rangle |^{-(n-5)/2},
\]
where $c>0$ is a constant, $\xi \in \mathbb S^n$, $z_0$ is some fixed point in the interior of $\mathbb B^{n+1}$, and $v$ fulfills the boundary condition \eqref{eq:boundarycond6}.
\end{theorem}

As we shall see in the proof of Theorem \ref{thmTraceOrder6} that the boundary condition \eqref{eq:boundarycond6} comes from the boundary condition \eqref{eq:6SpaceNeumannCondition}$_{1/3}$ and the sharp constant of \eqref{eq:TraceSpaceOrder6}$_{1/3}$ is exactly the sharp constant of \eqref{eq:TraceSpaceOrder6}$_{1/3}$, which is $( 8/3) (\Gamma(\frac{n+5}2)/\Gamma(\frac{n-5}2)) \omega_n^{5/n}$. Since the analysis in Beckner's approach is much less involved compared with Escobar's approach, to prove \eqref{eq:TraceOrder6}, we revisit Beckner's approach to prove a Beckner type trace inequalities of order six; see Theorems \ref{thm6BecknerTypeNoWeight}. Then we use it to prove \eqref{eq:TraceOrder6} as demonstrated in Subsection \ref{subsec-6SobolevBall}.

As can be easily seen, Beckner's approach has several advantages when proving functional inequalities on balls and on spheres. This paper just provides another example to highlight its merits.  Another example, recently announced by Xiong \cite{Xiong2018}, concerns a derivation of the sharp Moser--Trudinger--Onofri inequalities from the fractional Sobolev inequalities. The work of Xiong generalizes a similar result for spheres of lower dimensions recently obtained by Chang and Wang in \cite{ChangWang2017}. We note that Xiong also used spherical harmonics instead of using Branson's dimensional continuation argument which becomes increasing delicate when the dimension is large as hightlighted in \cite[Remark 2]{ChangWang2017}.

After completing this paper, it has just come to our attention that, recently in a paper continuing his work on the boundary operators associated to the Paneitz operator in \cite{Case2018}, Jeffrey Case and his co-author also obtained some sharp Sobolev trace inequalities involving the interior $W^{3,2}$-seminorm, including an analogue of the Lebedev--Milin inequality on several standard models of manifolds of dimension six; see \cite{Case}. Following \cite{Case2018}, their approach is based on energy inequalities related to conformally covariant boundary operators associated to the sixth-order GJMS operator found in their paper. Therefore, it is completely different from ours.

The rest of the paper consists of four sections. Section \ref{sec-Pre} is devoted to preliminaries. Sections \ref{sec-Order2} and \ref{sec-Order4} are devoted to proofs of \eqref{eq:TraceOrder2} and \eqref{eq:TraceOrder4} based on Escobar's approach. Beckner type trace inequalities with or without a weight are also proved in theses sections; see Theorems \ref{thm2BecknerTypeWeight}, \ref{thm2BecknerTypeNoWeight}, \ref{thm4BecknerTypeNoWeight}, and \ref{thm6BecknerTypeNoWeight}. We also consider the limiting cases, known as the Lebedev--Milin inequality, in these sections as well; see Theorems \ref{thmSMOrder2} and \ref{thmSMOrder4}. Section \ref{sec-Order6} is devoted to a proof of \eqref{eq:TraceOrder6} based on Beckner's approach; see Theorem \ref{thm6BecknerTypeNoWeight}. A Lebedev--Milin type inequality of order six is also considered in this section; see Theorem \ref{thmSMOrder6}. Finally, in Section \ref{sec-Order8+}, we state sharp Sobolev trace inequalities of order eight on $\R_+^{n+1}$ and $\bB^{n+1}$ and Lebedev--Milin inequality of order eight without proofs; see Theorems \ref{thmTraceSpaceOrder8}, \ref{eq:TraceOrder8}, and  \ref{thmSMOrder8}.

\tableofcontents

We should point out that throughout out the paper, there are arguments and computations more or less known to experts in this field. However, we aim to include them for the reader's convenience while trying to maintain the paper in a reasonable length.

As a final comment before closing this section, it is worth emphasizing that in order to avoid any possible mistake, most of computation in the proof of Proposition \ref{apx-propExtension6}, in Subsections \ref{subsec-6Beckner} and \ref{subsec-6SobolevBall}, and especially in Section \ref{sec-Order8+} was done by using a scientific computer software. This allows us to carry out a similar research for higher order Sobolev trace inequalities, for example, Sobolev trace inequality of order ten on $\bB^{n+1}$, if there is strong motivation to work.


\section{Preliminaries} 
\label{sec-Pre}

First we need some notations and convention used throughout the paper. We often write $ X = (x,y) \in \R^{n+1}$ and denote $\mathbb B^{n+1} = \{X\in \R^{n+1}\,:\, |X| < 1\}$. By $\R_+^{n+1}$ we mean the set $\{(x,y) \in \R^{n+1} : y > 0\}$. We shall also denote by $\delta$ Kronecker's symbol and therefore Einstein's summation convention will be used often. 

Now we discuss the conformal equivalence between $\mathbb B^{n+1}$ and $\R_+^{n+1}$. To see why these sets are conformally equivalent, we work on $\R^{n+2}$. Therefore, a point $(x,y) \in \R^{n+1}$ will be identified with the point $(x,y,0)$ in $\R^{n+2}$. Furthermore, any point in $\R^{n+2}$ will be denoted by $(x,y,z)$ with $y, z, \in \R$ or by $(X,z)$ with $z \in \R$. 

Consider the stereographic projection $\mathcal S :\R^{n+1} \to \bS^{n+1} \subset \R^{n+2}$ given by
\[
\mathcal S(x,y) = \Big(\frac{2x}{1+ |x|^2 + y^2}, \frac{2y}{1+ |x|^2 + y^2}, \frac{|x|^2 + y^2 - 1}{1+ |x|^2+ y^2}\Big).
\]
The inverse of $S$, denoted by $S^{-1}$, is
\[
\mathcal S^{-1}(x,y,z) = \Big(\frac{x}{1 - z}, \frac{y}{1-z} \Big).
\]
We also denote by $R$ a quarter-turn of $\mathbb S^{n+1}$ in the plan containing the last two coordinate axes $Oy$ and $Oz$ in $\R^{n+2}$, that maps $(0,1,0)$ to $(0,0,1)$. Clearly, such a map $R$ is given by
\[
R (x,y,z) = (x, z, -y).
\]
Then we define $B: \R_+^{n+1} \to \mathbb B^{n+1}$ by
\[
B = \mathcal S^{-1} \circ R \circ \mathcal S \big|_{\mathbb B^{n+1}}.
\] 
It is not hard to verify that the mapping $B$ is well-defined and conformal. Furthermore, it is immediate to see that
\[
B(x,y) = \left(\frac{2x}{(1+y)^2 + |x|^2}, \frac{|x|^2+y^2-1}{(1+y)^2 + |x|^2}\right).
\]
We note that the mapping $B$ takes a similar form to the mapping $F^{-1}$ in \cite[p. 691]{Escobar88}. Clearly, the Jacobian matrix of $B$, denoted by $DB$, is given by
\[
\begin{split}
DB&(x,y)= \frac 2{\big[(1+y)^2 + |x|^2 \big]^2} \\
\times &
{\setstretch{2.00}
\left(
\begin{array}{ccc:c}
&\cdots & & -2x_1(1+y) \\
\vdots & \big[(1+y)^2 + |x|^2 \big] \delta_{ij} - 2x_i x_j & \vdots & \vdots \\
 & \cdots & & -2x_n(1+y) \\
\hdashline 
2x_1(1+y) & \cdots & 2x_n(1+y) & (1+y)^2-|x|^2
\end{array}
\right) } 
.
\end{split}
\]
Hence in short we can rewrite
\begin{equation}\label{eq:JacobianMatrixOfB}
DB(x,y) =
\frac 2{\big[(1+y)^2 + |x|^2 \big]^2}
{\setstretch{2.00}
\left(
\begin{array}{c:c}
 \big[(1+y)^2 + |x|^2 \big] I_n - 2x \otimes x & -2x(1+y) \\
\hdashline 
2x^t(1+y) & (1+y)^2-|x|^2
\end{array}
\right) } 
.
\end{equation}
We can easily verify that
\begin{equation}\label{eq:DBDBt}
DB \cdot DB^t = \Big( \frac2{(1+y)^2 + |x|^2} \Big)^2 I_{n+1},
\end{equation}
where $DB^t$ denotes the transpose of $DB$. From this and throughout this paper, if we denote
\[
\Phi (X) = \frac2{(1+y)^2 + |x|^2},
\]
then it is not hard to verify that the Jacobian of $B$ is given by
\[
J_B(X) = \Phi (X)^{n+1} .
\]
For simplicity, we shall also use the same letter $\mathcal S$ to denote the stereographic projection from $\R^n$ to $\bS^n$. Clearly, in this new perspective, $\mathcal S$ is given by
\[
\mathcal S(x) = \Big(\frac{2x}{1+ |x|^2}, -\frac{1 - |x|^2}{1+ |x|^2}\Big).
\]
Note that $\mathcal S(x) = B(x,0)$ and therefore the Jacobian of $\mathcal S$ is
\[
J_{\mathcal S}(x) = \Big(\frac{2}{1+ |x|^2}\Big)^n.
\]
We have the following simple observation.

\begin{lemma}\label{lem:NablaDelta}
Let $a \in \R$, then we have
\[
\nabla \Phi (X)^{a} = -a \Phi (X)^{a+1} ( x, 1+y )
\]
and
\[
\Delta \Phi (X)^{a} = -a(n-1-2a) \Phi (X)^{a+1}.
\]
\end{lemma}

\begin{proof}
This is elementary and follows from direct verification.
\end{proof}

For simplicity, let us emphasize that we sometime write the composition $f \circ g$ evaluated at a point $p$, that is $(f \circ g)(p)$, by $f(g)$ if no confusion occurs.

\begin{lemma}\label{lem:Identity}
We have the following identity
\[\begin{split}
\Phi ^{-2} &\Delta(F \circ B) = (\Delta F) (B )-(n-1)\Big(\sum_{j=1}^n (\partial_j F)(B ) x_j -(\partial_{n+1}F)(B )(1+y)\Big). 
\end{split}\]
In other words, we have
\[
\Phi ^{-2} \Delta(F \circ B) = (\Delta F) (B ) + (n-1) \langle \nabla F(B ), (-x,1+y) \rangle. 
\]
\end{lemma}

\begin{proof}
For simplicity and from now on, we set
\[
M = (1+y)^2 + |x|^2.
\]
Under this convention, the Jacobian matrix of $B$ given in \eqref{eq:JacobianMatrixOfB} is simply
\[
DB(x,y) =
\begin{pmatrix}
\dfrac2M I_n -\dfrac{2x}{M} \otimes \dfrac{2x}M& -\dfrac{4x(1+y)}{M^2}\\
\dfrac{4x^t(1+y)}{M^2} &-\dfrac2M + \dfrac{4(1+y)^2}{M^2}
\end{pmatrix}.
\]
Using this matrix, we can easily calculate $\nabla (F \circ B)$. Indeed, for $i=1,2,...,n$, we have that
\begin{equation}\label{eq:Partial(i)F}
\partial_i(F \circ B) = \sum_{j=1}^n (\partial_j F)(B) \Big(\frac{2 \delta_i^j}M - \frac{4x_ix_j}{M^2}\Big) + (\partial_{n+1} F)(B) \frac{4x_i(1+y)}{M^2}
\end{equation}
and that
\begin{equation}\label{eq:Partial(n+1)F}
\partial_{n+1}( F \circ B ) =-\sum_{j=1}^n (\partial_j F)(B) \frac{4x_j(1+y)}{M^2} + (\partial_{n+1}F)(B)\Big( -\frac 2M + \frac{4(1+y)^2}{M^2} \Big).
\end{equation}
Using our preceding calculation, it is easy to calculate $\Delta ( F \circ B )$. Indeed, for each $1 \leqslant i \leqslant n$, from \eqref{eq:Partial(i)F} we have
\begin{align*}
\partial_i^2 ( F \circ B ) =& \sum_{l=1}^n\sum_{j=1}^n (\partial_{jl}F)(B)\Big(\frac{2 \delta_i^j}M - \frac{4x_ix_j}{M^2}\Big)\Big(\frac{2 \delta_i^l}M - \frac{4x_ix_l}{M^2}\Big)\\
& +\sum_{j=1}^n (\partial_j F)(B) \Big(\frac{-8x_i\delta_{i,j}}{M^2} -\frac{4x_j}{M^2} + \frac{16x_i^2x_j}{M^3}\Big)\\
& + \sum_{j=1}^n (\partial_{j,n+1} F)(B)\Big(\frac{2 \delta_i^j}M - \frac{4x_ix_j}{M^2}\Big)\frac{4x_i(1+y)}{M^2} \\
&+ 16 (\partial_{n+1}^2 F)(B) \frac{x_i^2(1+y)^2}{M^4} \\
&+4(\partial_{n+1} F)(B)\Big(\frac{1+y}{M^2}-\frac{4x_i^2 (1+y)}{M^3}\Big) .
\end{align*}
Hence
\begin{align*}
\sum_{i=1}^n \partial_i^2 ( F \circ B ) =& \sum_{j,l=1}^n (\partial_{jl}F)(B)\left(\frac{4\delta_{l,j}}{M^2}- \frac{16 (1+y)^2x_jx_l}{M^4}\right)\\
& +\sum_{j=1}^n (\partial_j F)(B) \left(\frac{-4(n+2)x_j}{M^2} + \frac{16|x|^2x_j}{M^3}\right)\\
& + 8\sum_{j=1}^n (\partial_{j,n+1} F)(B)\frac{(1+y)^2 -|x|^2}{M^4} x_j(1+y)\\
& + 16(\partial_{n+1}^2 F)(B) \frac{|x|^2(1+y)^2}{M^4} \\
&+ 4(\partial_{n+1} F)(B)\Big(\frac{n}{M^2}-\frac{4 |x|^2}{M^3}\Big) (1+y) \\
= & I_1 + I_2+ I_3+ I_4+I_5.
\end{align*}
We also have from \eqref{eq:Partial(n+1)F} the following
\begin{align*}
\partial_{n+1}^2 ( F \circ B )= &\sum_{j,l=1}^n (\partial_{jl}F)(B) \frac{16x_jx_l(1+y)^2}{M^4} -\sum_{j=1}^n ( \partial_j F)(B) \Big(\frac{4x_j}{M^2} -\frac{16x_j(1+y)^2}{M^3}\Big)\\
& -8\sum_{j=1}^n (\partial_{j,n+1}F)(B)\frac{(1+y)^2-|x|^2}{M^4}x_j(1+y)\\
& + 4(\partial_{n+1}^2 F)(B) \frac{((1+y)^2 -|x|^2)^2}{M^4}\\
&+4(\partial_{n+1}F)(B)\frac{(3|x|^2-(1+y)^2)(1+y)}{M^3} \\
= & II_1 + II_2+ II_3+ II_4+II_5.
\end{align*}
Observe that
\[
\left\{
\begin{split}
I_1 + II_1 =& \sum_{i=1}^n (\partial_{ii} F) (B),\\
I_2 + II_2 = &-\frac {4(n-1)}{M^2} \sum_{j=1}^n (\partial_j F)(B) , \\
I_3 + II_3 = & 0, \\
I_4 + II_4 = & (\partial_{n+1}^2 F) (B), \\
I_5 + II_5 = & \frac {4(n-1)}{M^2} (1+y) (\partial_{n+1} F) (B) .
\end{split}
\right.
\]
From this we obtain the desired identities.
\end{proof}

\begin{corollary}\label{cor:Identity}
There holds
\[
\langle \nabla (F\circ B) , (x, 1+y) \rangle= \Phi \langle \nabla F(B) , (-x, 1+y) \rangle
\]
and
\[\begin{split}
 \Delta(F \circ B) (X)= (\Delta F) (B (X)) \Phi ^2+ (n-1) \langle\nabla(F \circ B)(X), (x,1+y)\rangle \Phi ,
\end{split}\]
where $X = (x,y) \in \R^{n+1}$.
\end{corollary}

\begin{proof}
In view of \eqref{eq:Partial(i)F} and \eqref{eq:Partial(n+1)F} we obtain
\begin{align*}
\langle\nabla (F \circ B) & , (x,1+y)\rangle\\
=& \sum_{i=1}^n \sum_{j=1}^n (\partial_j F)(B) \Big(\frac{2 \delta_{i,j}}M -\frac{4 x_i x_j }{M^2}\Big)x_i + \frac{4(1+y)}{M^2} (\partial_{n+1}F)(B) |x|^2\\
& -\frac{4(1+y)^2}{M^2} \sum_{j=1}^n (\partial_j F)(B) x_j + 2 (1+y)( \partial_{n+1}F)(B)\frac{(1+y)^2 - |x|^2}{M^2}\\
=& -\frac2M\sum_{j=1}^n( \partial_j F)(B) x_j +\frac 2M (\partial_{n+1} F)(B) (1+y) .
\end{align*}
From this and Lemma \ref{lem:Identity} we have the desired result.
\end{proof}

The main purpose of this section is to prove Proposition \ref{thmIDENTITY} below. Let us first consider the case $k=1$ in Proposition \ref{thmIDENTITY}. Let $F: \mathbb B^{n+1} \to \R$ be arbitrary, we define $f_1 : \R^n \to \R$, in terms of $F$, via the following rule
\begin{equation*}\label{eq:f_1FB}
f_1(X) = (F \circ B)(X) \Phi (X)^{\frac{n-1}2}.
\end{equation*}
The next result provides a relation between $\Delta f_1$ and $(\Delta F) (B)$.

\begin{proposition}\label{prop:RelationF1}
There holds
\begin{equation}\label{eq:relation1}
\Delta f_1 = (\Delta F)(B ) \Phi ^{\frac{n+3}2}.
\end{equation}
\end{proposition}

\begin{proof}
This fact is a consequence of the preceding corollary. In fact, it follows from Corollary \ref{cor:Identity} that
\[\begin{split}
\Delta f_1 = & \Delta \big( F(B ) \big) \Phi ^{\frac{n-1}2}+2 \nabla \big( F(B ) \big) \nabla \Phi ^{\frac{n-1}2} + F(B(X)) \Delta \Phi ^{\frac{n-1}2} \\
= & (\Delta F)(B )\Phi ^{\frac{n+3}2} + (n-1) \langle \nabla (F (B ) ), (x, 1+y)\rangle \Phi ^{\frac{n+1}2}+2 \nabla \big( F(B ) \big) \nabla \Phi ^{\frac{n-1}2} \\
= & (\Delta F)(B)\Phi ^{\frac{n+3}2}
\end{split}\]
as claimed.
\end{proof}

To generalize \eqref{eq:relation1} for higher order derivatives, we first mimic the proof of Lemma 2.2 in \cite{Hang} to obtain another useful identity.

\begin{lemma}\label{lemConformalCovariantIdentity}
There holds
\[
\Delta (\Phi ^{-m-1} \Delta^m u ) = \Phi ^{-m} \Delta^{m+1} ( \Phi ^{-1} u )
\]
for any non-negative integer $m$.
\end{lemma}

\begin{proof}
Given any non-negative number $a$, it is easy to verify that
\[
\left\{
\begin{split}
\nabla \Big( \frac{|x|^2 + (1+y)^2}2 \Big)^a & =a\Big( \frac{|x|^2 + (1+y)^2}2 \Big)^{a-1} (x, 1+y),\\
 \Delta \Big( \frac{|x|^2 + (1+y)^2}2 \Big)^a &= a (2a + n -1) \Big( \frac{|x|^2 + (1+y)^2}2 \Big)^{a-1}.
\end{split}
\right.
\]
Clearly, the case $m=0$ is trivial. To consider the case $m>0$, we first observe
\[
\Delta \Big( \frac{|x|^2 + (1+y)^2}2 u \Big) = (n+1) u + 2 (x, 1+y) \nabla u + \frac{|x|^2 + (1+y)^2}2 \Delta u.
\]
By induction on $k$ we get
\[
\Delta^k \Big( \frac{|x|^2 + (1+y)^2}2 u \Big) = a_k \Delta^{k-1} u + b_k (x, 1+y) \nabla \Delta^{k-1} u + \frac{|x|^2 + (1+y)^2}2 \Delta^k u
\]
with $a_k=k(2k+n-1)$ and $b_k = 2k$. Indeed,
\[
\begin{split}
\Delta^{k+1} \Big( &\frac{|x|^2 + (1+y)^2}2 u \Big) \\
=& \Delta \Big( a_k \Delta^{k-1} u + b_k (x, 1+y) \nabla \Delta^{k-1} u + \frac{|x|^2 + (1+y)^2}2 \Delta^k u \Big) \\
=& a_k \Delta^{k} u + 2 b_k\Delta^{k} u + b_k (x, 1+y) \nabla \Delta^{k} u \\
& + (n+1) \Delta^k u + 2 (x, 1+y) \nabla \Delta^k u + \frac{|x|^2 + (1+y)^2}2 \Delta^{k+1} u \\
= & \big( a_k + 2b_k + n + 1 \big) \Delta^{k} u + (b_k+2) (x, 1+y) \nabla \Delta^k u + \frac{|x|^2 + (1+y)^2}2 \Delta^{k+1} u\\
= & a_{k+1} \Delta^{k} u + b_{k+1} (x, 1+y) \nabla \Delta^k u + \frac{|x|^2 + (1+y)^2}2 \Delta^{k+1} u.
\end{split}
\]
Using this formula, we deduce that
\[
\begin{split}
\Delta \Big( \Big( \frac{|x|^2 + (1+y)^2}2 \Big)^{m+1} \Delta^m u \Big) =& (m+1) (2m + n + 1) \Big( \frac{|x|^2 + (1+y)^2}2 \Big)^{m} \Delta^m u \\
&+2m\Big( \frac{|x|^2 + (1+y)^2}2 \Big)^{m-1} (x, 1+y) \nabla \Delta^k u \\
& + \Big( \frac{|x|^2 + (1+y)^2}2 \Big)^{m+1} \Delta^{m+1} u\\
= &\Big( \frac{|x|^2 + (1+y)^2}2 \Big)^m \Delta^{m+1} \Big( \frac{|x|^2 + (1+y)^2}2 u \Big).
\end{split}
\]
Thus, for any non-negative integer $m$, we have just shown that
\[
\Delta (\Phi ^{-m-1} \Delta^m u ) = \Phi ^{-m} \Delta^{m+1} ( \Phi ^{-1} u )
\]
as claimed.
\end{proof}

We are now in position to generalize Proposition \ref{prop:RelationF1}. We prove the following theorem.

\begin{proposition}\label{thmIDENTITY}
For any integer $1\leqslant k < n/2$, define 
\[
f_k = F \circ B \, \Phi ^{\frac{n+1-2k}2}.
\]
Then we have the following identity
\begin{equation}\label{eq:relationgeneral}
\Delta^k f_k =  (\Delta^k F)\circ B\, \Phi^{\frac{n+1+2k}2} .
\end{equation}
\end{proposition}

\begin{proof}
We prove \eqref{eq:relationgeneral} by induction. Thanks to \eqref{eq:relation1}, the statement holds for $k=1$. Assume by induction that \eqref{eq:relationgeneral} holds up to some $k<\lfloor n/2 \rfloor-1$, that is
\begin{equation}\label{eq:RelationInductionK}
(\Delta^k F) \circ B  = \Phi ^{-\frac{n+1+2k}2} \Delta^k \big( F \circ B \, \Phi ^{\frac{n+1-2k}2} \big). 
\end{equation}
To compute $\Delta^{k+1} f_{k+1}$, it suffices to compute $(\Delta^{k+1} F) \circ B$. Indeed, by Corollary \ref{cor:Identity}, we have
\begin{align*}
(\Delta^{k+1} F)\circ B = & (\Delta (\Delta^k F))\circ B \\
= &\Delta ((\Delta^k F)\circ B) \Phi^{-2} -(n-1) \langle \nabla ((\Delta^k F)\circ B),(x,1+y)\rangle \Phi^{-1} \\
=& I \Phi^{-2} - (n-1) II \Phi^{-1}.
\end{align*}
To simplify notation, we denote 
\[
u = F\circ B\, \Phi^{\frac{n+1-2k}2}.
\] 
Then the induction assumption \eqref{eq:RelationInductionK} becomes
\begin{equation}\label{eq:RelationInductionKNewForm}
(\Delta^k F)\circ B = \Phi^{-\frac{n+1+2k}2} \Delta^k u.
\end{equation}
We now compute $I$. Clearly, by \eqref{eq:RelationInductionKNewForm}, we have
\begin{align*}
I = & \Delta(\Phi^{-\frac{n-1}2} \Phi^{-k-1} \Delta^k u) \\
=& \Delta (\Phi^{-\frac{n-1}2})\Phi^{-k-1} \Delta^k u + \Phi^{-\frac{n-1}2} \Delta(\Phi^{-k-1} \Delta^k u) \\
& + 2 \langle \nabla \Phi^{-\frac{n-1}2}, \nabla(\Phi^{-k-1} \Delta^k u)\rangle \\
=& \Delta (\Phi^{-\frac{n-1}2})\Phi^{-k-1} \Delta^k u + \Phi^{-\frac{n-1}2} \Delta(\Phi^{-k-1} \Delta^k u)  \\
& + 2 \langle \nabla \Phi^{-\frac{n-1}2}, \nabla \Phi^{-k-1}\rangle \Delta^k u +2  \Phi^{-k-1} \langle  \nabla \Phi^{-\frac{n-1}2}, \nabla \Delta^k u\rangle.
\end{align*}
Using Lemma \ref{lem:NablaDelta}, we can easily check that
\[
\nabla \Phi^{-\frac{n-1}2} = \frac{n-1}2 \Phi^{-\frac{n-3}2}(x,1+y),\quad 
\nabla \Phi^{-k-1} = (k+1) \Phi^{-k}(x,1+y),
\]
and
\[
 \Delta (\Phi^{-\frac{n-1}2}) = (n -1)^2 \Phi^{-\frac{n-3}2}.
\]
Therefore, these identities and Lemma \ref{lemConformalCovariantIdentity} yield
\begin{align*}
I  = & (n-1)^2 \Phi^{-\frac{n-1+2k}2} \Delta^k u + \Phi^{-\frac{n-1+2k}2} \Delta^{k+1}(\Phi^{-1} u)\\
& + 2(n-1)(k+1) \Phi^{-\frac{n-1+2k}2} \Delta^k u + (n-1) \Phi^{-\frac{n-1+2k}2} \langle (x, 1+y), \nabla \Delta^k u \rangle \\
= & \Phi^{-\frac{n-1+2k}2} \Delta^{k+1}(\Phi^{-1} u) + (n-1)(n+1+2k)\Phi^{-\frac{n-1+2k}2} \Delta^k u\\
&  + (n-1) \Phi^{-\frac{n-1+2k}2} \langle \nabla \Delta^k u, (x, 1+y)\rangle.
\end{align*}
On the other hand, by \eqref{eq:RelationInductionKNewForm} and Lemma \ref{lem:NablaDelta}, we also have
\begin{align*}
II = & \langle \nabla ((\Delta^k F)\circ B),(x,1+y)\rangle\\
 =& \langle \nabla(\Phi^{-\frac{n+1+2k}2} \Delta^k u), (x,1+y)\rangle \\
= &(n+1+2k) \Phi^{-\frac{n+1+2k}2} \Delta^k u + \Phi^{-\frac{n+1+2k}2} \langle \nabla \Delta^k u, (x,1+y)\rangle.
\end{align*}
Consequently, we get
\begin{align*}
(\Delta^{k+1} F)\circ B =&  I \Phi^{-2} - (n-1) II \Phi^{-1} \\
=& \Phi^{-\frac{n+3+2k}2} \Delta^{k+1}(\Phi^{-1} u) \\
= & \Phi^{-\frac{n+1+2(k+1)}2} \Delta^{k+1} \big( F\circ B \Phi^{\frac{n+1 -2(k+1)}2} u \big)
\end{align*}
as wanted, which, by induction, completes the proof.
\end{proof}


\section{Sobolev trace inequality of order two}
\label{sec-Order2}

The main purpose of this section is to provide a new proof of the Sobolev trace inequality of order two on spheres. As we shall soon see later, our argument depends on the sharp Sobolev trace inequality \eqref{eq:SobolevTraceSpaceOrder2} on $\R^{n+1}_+$, that is
\[
\frac{\Gamma (\frac{n+1}2)}{\Gamma (\frac{n-1}2)} \omega_n^{1/n} \Big( \int_{\R^n} |U(x,0)|^{\frac{2n}{n-1}} dx \Big)^{\frac{n-1}n} \leqslant \int_{\R^{n+1}_+} |\nabla U|^2 dxdy.
\]

\subsection{Sharp Sobolev trace inequality of order two on $\bB^{n+1}$: Proof of Theorem \ref{thmTraceOrder2}}

This subsection is devoted to a proof of Theorem \ref{thmTraceOrder2}. The proof consisting of four steps is divided into two parts. In the first three steps, we prove \eqref{eq:TraceOrder2} for any harmonic extension. Then in the last part, we prove \eqref{eq:TraceOrder2} for any smooth extension. 

\medskip\noindent\textbf{Step 1}. Given $f \in C^\infty (\mathbb S^n)$ and suppose that $u$ is a harmonic extension of $f$ to $\bB^{n+1}$. Then, in terms of $u$, we define the function $U$ on $\R^{n+1}_+$ by
\[
U (x,y) = (u \circ B )(x,y) \Big(\frac2{(1+y)^2 +|x|^2}\Big)^{\frac{n-1}2}.
\]
Thanks to \eqref{eq:relation1} and the harmonicity of $u$, we have the relation
\[
(\Delta U)(x,y) = \Delta (u \circ B )(x,y) \Big(\frac2{(1+y)^2 +|x|^2}\Big)^{\frac{n+1}2} =0
\]
on $\R^{n+1}_+$. Hence $U$ is a harmonic extension of $f$ to the upper halfspace $\R^{n+1}_+$ . Thus, we can apply the Sobolev trace inequality \eqref{eq:SobolevTraceSpaceOrder2} on $\R^{n+1}_+$ for $U$. The idea is to transform this trace inequality on $\R_+^{n+1}$ to an equivalent trace inequality on $\bB^{n+1}$. To this purpose, we have to express $\int_{\R^n} |U(x,0)|^{\frac{2n}{n-1}} dx$ and $\int_{\R^{n+1}_+} |\nabla U|^2 dxdy$ in terms of $u$ and this is the content of the next two steps.

\medskip\noindent\textbf{Step 2}. First we calculate $\int_{\R^n} |U(x,0)|^{\frac{2n}{n-1}} dx$. Clearly,
\begin{equation}\label{eq:2LHS}
|U(x,0)|^{\frac{2n}{n-1}} = |u(\mathcal S(x))|^{\frac{2n}{n-1}} J_{\mathcal S}(x) =|f(\mathcal S(x))|^{\frac{2n}{n-1}} J_{\mathcal S}(x).
\end{equation}
From this we deduce that
\[
\int_{\R^n} |U(x,0)|^{\frac{2n}{n-1}} dx = \int_{\bS^n} |f|^{\frac{2n}{n-1}} d\omega .
\]

\medskip\noindent\textbf{Step 3}. Now we calculate $\int_{\R^{n+1}_+} |\nabla U|^2 dxdy$. Without writing the variable $X$, it follows from Corollary \ref{cor:Identity} that
\[
\nabla U = \Big[DB^t \cdot (\nabla u)(B ) - \frac {n-1}2 u(B ) (x,1+y) \Phi \Big] \Phi ^{\frac{n-1}2},
\]
where $DB$ is the Jacobian matrix of $B$ given in \eqref{eq:JacobianMatrixOfB}, that is
\[
DB(x,y) =
\begin{pmatrix}
\Phi I_n - \Phi^2 x \otimes x & - \Phi^2 x(1+y) \\
\Phi^2 x^t(1+y) &- \Phi + \Phi^2 (1+y)^2
\end{pmatrix}.
\]
Thanks to \eqref{eq:DBDBt}, we know that
\[
\langle DB^t \cdot (\nabla u)(B ), DB^t \cdot (\nabla u)(B ) \rangle = (\nabla u)(B )^t DB \cdot DB^t \cdot (\nabla u)(B ) = |\nabla u(B )|^2.
\] 
Furthermore, in view of Corollary \ref{cor:Identity}, we easily get
\[\begin{aligned}
\langle DB^t \cdot (\nabla u)(B ), (x,1+y) \rangle =& \langle \nabla (u \circ B ) , (x,1+y) \rangle \\
= &- \Phi \langle \nabla u(B ), (x,-1-y) \rangle \\
= &- \langle \nabla u(B ), B - e_{n+1} \rangle
\end{aligned}\]
where $e_{n+1} =(0,\ldots,0,1)$. Here we have just used the elementary fact
\begin{equation}\label{eqPhi-B-e}
\Phi (x,-1-y) = B (x, y) - e_{n+1} .
\end{equation}
From these facts, we arrive at
\[
|\nabla U |^2 =
\left(
\begin{split}
& |\nabla u(B )|^2 \Phi ^2 + (n-1) u(B) \langle \nabla u (B ) , B -e_{n+1} \rangle \Phi \\
&+\Big(\frac{n-1}2 \Big)^2 u(B )^2 | (x,1+y)|^2 \Phi ^2 
\end{split}
\right)\Phi ^{n-1}.
\]
Notice by \eqref{eqPhi-B-e} that 
\[
|B(x,y) -e_{n+1}|^2 = 2 \Phi .
\] 
Therefore, we can rewrite $|\nabla U |^2$ as follows
\begin{equation}\label{eqOrder2|NablaU|^2}
|\nabla U |^2 =
\Big(
|\nabla u(B )|^2 + (n-1) \frac{\langle \nabla (u (B )^2), B -e_{n+1} \rangle}{|B -e_{n+1}|^2} +\frac{(n-1)^2 u(B )^2}{|B -e_{n+1}|^2}
\Big)\Phi ^{n+1}.
\end{equation}
Keep in mind that the Jacobian of $B$ is $\Phi ^{n+1}$. Hence from \eqref{eqOrder2|NablaU|^2} a simple change of variables leads us to
\begin{align*}
\int_{\R^{n+1}_+} |\nabla U|^2 dx dy =& \int_{\bB^{n+1}} |\nabla u|^2 dz + (n-1) \int_{\bB^{n+1}} \frac{\langle \nabla (u^2) , z -e_{n+1} \rangle }{|z -e_{n+1}|^2} dz\\
&+ (n-1)^2 \int_{\bB^{n+1}} \frac{u^2}{|z -e_{n+1}|^2} dz.
\end{align*}
Keep in mind that
\[
\nabla \cdot \frac{ z -e_{n+1} }{|z -e_{n+1}|^2} = \frac{n-1}{|z -e_{n+1}|^2}
\]
in $\R^{n+1}$. Hence, integrating by parts yields
\[
 \int_{\bB^{n+1}} \frac{\langle \nabla (u^2), z -e_{n+1} \rangle}{|z -e_{n+1}|^2} dz = - (n-1)\int_{\bB^{n+1}} \frac{u^2}{|z -e_{n+1}|^2} dz + \int_{\bS^n} u^2 \frac{ \langle \omega -e_{n+1}, \omega \rangle}{|\omega -e_{n+1}|^2} d\omega ,
\]
where $\omega = x/|x|$. From this we obtain
\[
\int_{\R^{n+1}_+} |\nabla U|^2 dx dy = \int_{\bB^{n+1}} |\nabla u|^2 dz + (n-1)\int_{\bS^n} u^2 \frac{ \langle \omega -e_{n+1}, \omega \rangle }{|\omega -e_{n+1}|^2} d\omega .
\]
However, it is easy to see that
\[
 \langle \omega -e_{n+1}, \omega \rangle= 1 - \langle \omega, e_{n+1} \rangle = \frac12 |\omega -e_{n+1}|^2.
\] 
Thus, we have just shown that
\begin{equation}\label{eq:2RHS}
\int_{\R^{n+1}_+} |\nabla U|^2 dx dy =\int_{\bB^{n+1}} |\nabla u|^2 dz + \frac{n-1}2 \int_{\bS^n} |f|^2 d\omega .
\end{equation}
Combining \eqref{eq:2LHS}, \eqref{eq:2RHS}, and the sharp Sobolev trace inequality \eqref{eq:SobolevTraceSpaceOrder2} for $U$ gives
\[
\begin{aligned}
\frac{\Gamma(\frac{n+1}2)}{\Gamma(\frac{n-1}2)}  \omega_n^{1/n}\lt(\int_{\bS^n} |f|^{\frac{2n}{n-1}} d\omega \rt)^{\frac{n-1}n} \leqslant \int_{\bB^{n+1}} |\nabla u|^2 dz + \frac{n-1}2 \int_{\bS^n} |f|^2 d\omega
\end{aligned}
\]
provided $u$ is a harmonic extension of $f$ to $\mathbb B^n$. This completes Step 3.

\medskip\noindent\textbf{Step 4}. In this step, we prove \eqref{eq:TraceOrder2}. Indeed, given $f\in C^\infty(\bS^n)$ it is well-known that the minimizing problem
\begin{equation}\label{eq:varproblem2}
\inf_{w } \Big \{ \int_{\bB^{n+1}} |\nabla w|^2 dx : w \big|_{\mathbb S^n} = f \Big\}
\end{equation}
is attained by some harmonic extension $u$ of $f$ in $\bB^{n+1}$. Therefore, we can repeat from Step 1 to Step 3 to get the following estimate
\begin{equation}\label{eq:TraceOrder2p}
\begin{aligned}
\frac{\Gamma(\frac{n+1}2)}{\Gamma(\frac{n-1}2)}  \omega_n^{1/n}\lt(\int_{\bS^n} |f|^{\frac{2n}{n-1}} d\omega \rt)^{\frac{n-1}n} \leqslant \int_{\bB^{n+1}} |\nabla u|^2 dz + \frac{n-1}2 \int_{\bS^n} |f|^2 d\omega.
\end{aligned}
\end{equation}
Since $u$ is a minimizer of \eqref{eq:varproblem2}, any smooth extension $v$ of $f$ to $\mathbb B^n$ enjoys the estimate
\[
\int_{\bB^{n+1}} |\nabla u|^2 dx \leqslant \int_{\bB^{n+1}} |\nabla v|^2 dx.
\]
The desired inequality follows from the preceding estimate and \eqref{eq:TraceOrder2p}. The assertion for which equality in \eqref{eq:TraceOrder2} is attained is well-known.


\subsection{A weighted Beckner inequality of order two on $\bB^{n+1}$}
\label{subsec-2Beckner}

This subsection is devoted to a weighted Beckner inequality; see Theorem \ref{thm2BecknerTypeWeight} below. Let $f \in C^\infty(\bS^n)$ and $b \in (0,1)$. Let also $u$ be a smooth extension of $f$ in $\bB^{n+1}$ such that
\begin{equation}\label{eq:2BecknerRequire}
\Delta u(z) -\frac{2b }{1-|z|^2} \langle \nabla u(z), z \rangle= 0
\end{equation}
for $z \in \bB^{n+1}$. Suppose that the function $f$ has the following spherical harmonic decomposition
\[
f (\omega)= \sum_{k=0}^\infty Y_k(\omega ),
\]
where $\omega = x/|x|$ and $Y_k$ is a spherical harmonic of order $k \geqslant 0$. From this we decompose $u$ to get 
\[
u (z) = \sum_{k=0}^\infty f_k(r) Y_k(\omega ).
\] 
Recall that $\Delta = \Delta_r +(1/r^2) \widetilde \Delta$ with $\Delta_r = \partial^2_r + (n/r) \partial_ r$ in $\R^{n+1}$ and $\widetilde \Delta Y_k = -c_k Y_k$ with $c_k = k(n+k-1)$. For simplicity, let us denote by $L_k$ the following operator
\[
L_k : f(r) \mapsto f''(r) + \frac nr f'(r) - \frac {c_k}{r^2}.
\]
Now, on one hand,  we know that
\[\begin{aligned}
\Delta \big(f_k(r) Y_k(\omega ) \big) =&  \Delta_r\big(f_k(r) Y_k(\omega ) \big) +(1/r^2)\widetilde \Delta \big(f_k(r) Y_k(\omega ) \big) \\
=& \Big( L_k f_k (r) - \frac {2br}{1-r^2} f'_k \Big) Y_k (\omega).
\end{aligned}\]
On the other hand, for each $1 \leqslant i \leqslant n+1$, there holds
\[
\partial_i \big(f_k(r) Y_k(\omega ) \big) = f_k' (r) \partial_i (r) Y_k (\omega) + f_k (r) (\partial^j Y_k) (\omega)\partial_i (\frac {x_j}{r})
\]
leading to
\[
\langle f_k(r) Y_k(\omega ) , z \rangle = r f_k' (r) Y_k(\omega) + f_k(r) (\partial^j Y_k) (\omega) \Big( \frac{\delta_j^i}r x_i - \frac{x_i^2 x_j}{r^3} \Big)= r f_k' (r) Y_k(\omega).
\]
(Note that the Einstein convention was used in the previous equation.) Hence, it follows from \eqref{eq:2BecknerRequire} that the coefficients $f_k$ satisfy
\begin{equation}\label{eq2f_k}
f_k''(r) + \Big(\frac{n}r -\frac{2b r}{1-r^2}\Big) f_k'(r) - \frac{c_k}{r^2} f_k(r) = 0
\end{equation}
for any $r\in [0,1)$ and definitely $f_k(1) =1$. Recall that in the preceding decomposition, we know that $c_k = k(n+k-1)$ for all $k\geqslant 0$. Hence $f_0 \equiv 1$. Given $b \in (0,1)$, our aim is to understand
\[
\lim_{r\nearrow 1} \Big(\frac{1-r^2}2\Big)^b f_k'(r)
\]
for $k\geqslant 1$. In the following result, we describe this limit.

\begin{proposition}\label{limits}
For $k\geqslant 1$ and $b \in (-1,1)$, let $\alpha_k$ and $\beta_k$ be solutions of
\[
\alpha_k + \beta_k = \frac{n+2k-1+2b}2,\quad \alpha_k \beta_k = \frac{bk}2.
\]
Define
\begin{equation}\label{eq:Abk}
A(b,k) = 2^{-b} \frac{\Gamma(b+1)}{\Gamma(1-b)} \frac{\Gamma(\beta_k+1-b)\Gamma(\alpha_k+1-b)}{\Gamma(\alpha_k+1) \Gamma(\beta_k+1)} k
\end{equation}
if $k\geqslant 1$ and $A(b,0) =0$. Then we have 
\begin{equation*}\label{eq:limitsbien}
\lim_{r\nearrow 1} \Big(\frac{1-r^2}2\Big)^b f_k'(r) = A(b,k).
\end{equation*}
\end{proposition}

For clarity, we put the proof of Proposition \ref{limits} in Appendix \ref{apd-Limits}. Using integration by parts and the equation \eqref{eq:2BecknerRequire} satisfied by $u$, we obtain
\[
\begin{aligned}
\int_{\bB^{n+1}} |\nabla u|^2 &\Big(\frac{1-|z|^2}2\Big)^b dz \\
=& -\int_{\bB^{n+1}} u \partial^i\Big[\Big(\frac{1-|z|^2}2\Big)^b \partial_i u \Big] dz + \int_{\mathbb S^n} \Big[ u \Big(\frac{1-|z|^2}2\Big)^b \langle \nabla u, \omega \rangle \Big] \Big|_{|z| = 1} d\omega \\
=&\int_{\mathbb S^n} \Big[u \Big(\frac{1-r^2}2\Big)^b \sum_{ k \geqslant 1} f_k' (r) Y_k (\omega ) \Big] \Big|_{r= 1} d\omega .
\end{aligned}
\]
Hence, applying Proposition \ref{limits} gives
\begin{equation}\label{eq:Equalnice}
\int_{\bB^{n+1}} |\nabla u|^2 \Big(\frac{1-|z|^2}2\Big)^b dz = \sum_{k=1}^\infty A(b,k) \int_{\bS^n} |Y_k|^2  d\omega .
\end{equation}
In the sequel, we shall choose $b=1-s$ for $s \in (0,1)$. Now we define the function $F$ on $\R^n$ by
\begin{equation}\label{eq:2FfS}
F(x) = f(\mathcal S(x)) J_{\mathcal S}(x)^{\frac{n- s}{2n}}.
\end{equation}
Then we have
\[
\int_{\bS^n} |f|^{\frac{2n}{n- s}} d\omega = \int_{\R^n} |F|^{\frac{2n}{n- s}} dx
\]
and by Lemma 8 in \cite{GN}, we have the following interesting identity
\begin{equation}\label{eq:GN}
\|(-\Delta)^{s/2} F\|_{L^2(\R^n)}^2 = \sum_{k=0}^\infty\frac{\Gamma(k+n/2+s)}{\Gamma(k+n/2-s)} \int_{\bS^n} |Y_k|^2 ( \omega ) d\omega .
\end{equation}
We now use the fractional Sobolev inequality \eqref{eq:SobolevSpaceOrderS/2} applied to $F$ and \eqref{eq:Equalnice} to get
\begin{equation}\label{eq:fSobolev}
\begin{split}
\frac{\Gamma(\frac{n+s}2)}{\Gamma(\frac{n-s}2)} \omega_n^{s/n} & \lt(\int_{\bS^n} |f|^{\frac{2n}{n-s}} d\omega \rt)^{\frac{n-s}n}\\
\leqslant & \sum_{k=0}^\infty\frac{\Gamma(k+n/2+s/2)}{\Gamma(k+n/2-s/2)}  \int_{\bS^n} |Y_k|^2 ( \omega ) d\omega \\
=& \int_{\bB^{n+1}} |\nabla u|^2 \Big(\frac{1-|z|^2}2\Big)^{1-s} dz \\
&+ \sum_{k=0}^\infty \lt(\frac{\Gamma(k+n/2+s/2)}{\Gamma(k+n/2-s/2)}  -A(1-s,k)\rt) \int_{\bS^n} |Y_k|^2 ( \omega ) d\omega .
\end{split}
\end{equation}

Clearly, equality in \eqref{eq:fSobolev} occurs if, and only if, equality in the fractional Sobolev inequality \eqref{eq:SobolevSpaceOrderS/2} occurs. Our next result is as follows.

\begin{theorem}\label{thm2BecknerTypeWeight}
Let $n\geqslant 1$ and $0< s < \min\{2,n\}$. Let $f\in C^\infty(\bS^n)$ and $v$ be a smooth extension of $f$ to the unit ball $\bB^{n+1}$. Suppose that $f$ has a decomposition on spherical harmonics as $f =\sum_{k=0}^\infty Y_k(\omega )$. Then the following inequality holds
\begin{equation}\label{eq:weightedtraceOrder1}
\begin{split}
\frac{\Gamma(\frac{n+s}2)}{\Gamma(\frac{n-s}2)}\omega_n^{s/n} &\Big(\int_{\bS^n} |f|^{\frac{2n}{n-s}} d\omega \Big)^{\frac{n-s}n} \\
\leqslant &\int_{\bB^{n+1}} |\nabla v|^2 \Big(\frac{1-|z|^2}2\Big)^{1-s} dz\\
& + \sum_{k=0}^\infty \Big(\frac{\Gamma(k+n/2+s/2)}{\Gamma(k+n/2-s/2)}  -A(1-s,k)\Big) \int_{\bS^n} |Y_k|^2 d\omega ,
\end{split}
\end{equation}
where $A(1-s,k)$ is given in \eqref{eq:Abk}. Moreover, equality holds if, and only if, $v$ is a harmonic extension of a function of the form
\[
c|1-\langle z_0, \xi \rangle |^{-(n-s)/2},
\]
where $c>0$ is a constant, $\xi \in \mathbb S^n$, and $z_0$ is some fixed point in the interior of $\mathbb B^{n+1}$.
\end{theorem}

\begin{proof}
Let $f \in C^\infty(\bS^n)$, it is well known that the problem
\[
\inf\lt\{\int_{\bB^{n+1}} |\nabla u|^2 \Big(\frac{1-|z|^2}2\Big)^{1-s} dz\, :\, u\big{|}_{\bS^n} =f\rt\}
\]
is attained by some function $u$ such that
\[
\Delta u(z) - \frac{2(1-s) }{1-|z|^2} \langle \nabla u(z) , z \rangle =0
\]
in $\bB^{n+1}$ and $u \equiv f$ on $\bS^n$. The inequality \eqref{eq:weightedtraceOrder1} is then followed from \eqref{eq:fSobolev} and the fact that
\[
\int_{\bB^{n+1}} |\nabla u|^2 \Big(\frac{1-|z|^2}2\Big)^{1-s} dz \leqslant \int_{\bB^{n+1}} |\nabla v|^2 \Big(\frac{1-|z|^2}2\Big)^{1-s} dz
\]
since $u$ is a minimizer. Let us now determine the equality case in \eqref{eq:weightedtraceOrder1}. To this purpose, we need to find all functions $f$. Indeed, the equality case comes from the fact that the equality in the fractional Sobolev inequality \eqref{eq:SobolevSpaceOrderS/2} applied to $F$ occurs. In this scenario, there exist some positive constants $c, \mu$, some $x^0 \in \R^n$ such that
\[
F(x) = c\big( \mu + |x - x^0|^2 \big)^{-(n-s)/2}
\]
for all $x \in \R^n$. To find the corresponding $f$, we make use of \eqref{eq:2FfS} to get
\[
f(\xi) = J_{\mathcal S}(\mathcal S^{-1}(\xi))^{-(n-s)/(2n)} F(\mathcal S^{-1}(\xi))
\]
for $\xi \in \bS^{n+1}$. Therefore, the function $f$ is simply the lifting of the optimizer for the fractional Sobolev inequality in $\R^n$ via the stereographic projection $\mathcal S$. If we denote $\xi = (\xi_1,...,\xi_{n+1})$, then $\mathcal S^{-1}(\xi) = (\xi_1,...,\xi_n)/(1-\xi_{n+1})$, which then gives $J_{\mathcal S}(\mathcal S^{-1}(\xi))=(1-\xi_{n+1})^n$. From this we obtain
\[\begin{split}
f(\xi) = &c\Big[ (1-\xi_{n+1}) \Big( \mu + |x^0|^2 + 1 - 2 \Big\langle \big(x^0,-1\big), \frac \xi {1-\xi_{n+1}} \Big\rangle \Big) \Big]^{-(n-s)/2}\\
=& c\Big[   \mu + |x^0|^2 + 1 -  \big\langle \big(x^0,  \mu + |x^0|^2 - 1 \big),  \xi  \big\rangle  \Big]^{-(n-s)/2}
\end{split}\]
for $\xi \in \bS^{n+1}$. Observe that $|\mu + |x^0|^2 - 1|/(\mu + |x^0|^2 + 1)<1$. Thus, we have just shown that $f$ takes the form
\[
c \big( 1 - \langle z_0 , \xi \rangle \big)^{-(n-s)/2},
\]
for some constant $c>0$ and for some fixed point $z_0$ in the interior of $\mathbb B^{n+1}$. From this we have the conclusion.
\end{proof}

We note that the weighted Sobolev trace inequality \eqref{eq:weightedtraceOrder1} shares some similarities with the weighted trace inequality obtained by Case in \cite[Theorem 1.1]{Case2017} and the weighted trace inequality obtained by Jin and Xiong in \cite[Theorem 1.1]{JinXiong}. While the weighted trace inequality of Case involves the interior $L^2$-norm of the extension $v$, the weighted trace inequality of Jin and Xiong only requires a boundary $L^2$-norm of $f$. In our inequality \eqref{eq:weightedtraceOrder1}, a term involving $\int_{\bS^n} |Y_k|^2 d\omega$ appears, which, more or less, involves a boundary $L^2$-norm of $f$. It is also worth noticing that when $n \geqslant 2$, our restriction for $s$ is that $0<s<2$ and we are not sure if this is optimal compared with \cite[Remark 1.2]{JinXiong}.

As an application of Theorem \ref{thm2BecknerTypeWeight}, let us consider the case $n \geqslant 2$ and $s = 1$. Recall that
\[
\int_{\mathbb S^n} |f|^2 d\omega = \int_{\bS^n} |Y_k|^2 d\omega . 
\]
Hence, in the present scenario, Inequality \eqref{eq:weightedtraceOrder1} becomes the usual trace inequality \eqref{eq:TraceOrder2}. We also note that if $u$ is a harmonic extension of $f$, then $f_k$ solves $L_k f_k = 0$. From this and $f_k(1)=1$ we obtain $f_k(r)=r^k$. Consequently, we get
\[
\langle \nabla u , \omega \rangle = \sum_{k=0}^\infty k  Y_k (\omega).
\]
Thus,
\[
\int_{\mathbb B^{n+1}} |\nabla u|^2 dz = \int_{\mathbb S^n}  u \langle \nabla u, \omega \rangle dz = \sum_{k=0}^\infty k \int_{\mathbb S^n} |Y_k|^2 d\omega;
\]
see \cite[page 232]{beckner1993}. Therefore, we can mimic the argument in \eqref{eq:fSobolev} to get the following result.

\begin{theorem}\label{thm2BecknerTypeNoWeight}
Let $n\geqslant 1$ and $0< s <n$. Let $f\in C^\infty(\bS^n)$ and $v$ be a harmonic extension of $f$ to the unit ball $\bB^{n+1}$. Suppose that $f$ has a decomposition on spherical harmonics as $f =\sum_{k=0}^\infty Y_k(\omega )$. Then the following inequality holds
\begin{equation}\label{eq:fBecknerTypeNoWeight}
\begin{split}
\frac{\Gamma(\frac{n+s}2)}{\Gamma(\frac{n-s}2)} &\omega_n^{s/n}  \lt(\int_{\bS^n} |f|^{\frac{2n}{n-s}} d\omega \rt)^{\frac{n-s}n}\\
\leqslant & \int_{\bB^{n+1}} |\nabla v|^2 dz  + \sum_{k=0}^\infty \Big(\frac{\Gamma(k+n/2+s/2)}{\Gamma(k+n/2-s/2)}  - k \Big) \int_{\bS^n} |Y_k|^2 d\omega .
\end{split}
\end{equation}
Moreover, equality holds if, and only if, $v$ is a harmonic extension of a function of the form
\[
c|1-\langle z_0, \xi \rangle |^{-(n-s)/2},
\]
where $c>0$ is a constant, $\xi \in \mathbb S^n$, and $z_0$ is some fixed point in the interior of $\mathbb B^{n+1}$.
\end{theorem}

Apparently, Inequality \eqref{eq:fBecknerTypeNoWeight} includes \cite{beckner1993} as a special case because when $s = 1$ our inequality \eqref{eq:fBecknerTypeNoWeight} becomes Beckner's inequality. In the final part of this section, we treat the limiting case $n = 1$.


\subsection{A classical Ledebev--Milin inequality of order two on $\bB^2$}
\label{subsec-2LM}

Let us now consider the limiting case of Theorem \ref{thm2BecknerTypeNoWeight}, namely, $n=1$ and $0 < s < 1$. Suppose that $f \in C^\infty(\bS^1)$ with $\int_{\bS^1} f d\omega =0$ and let $v$ be a smooth extension of $f$ in $\bB^2$. As before, we decompose $f$ in terms of spherical harmonics to get
\[
f = \sum_{k=1}^\infty Y_k(\omega )
\]
Clearly, the function $1 + ((1-s)/2) v$ is also a smooth extension of $1 + ((1-s)/2) f$ in $\mathbb B^2$ and $\omega_1 = 2\pi$. Therefore, we can apply Theorem \ref{thm2BecknerTypeWeight} to get
\begin{align*}
\frac{\Gamma(\frac{1+s}2)}{\Gamma(\frac{1-s}2)}& (2\pi)^{s} \Big(\int_{\bS^1} \Big(1 + \frac{1-s}2 f\Big)^{\frac{2}{1-s}} d\omega \Big)^{1-s} \\
&\leqslant \frac{(1-s)^2}4 \int_{\bB^2} |\nabla v|^2 \Big(\frac{1-|z|^2}2\Big)^{1-s} dz + \frac{\Gamma(\frac{1+s}2)}{\Gamma(\frac{1-s}2)} 2\pi \\
&\quad + \frac{(1-s)^2}4 \sum_{k=1}^\infty \Big(\frac{\Gamma(k+1/2+s/2)}{\Gamma(k+ 1/2-s/2)} -A(1-s,k)\Big)\int_{\bS^1}|Y_k|^2 d\omega .
\end{align*}
Dividing both sides by $(1-s)^2$ and making use of $\Gamma(\frac{3-s}2) = (1/2)(1-s)\Gamma(\frac{1-s}2)$ to get
\begin{align*}
\frac{\Gamma(\frac{1+s}2)}{\Gamma(\frac{3-s}2)} \frac 1{1-s} &\Big[ \Big(\frac1{2\pi}\int_{\bS^1} \Big(1 + \frac{1-s}2 f\Big)^{\frac{2}{1-s}} d\omega \Big)^{1-s}-1 \Big] \\
\leqslant &\frac1{4\pi}\int_{\bB^2} |\nabla v|^2 \Big(\frac{1-|z|^2}2\Big)^{1-s} dz \\
& + \frac{1}{4\pi}\sum_{k=1}^\infty \Big( \frac{\Gamma(k+1/2+s/2)}{\Gamma(k+ 1/2-s/2)} -A(1-s,k)\Big)\int_{\bS^1}|Y_k|^2 d\omega .
\end{align*}
Note that 
\[
\lim_{s\to 1/2} \frac{\Gamma(k+1/2+s/2)}{\Gamma(k+ 1/2-s/2)} -A(1-s,k) =0,
\]
for any $k \geqslant 1$. Hence ``formally'' letting $s\nearrow 1$ and applying the l'H\^opital rule, we obtain
\[
\log \Big(\frac1{2\pi} \int_{\bS^1} e^f d\omega \Big) \leqslant \frac1{4\pi} \int_{\bB^2} |\nabla v|^2 dx.
\]
for any function $f$ with $\int_{\bS^1} f d\omega =0$. For general function $f$, we apply the previous inequality for $f-(1/(2\pi)) \int_{\bS^1} f d\omega$ to get the classical Lebedev--Milin inequality; see \cite{LM}, see also \cite[Inequality (4')]{OPS}.

\begin{theorem}[Lebedev--Milin inequality of order two]\label{thmSMOrder2}
Let $f\in C^\infty(\bS^1)$ and suppose that $v$ is a smooth extension of $f$ to the unit ball $\bB^2$. Then we have the following sharp trace inequality
\begin{equation}\label{eqSMorder2}
\log \Big(\frac1{2\pi} \int_{\bS^1} e^f d\omega \Big) \leqslant \frac1{4\pi} \int_{\bB^2} |\nabla v|^2 dx + \frac1{2\pi} \int_{\bS^1} f d\omega .
\end{equation}
Moreover, equality in \eqref{eqSMorder2} holds if, and only if, $v$ is a harmonic extension of a function of the form
\[
c - \log |1 -  \langle z_0, \xi \rangle|,
\]
where $c>0$ is a constant, $\xi \in \bS^1$, and $z_0$ is some fixed point in the interior of $\mathbb B^2$
\end{theorem}

We note that we can also apply Theorem \ref{thm2BecknerTypeNoWeight} to obtain \eqref{eqSMorder2} whose proof is left for interested reader.


\section{Sobolev trace inequality of order four: Proof of Theorem \ref{thmTraceOrder4}}
\label{sec-Order4}

The main purpose of this section is to provide a new proof of the Sobolev trace inequality of four two on $\mathbb S^n$ based on Escobar's approach. As in the preceding section, the key ingredient in this approach is the sharp Sobolev trace inequality
\[
2\frac{\Gamma(\frac{n+3}2)}{\Gamma(\frac{n-3}2)} \omega_n^{3/n} \Big(\int_{\R^n} |U(x,0)|^{\frac{2n}{n-3}} dx\Big)^{\frac{n-3}n} \leqslant \int_{\R^{n+1}_+} |\Delta U(x,y)|^2 dx dy.
\]
for any function $U$ satisfying the boundary condition $\partial_y U(x,y)\big{|}_{y=0} = 0$; see \eqref{eq:SobolevTraceSpaceOrder4}.

\subsection{Sharp Sobolev trace inequality of order four on $\bB^{n+1}$: Proof of Theorem \ref{thmTraceOrder4}}

This subsection is devoted to a proof of Theorem \ref{thmTraceOrder4}. For clarity, we divide the proof into several steps.

\medskip\noindent\textbf{Step 1}. To proceed the proof, we let $f\in C^\infty(\bS^n)$ and consider $u$ a biharmonic extension of $f$ to $\mathbb B^n$ satisfying $u \in \mathscr V_f$ with
\begin{equation}\label{eq:varproblem4set}
\mathscr V_f = \Big \{ w : \, w \big{|}_{\bS^n} = f,\, \frac{\partial}{\partial \nu} w\big{|}_{\bS^n} = -\frac{n-3}2 f \Big\}.
\end{equation}
Let $U$ be function defined on $\R^{n+1}_+$ by
\[
U = u(B) \Phi^{\frac{n-3}2}.
\]
As always, $u(B)$ is being understood as $u \circ B$. Thanks to the bi-harmonicity of $u$, we can apply Proposition \ref{thmIDENTITY} for $k=2$ to get
\[
\Delta^2 U = (\Delta^2 u)(B) \Phi^{\frac{n+5}2} =0.
\]
Thus, we have just proved that $U$ is a biharmonic extension of $f$ to the upper halfspace $\R_+^{n+1}$. Recall that the Jacobian matrix of $B$ is given by
\[
DB(x,y) =
\begin{pmatrix}
\Phi I_n - \Phi^2 x \otimes x & - \Phi^2 x(1+y) \\
\Phi^2 x^t(1+y) &- \Phi + \Phi^2 (1+y)^2
\end{pmatrix}.
\]
Hence
\begin{align}\label{eqPartial_yU(x,y)}
\begin{aligned}
\partial_y U=& 
\left[ 
\begin{aligned}
\sum_{j=1}^n (\partial_j u)(B ) \big(- \Phi^2 x_j(1+y) \big) + \\
(\partial_{n+1} u)(B ) \big( -\Phi + \Phi^2 (1+y)^2 \big) 
\end{aligned}
\right] \Phi^{\frac{n-3}2}
 -\frac{ n-3}2 u (B) (1+y) \Phi^{\frac{n-1}2}.
\end{aligned}
\end{align}
In particular, there holds
\[
\begin{aligned}
\partial_y U(x,0) =& -\Big[ B(x,0) \cdot \nabla u(B(x,0)) +\frac{n-3}2 u(B(x,0)) \Big]\Big( \frac 2{1+|x|^2} \Big)^{\frac{n-1}2}.
\end{aligned}
\]
Notice that $B(x,0)$ is normal to $\mathbb S^n$, thus $B(x,0) \cdot \nabla u(B(x,0))$ becomes $\partial_\nu u(B(x,0))$.
Thus the Neumann boundary condition \eqref{eq:4NeumannCondition} becomes
\[
\partial_y U(x,0) = 0,
\]
thanks to $u(B(x,0)) = f(B(x,0))$. From this, we can apply the Sobolev trace inequality \eqref{eq:SobolevTraceSpaceOrder4} for $U$. Our aim is to transform this trace inequality on $\R_+^{n+1}$ to the desired trace inequality on $\bB^{n+1}$. To this purpose, we need to compute $\int_{\R^n} |U(x,0)|^{\frac{2n}{n-3}} dx$ and $\int_{\R^{n+1}_+} |\Delta U(x,y)|^2 dx dy$ as shown in the rest of our argument.

\medskip\noindent\textbf{Step 2}. First we compute $\int_{\R^n} |U(x,0)|^{\frac{2n}{n-3}} dx$ in terms of $u$. Still using the stereographic projection $\mathcal S$ we deduce that
\[
|U(x,0)|^{\frac{2n}{n-3}} = |u(\mathcal S(x))|^{\frac{2n}{n-3}} J_{\mathcal S}(x) =|f(\mathcal S(x))|^{\frac{2n}{n-3}} J_{\mathcal S}(x).
\]
From this we deduce that
\begin{equation}\label{eq:4LHS}
\int_{\R^n} |U(x,0)|^{\frac{2n}{n-3}} dx = \int_{\bS^n} |f|^{\frac{2n}{n-3}} d\omega .
\end{equation}

\medskip\noindent\textbf{Step 3}. We now compute $\Delta U(x,y)$ in terms of $u$. Clearly,
\[
\Delta U = \Delta (u \circ B) \Phi^{\frac{n-3}2} + 2 \big \langle \nabla (u \circ B ) ,\nabla \Phi^{\frac{n-3}2} \big \rangle  + u( B ) \Delta \Phi^{\frac{n-3}2}.
\]
By Lemma \ref{lem:NablaDelta}, there holds
\[
\nabla \Phi^{\frac{n-3}2} =-\frac{n-3}2 ( x, 1+y ) \Phi^\frac{n-1}2 ,\quad \Delta \Phi^{\frac{n-3}2} =-(n-3) \Phi^\frac{n-1}2.
\]
In view of Lemma \ref{lem:Identity}, we easily get
\[
\nabla (u \circ B ) = - \Phi \langle \nabla u(B ), (x,-1-y) \rangle
\]
and
\[
\Delta (u \circ B) = \Phi^2 \Delta u(B ) - (n-1) \Phi^2 \langle \nabla u(B ), (x,-1-y) \rangle.
\]
Thus, we have just compute
\begin{align}\label{eqO4-DeltaU}
\begin{aligned}
\Delta U = &\big[ \Delta u(B ) - (n-1) \langle \nabla u(B ), (x,-1-y) \rangle \big] \Phi^{\frac{n+1}2}\\
&+ (n-3) \langle \nabla u(B ), (x, -1 -y ) \rangle \Phi^{\frac{n+1}2}   -(n-3) u (B ) \Phi^{\frac{n-1}2}\\
=& \Delta u(B ) \Phi^{\frac{n+1}2} - 2 \langle \nabla u (B ) , (x,-1-y) \rangle \Phi^{\frac{n+1}2} -(n-3) u (B ) \Phi^{\frac{n-1}2}.
\end{aligned}
\end{align}
Making use of \eqref{eqPhi-B-e}, we can further write $\Delta U$ as follows
\[
\begin{aligned}
\Delta U = \Big[ \Delta u(B ) 
- 4 \frac{ \langle \nabla u (B ) , B -e_{n+1} \rangle}{|B(x,y) -e_{n+1}|^2} 
- 2(n-3) \frac{ u (B )}{|B(x,y) -e_{n+1}|^2} 
\Big] \Phi^{\frac{n+1}2}.
\end{aligned}
\]
From this, integrating over $\R_+^{n+1}$ leads us to
\begin{align}\label{eq:order41}
\begin{aligned}
\int_{\R^{n+1}_+}|\Delta U|^2 dx dy =& \int_{\bB^{n+1}} |\Delta u|^2 dz + 16\int_{\bB^{n+1}} \frac{ \langle \nabla u, z-e_{n+1} \rangle^2}{|z-e_{n+1}|^4} dz \\
& + 4(n-3)^2 \int_{\bB^{n+1}} \frac{u^2}{|z-e_{n+1}|^4} dz \\
&-8\int_{\bB^{n+1}} \Delta u \, \Big\langle \nabla u, \frac{z-e_{n+1}}{|z-e_{n+1}|^2} \Big\rangle dz \\
& -4(n-3) \int_{\bB^{n+1}} \frac{u \Delta u}{|z-e_{n+1}|^2}dz \\
&+ 16(n-3) \int_{\bB^{n+1}}u \Big\langle \nabla u, \frac{z-e_{n+1}}{|z-e_{n+1}|^4} \Big\rangle dz.
\end{aligned}
\end{align}
We now compute the last three terms on the right hand side of \eqref{eq:order41}. First we compute the term involving $\Delta u \, \langle \nabla u , (z-e_{n+1})/|z-e_{n+1}|^2 \rangle $. Recall that $\omega = x/|x|$. Using integration by parts, we first have
\[
\begin{aligned}
\int_{\bB^{n+1}} \Delta u \, &\Big\langle \nabla u ,  \frac{z-e_{n+1}}{|z-e_{n+1}|^2} \Big\rangle dz \\
=& -\int_{\bB^{n+1}} \Big\langle \nabla u , \nabla \Big\langle \nabla u , \frac{z-e_{n+1}}{|z-e_{n+1}|^2} \Big\rangle \Big\rangle dz 
+ \int_{\bS^n} \frac{\partial u}{\partial \nu} \Big \langle  \nabla u , \frac{\omega-e_{n+1}}{|\omega-e_{n+1}|^2} \Big \rangle  d\omega \\
=& -\int_{\bB^{n+1}} \Big\langle \nabla u , \frac{ \nabla \big\langle \nabla u , z-e_{n+1} \big\rangle}{|z-e_{n+1}|^2} \Big\rangle dz 
+ 2\int_{\bB^{n+1}} \frac{\langle \nabla u, z-e_{n+1} \rangle^2}{|z-e_{n+1}|^4} dz \\
& + \int_{\bS^n} \frac{\partial u}{\partial \nu}\nabla u\cdot \frac{\omega-e_{n+1}}{|\omega-e_{n+1}|^2} d\omega ,
\end{aligned}
\]
where we have used $\nabla |z -e_{n+1}|^2 = 2 (z-e_{n+1} )$ once to get the middle term on the right most hand side of the preceding computation. It remains to compute the first term on the right most hand side. For simplicity, we use the Einstein convention with indexes running from $1$ to $n+1$. It is not hard to verify that
\[
\begin{aligned}
\partial^i u \, \partial_i ( \langle \nabla u , z-e_{n+1} \rangle ) =& \partial^i u \, \partial_i \big( \partial^j u \, (z-e_{n+1})_j \big) \\
= & \partial^i u \, \big[ (\partial_i \partial^j u) \, (z-e_{n+1})_j + (\partial^j u) \, \partial_i \big( (z-e_{n+1})_j \big)\big] \\
=& \frac 12 \partial^j \big( |\nabla u|^2 \big) \, (z-e_{n+1})_j + (\partial^i u)(\partial_i u).
\end{aligned}
\]
Thus,
\[
\Big\langle \nabla u , \frac{ \nabla \big\langle \nabla u , z-e_{n+1} \big\rangle}{|z-e_{n+1}|^2} \Big\rangle = \frac{ |\nabla u|^2 + \frac 12 \big\langle \nabla (|\nabla u|^2) , z-e_{n+1} \big\rangle}{|z-e_{n+1}|^2} 
\]
which helps us to write
\[
\begin{aligned}
\int_{\bB^{n+1}} \Delta u \, \Big\langle& \nabla u ,  \frac{z-e_{n+1}}{|z-e_{n+1}|^2} \Big\rangle dz \\
=& -\int_{\bB^{n+1}} \frac{|\nabla u|^2}{|z -e_{n+1}|^2} dz + 2\int_{\bB^{n+1}} \frac{\langle \nabla u, z-e_{n+1} \rangle^2}{|z-e_{n+1}|^4} dz \\
& -\frac12\int_{\bB^{n+1}} \Big \langle  \nabla (|\nabla u|^2),  \frac{z-e_{n+1}}{|z-e_{n+1}|^2} \Big \rangle   dz + \int_{\bS^n} \frac{\partial u}{\partial \nu} \Big \langle  \nabla u , \frac{\omega-e_{n+1}}{|\omega-e_{n+1}|^2} \Big \rangle  d\omega 
\end{aligned}
\]
and by applying integration by parts we arrive at
\begin{align}\label{eq:IBPorder4}
\begin{aligned}
\int_{\bB^{n+1}} \Delta u \,& \Big\langle \nabla u ,  \frac{z-e_{n+1}}{|z-e_{n+1}|^2} \Big\rangle dz \\
=&\frac{n-3}2\int_{\bB^{n+1}} \frac{|\nabla u|^2}{|z -e_{n+1}|^2} dz + 2\int_{\bB^{n+1}} \frac{\langle \nabla u, z-e_{n+1} \rangle^2}{|z-e_{n+1}|^4} dz \\
& - \frac12 \int_{\bS^n} |\nabla u|^2 \frac{ \langle \omega -e_{n+1}, \omega \rangle}{|\omega -e_{n+1}|^2} d\omega 
+ \int_{\bS^n} \frac{\partial u}{\partial \nu} \Big \langle  \nabla u , \frac{\omega-e_{n+1}}{|\omega-e_{n+1}|^2} \Big \rangle  d\omega .
\end{aligned}
\end{align}
Now we compute the term involving $u \langle \nabla u, (z -e_{n+1})/|z -e_{n+1}|^4 \rangle$ . We again apply integration by parts to get
\begin{align}\label{eq:IBPorder42}
\begin{aligned}
\int_{\bB^{n+1}} u \Big \langle \nabla u, &\frac{z -e_{n+1}}{|z -e_{n+1}|^4} \Big\rangle dz \\
=& \frac12 \int_{\bB^{n+1}} \Big \langle \nabla u^2, \frac{z -e_{n+1}}{|z -e_{n+1}|^4} \Big \rangle dz \\
=& -\frac{n-3}2 \int_{\bB^{n+1}} \frac{u^2}{|z -e_{n+1}|^4} 
+\frac12 \int_{\bS^n} u^2 \frac{\langle \omega -e_{n+1}, \omega \rangle}{|\omega -e_{n+1}|^4} d\omega .
\end{aligned}
\end{align}
Finally, the term involving $u \Delta u/|z-e_{n+1}|^2$ can be computed similarly to get
\begin{align}\label{eq:IBPorder41}
\begin{aligned}
\int_{\bB^{n+1}} \frac{u \Delta u}{|z-e_{n+1}|^2}dz =& \int_{\bS^n}\frac{\partial u}{\partial \nu} \frac{u}{|\omega -e_{n+1}|^2} d\omega 
- \int_{\bB^{n+1}} \Big \langle \nabla u , \nabla \frac{u}{|z-e_{n+1}|^2} \Big \rangle dz\\
=&-\int_{\bB^{n+1}} \frac{|\nabla u|^2}{|z-e_{n+1}|^2}dz + 2\int_{\bB^{n+1}} \Big \langle  u\nabla u , \frac{z -e_{n+1}}{|z -e_{n+1}|^4} \Big \rangle dz \\
& + \int_{\bS^n}\frac{\partial u}{\partial \nu} \frac{u}{|\omega -e_{n+1}|^2} d\omega \\
=&-\int_{\bB^{n+1}} \frac{|\nabla u|^2}{|z-e_{n+1}|^2}dz -(n-3) \int_{\bB^{n+1}} \frac{u^2}{|z -e_{n+1}|^4} \\
&+ \int_{\bS^n} u^2 \frac{\langle \omega -e_{n+1}, \omega \rangle}{|\omega -e_{n+1}|^4} d\omega + \int_{\bS^n}\frac{\partial u}{\partial \nu} \frac{u}{|\omega -e_{n+1}|^2} d\omega ,
\end{aligned}
\end{align}
where we have used \eqref{eq:IBPorder42} once. Plugging \eqref{eq:IBPorder4}, \eqref{eq:IBPorder42}, \eqref{eq:IBPorder41} into \eqref{eq:order41} and using
\[
(\omega -e_{n+1})\cdot \omega = 1 - \omega_{n+1}, \quad |\omega -e_{n+1}|^2 = 2(1 -\omega_{n+1})
\]
to get
\[
\begin{aligned}
\int_{\R^{n+1}_+}|\Delta U|^2 dx dy &= \int_{\bB^{n+1}} |\Delta u|^2 dz + 2 \int_{\bS^n} |\nabla u|^2 d\omega -4\int_{\bS^n} \frac{\partial u}{\partial \nu} \Big \langle  \nabla u ,\frac{\omega -e_{n+1}}{1-\omega_{n+1}} \Big \rangle  d\omega \\
&\quad -2(n-3) \int_{\bS^n} \frac{\partial u}{\partial \nu} \frac{u}{1 -\omega_{n+1}} d\omega + (n-3) \int_{\bS^n} \frac{u^2}{1 -\omega_{n+1}} d\omega .
\end{aligned}
\]
Using the decomposition $\nabla u (\omega ) = \omega \partial_\nu u(\omega ) + \widetilde\nabla u(\omega )$, the assumptions $\partial_\nu u = -((n-3)/2) f$ and $u =f$ on $\bS^n$, and the fact that $\langle \widetilde \nabla u, \omega \rangle=0$, we further have
\begin{align}\label{eq:order43}
\begin{aligned}
\int_{\R^{n+1}_+}|\Delta U|^2 dx dy = &\int_{\bB^{n+1}} |\Delta u|^2 dz + 2\int_{\bS^n} |\widetilde \nabla f|^2 d\omega -2 \int_{\bS^n} \Big(\frac{\partial u}{\partial \nu}\Big)^2 d\omega \\
& -4\int_{\bS^n} \frac{\partial u}{\partial \nu} \Big \langle \widetilde \nabla u , \frac{\omega -e_{n+1}}{1-\omega_{n+1}} \Big \rangle d\omega-2(n-3) \int_{\bS^n} \frac{\partial u}{\partial \nu} \frac{u}{1 -\omega_{n+1}} d\omega \\
& + (n-3) \int_{\bS^n} \frac{u^2}{1 -\omega_{n+1}} d\omega \\
=& \int_{\bB^{n+1}} |\Delta u|^2 dz + 2\int_{\bS^n} |\widetilde \nabla f|^2 d\omega - \frac{(n-3)^2}2 \int_{\bS^n} f^2 d\omega \\
& +(n-3) \int_{\bS^n} \Big \langle \widetilde \nabla f^2, \frac{\omega -e_{n+1}}{1-\omega_{n+1}} \Big\rangle d\omega \\
&+ (n-3)(n-2) \int_{\bS^n} \frac{f^2}{1-\omega_{n+1}} d\omega .
\end{aligned}
\end{align}
Note that $\widetilde \nabla f^2 \cdot \omega =0$, hence
\begin{align}\label{eq:xcxc}
\begin{aligned}
\int_{\bS^n} \Big \langle\widetilde \nabla f^2 , \frac{\omega -e_{n+1}}{1-\omega_{n+1}} \Big \rangle d\omega 
&=\int_{\bS^n} \Big \langle \widetilde \nabla f^2 ,  \frac{\omega -e_{n+1}}{1-\omega_{n+1}} - \frac{\langle \omega -e_{n+1}, \omega \rangle}{1-\omega_{n+1}}\omega \Big \rangle d\omega \\
&=\int_{\bS^n}\Big \langle \widetilde \nabla f^2 , \frac{\omega \omega_{n+1} -e_{n+1}}{1 -\omega_{n+1}} \Big \rangle d\omega \\
&= \int_{\bS^n} \big \langle \widetilde \nabla f^2, \widetilde \nabla \big(\log (1 -\omega_{n+1})\big) \big \rangle d\omega \\
&= -\int_{\bS^n} f^2 \widetilde \Delta \lt(\log (1 -\omega_{n+1})\rt) d\omega ,
\end{aligned}
\end{align}
here $\widetilde \Delta$ denotes the spherical Laplacian on $\bS^n$. An easy computation yields
\[\begin{aligned}
\widetilde \Delta \big(\log (1 -\omega_{n+1})\big) =& \Delta \Big(\log \big(1 - \frac{x_{n+1}}{|x|} \big)\Big) \Big|_{|x|=1}\\
=&\frac{ n |x| x_{n+1} -|x|^2 - (n-1) x_{n+1}^2 }{ |x|^2 (|x|-x_{n+1})^2 }\Big|_{|x|=1}\\
=& (n-1)\frac{\omega_{n+1}}{1-\omega_{n+1}} -\frac1{1 -\omega_{n+1}}.
\end{aligned}\]
The previous equality together with \eqref{eq:xcxc} and \eqref{eq:order43} implies 
\begin{equation}\label{eq:4RHS}
\int_{\R^{n+1}_+}|\Delta U|^2 dx dy = \int_{\bB^{n+1}} |\Delta u|^2 dz + 2\int_{\bS^n} |\widetilde \nabla f|^2 d\omega + b_n \int_{\bS^n} f^2 d\omega,
\end{equation}
where $b_n = (n+1)(n-3)/2$. Thus, combining \eqref{eq:4LHS} and \eqref{eq:4RHS} gives
\[
\begin{aligned}
2 \frac{\Gamma(\frac{n+3}2)}{\Gamma(\frac{n-3}2)} \omega_n^{3/n} & \Big(\int_{\bS^n} |f|^{\frac{2n}{n-3}} d\omega \Big)^{\frac{n-3}n} \\
&\leqslant \int_{\bB^{n+1}} |\Delta u|^2 dz + 2\int_{\bS^n} |\widetilde{\nabla} f|^2 d\omega + b_n \int_{\bS^n} |f|^2 d\omega
\end{aligned}
\]
provided $u$ is a biharmonic extension of $f$ to $\mathbb B^n$, belonging to the set $\mathscr V_f$.

\medskip\noindent\textbf{Step 4}. In the final step, we prove Inequality \eqref{eq:TraceOrder4}. Indeed, given $f\in C^\infty(\bS^n)$ it is well-known that the minimizing problem
\begin{equation}\label{eq:varproblem4}
\inf_{w \in \mathscr V_f} \int_{\bB^{n+1}} |\Delta w|^2 dx,
\end{equation}
where $\mathscr V_f $ is given in \eqref{eq:varproblem4set}, is attained by some bi-harmonic function $u$ on $\bB^{n+1}$. In addition, $u \equiv f$ and $\partial_\nu u = -((n-3)/2)f$ on $\bS^n$. Therefore, we can repeat from Step 1 to Step 3 to get the following estimate
\begin{equation}\label{eq:TraceOrder4p}
\begin{aligned}
2 \frac{\Gamma(\frac{n+3}2)}{\Gamma(\frac{n-3}2)} \omega_n^{3/n} & \Big(\int_{\bS^n} |f|^{\frac{2n}{n-3}} d\omega \Big)^{\frac{n-3}n} \\
&\leqslant \int_{\bB^{n+1}} |\Delta u|^2 dz + 2\int_{\bS^n} |\widetilde{\nabla} f|^2 d\omega + b_n \int_{\bS^n} |f|^2 d\omega.
\end{aligned}
\end{equation}
Since $u$ is a minimizer of \eqref{eq:varproblem4}, any smooth extension $v$ belonging to the set $\mathscr V_f$ enjoys the estimate
\[
\int_{\bB^{n+1}} |\Delta u|^2 dx \leqslant \int_{\bB^{n+1}} |\Delta v|^2 dx.
\]
The desired inequality follows from the preceding estimate and \eqref{eq:TraceOrder4p}. The assertion for which equality in \eqref{eq:TraceOrder4} is attained is already known; see \cite[page 2739]{ac2015}.


\subsection{A Beckner type inequality of order four on $\bB^{n+1}$}
\label{subsec-4Beckner}

This subsection is devoted to a Beckner type trace inequality of order four in a same fashion of Beckner's inequality in Theorem \ref{thm2BecknerTypeNoWeight}. We do not treat the case with weights in the present paper and leave it for future papers. Let $f \in C^\infty(\bS^n)$ and let $u$ be a biharmonic extension of $f$ to $\bB^{n+1}$ satisfying certain boundary conditions as in \eqref{eq:4NeumannCondition}, namely
\begin{equation}\label{eq:4BecknerRequire}
\partial_\nu u =  -\frac{n-3}2 f
\end{equation}
on $\bS^n$. As in the previous section, we shall work with spherical harmonics. To this purpose, we decompose
\[
f (\omega)= \sum_{k=0}^\infty Y_k(\omega ),
\]
where $\omega = x/|x|$. From this we decompose $u$ to get 
\[
u (z) = \sum_{k=0}^\infty f_k(r) Y_k(\omega ).
\] 
Hence, it follows from \eqref{eq:4BecknerRequire} that the coefficients $f_k$ satisfy
\[
L_k^2 f_k = 0
\]
for any $r\in [0,1)$ and definitely $f_k(1) =1$. Solving the above differential equation gives
\[
f_k (r) = c_1(k) r^k + c_2(k) r^{k+2} 
\]
for some constants $c_1(k)$ and $c_2(k)$ to be determined. In fact, these constants can be computed explicitly by using the boundary conditions in \eqref{eq:4BecknerRequire}. Indeed, the condition $u = f$ on $\mathbb S^n$ implies that
\[
c_1 (k) + c_2(k) = 1
\]
for all $k \geqslant 1$. Now the condition $\partial_\nu u = -((n-3)/2)f$ on $\mathbb S^n$ tells us that
\[
k c_1 (k) + (k+2) c_2(k) =-\frac{n-3}2.
\]
From these facts, we compute to get
\[
c_1 (k) =\frac{n+1+2k}4 , \quad c_2(k) = -\frac{n-3+2k}4.
\]
Now we use integration by parts to get
\[
\begin{aligned}
\int_{\bB^{n+1}} (\Delta u)^2 dz =& \int_{\bB^{n+1}} u \Delta^2u dz + \int_{\mathbb S^n}  \Delta u   \langle \nabla u, \omega \rangle \big|_{|z| = 1} d\omega -  \int_{\mathbb S^n}  u  \langle \nabla \Delta u , \omega \rangle  \big|_{|z| = 1} d\omega \\
= & -\frac{n-3}2 \int_{\mathbb S^n} f \Delta u  d\omega -  \int_{\mathbb S^n}  u  \langle \nabla \Delta u , \omega \rangle  \big|_{|z| = 1} d\omega.
\end{aligned}
\]
Recall that
\[
u (z) = \sum_{k=0}^\infty \big[ c_1(k) r^k + c_2(k) r^{k+2}  \big] Y_k(\omega ),
\]
which implies that
\[
\Delta u (z) = 2 \sum_{k=0}^\infty (n+1+2k) c_2(k) r^kY_k(\omega ).
\]
Hence
\[
\langle \nabla \Delta u , \omega \rangle  \big|_{|z| = 1} = 2 \sum_{k=0}^\infty k(n+1+2k) c_2(k)  Y_k(\omega ).
\]
We are now in position to get
\begin{equation}\label{eq:4BecknerDeta^2}
\int_{\bB^{n+1}} (\Delta u)^2 dz = \frac 14 \sum_{k=0}^\infty \int_{\mathbb S^n} (n+1+2k) (n-3+2k)^2 | Y_k|^2  d\omega.
\end{equation}
Now let $0< s < n/3$. We define the function $F$ on $\R^n$ by
\[
F(x) = f(\mathcal S(x)) J_{\mathcal S}(x)^{\frac{n-3s}{2n}}.
\]
Then we have
\[
\int_{\bS^n} |f|^{\frac{2n}{n-3s}} d\omega = \int_{\R^n} |F|^{\frac{2n}{n-3s}} dx
\]
and as in \eqref{eq:GN} we still have
\[
\int_{\R^n} F(x) (-\Delta)^{3s/2} F(x) dx= \sum_{k=0}^\infty\frac{\Gamma(k+n/2+3s/2)}{\Gamma(k+n/2-3s/2)} \int_{\bS^n} |Y_k|^2 d\omega .
\]
We now use the fractional Sobolev inequality \eqref{eq:SobolevSpaceOrderS/2} to get
\begin{equation}\label{eq:4SobolevSpaceOrderS/2}
\frac{\Gamma(\frac{n+3s}2)}{\Gamma(\frac{n-3s}2)} \omega_n^{3s/n} \Big(\int_{\R^n} |F|^{\frac{2n}{n-3s}} dx\Big)^{\frac{n-3s}n} \leqslant\int_{\R^n} F(x) (-\Delta)^{3s/2} F(x) dx.
\end{equation}
Combining \eqref{eq:4SobolevSpaceOrderS/2} and \eqref{eq:4BecknerDeta^2} gives
\begin{equation}\label{eq:4BecknerTypeHarmonicNoWeight}
\begin{split}
2\frac{\Gamma(\frac{n+3s}2)}{\Gamma(\frac{n-3s}2)} \omega_n^{3s/n} &  \Big(\int_{\bS^n} |f|^{\frac{2n}{n-3s}} d\omega \Big)^{\frac{n-3s}n}\\
\leqslant & 2\sum_{k=0}^\infty           \frac{\Gamma(k+n/2+3s/2)}{\Gamma(k+n/2-3s/2)}         \int_{\bS^n} |Y_k|^2 d\omega \\
=& \int_{\bB^{n+1}} (\Delta u)^2 dz + \sum_{k=0}^\infty 
\left( 
\begin{split}
&2\frac{\Gamma(k+n/2+3s/2)}{\Gamma(k+n/2-3s/2)} \\
&-\frac{(n+1+2k) (n-3+2k)^2}4
\end{split} 
\right)  \int_{\bS^n} |Y_k|^2 d\omega .
\end{split}
\end{equation}

Clearly, equality in \eqref{eq:4BecknerTypeHarmonicNoWeight} occurs if, and only if, equality in the fractional Sobolev inequality \eqref{eq:4SobolevSpaceOrderS/2} occurs. We are now in position to state our next result.

\begin{theorem}\label{thm4BecknerTypeNoWeight}
Let $n\geqslant 3$ and $0< s < n/3$. Let $f\in C^\infty(\bS^n)$ and $v$ be a smooth extension of $f$ to the unit ball $\bB^{n+1}$ satisfying the boundary condition
\[
\frac {\partial v}{\partial \nu} = -\frac{n-3}2 f
\] 
on $\mathbb S^n$. Suppose that $f$ has a decomposition on spherical harmonics as $f =\sum_{k=0}^\infty Y_k(\omega )$. Then the following inequality holds
\begin{equation}\label{eq:4BecknerTypeNoWeight}
\begin{split}
2\frac{\Gamma(\frac{n+3s}2)}{\Gamma(\frac{n-3s}2)} & \omega_n^{3s/n}  \Big(\int_{\bS^n} |f|^{\frac{2n}{n-3s}} d\omega \Big)^{\frac{n-3s}n}
\leqslant \int_{\bB^{n+1}} (\Delta v)^2 dz\\
& + \sum_{k=0}^\infty \Big( 2\frac{\Gamma(k+n/2+3s/2)}{\Gamma(k+n/2-3s/2)} -\frac{(n+1+2k) (n-3+2k)^2}4 \Big) \int_{\bS^n} |Y_k|^2 d\omega .
\end{split}
\end{equation}
Moreover, equality in \eqref{eq:4BecknerTypeNoWeight} holds if, and only if, $v$ is a biharmonic extension of a function of the form
\[
c|1-\langle z_0, \xi \rangle |^{-(n-3s)/2},
\]
where $c>0$ is a constant, $\xi \in \mathbb S^n$, and $z_0$ is some fixed point in the interior of $\mathbb B^{n+1}$, and $v$ fulfills the above boundary condition.
\end{theorem}

Thanks to \eqref{eq:4BecknerTypeHarmonicNoWeight}, the proof of Theorem \ref{thm4BecknerTypeNoWeight} follows the same lines as in Step 4 of the previous subsection; hence we omit the details. The equality case in \eqref{eq:4BecknerTypeNoWeight} can be obtained by following the argument used in the proof of Theorem \ref{thm2BecknerTypeWeight}.

As an application of Theorem \ref{thm4BecknerTypeNoWeight}, let us consider the case $n \geqslant 4$ and $s = 1$. In this case, we easily re-obtain Inequality \eqref{eq:TraceOrder4}, namely,
\[
2\frac{\Gamma(\frac{n+3 }2)}{\Gamma(\frac{n-3 }2)} \omega_n^{3/n} \Big(\int_{\bS^n} |f|^{\frac{2n}{n-3}} d\omega \Big)^{\frac{n-3}n}  \leqslant \int_{\bB^{n+1}} |\Delta v|^2 dz + 2\int_{\bS^n} |\widetilde{\nabla} f|^2 d\omega + b_n \int_{\bS^n} |f|^2 d\omega
\]
with $b_n = (n+1)(n-3)/2$. This is because by the identity $\widetilde \Delta Y_k = - k(n+k-1) Y_k$ we obtain
\begin{equation}\label{eq:SHDDeltaf}
\widetilde \Delta f = - \sum_{k=0}^\infty k(n-1+k) Y_k (\omega),
\end{equation}
which leads to
\[\begin{aligned}
2 \int_{\bS^n} |\widetilde{\nabla} f|^2 d\omega + & \frac{(n-3)(n+1)}2 \int_{\bS^n} |f|^2 d\omega \\
 = & \sum_{k=0}^\infty \Big[ 2k(n-1+k) + \frac{(n-3)(n+1)}4 \Big] \int_{\mathbb S^n} |Y_k|^2 (\omega) d\omega
\end{aligned}\]
and
\[
2\frac{\Gamma(k+n/2+3 /2)}{\Gamma(k+n/2-3 /2)} -\frac{(n+1+2k) (n-3+2k)^2}4=2k(n-1+k) + \frac{(n-3)(n+1)}4.
\]
Clearly, Inequality \eqref{eq:4BecknerTypeNoWeight} provide us a Beckner type trace inequality of order four. Furthermore, the analysis in obtaining \eqref{eq:TraceOrder4} from \eqref{eq:4BecknerTypeNoWeight} is less involved and this suggests us to adopt this approach to prove the Sobolev trace inequality of order six on $\mathbb S^n$ in the next section.


\subsection{A Ledebev--Milin type inequality of order four on $\bB^4$}

In the last part of this section, we treat the limiting case $n = 3$. Our aim is to derive a Ledebev--Milin type inequality of order four similar to the one obtained in \cite[Theorem B]{ac2015}. To this purpose, we follow the strategy used to obtain Theorem \ref{thmSMOrder2}.

Suppose that $f \in C^\infty(\bS^3)$ with $\int_{\bS^3} f d\omega =0$ and let $v$ be a smooth extension of $f$ in $\bB^4$. As before, we decompose $f$ in terms of spherical harmonics to get
\[
f = \sum_{k=1}^\infty Y_k(\omega )
\]
Note that the function $ 1 + \frac 32 (1-s) v$ is also a smooth extension of $1 + \frac 32 (1-s) f$ in $\mathbb B^4$ and $\omega_3 = 2\pi^2$. Therefore, we can apply Theorem \ref{thm4BecknerTypeNoWeight} to get
\begin{align*}
\frac 23 \frac{\Gamma(\frac{3+3s}2)}{\Gamma(\frac{3-3s}2)}& \Big[ \Big(\frac 1{2\pi^2} \int_{\bS^3} \Big(1 + \frac{3(1-s)}2  f\Big)^{\frac{2}{1-s}} d\omega \Big)^{1-s}  - 1 \Big] \\
 \leqslant& \frac{ 3(1-s)^2}{8\pi^2} \int_{\bB^4} (\Delta v)^2 dz   \\
& + \frac{3(1-s)^2}{8\pi^2} \sum_{k=1}^\infty \Big(2\frac{\Gamma(k+\frac{3+3s}2)}{\Gamma(k+ \frac{3-3s}2)} -2(k+2)k^2 \Big)\int_{\bS^3}|Y_k|^2 d\omega .
\end{align*}
Dividing both sides by $(1-s)^2$ and making use of $\Gamma(\frac{5-3s}2) = (3/2)(1-s)\Gamma(\frac{3-3s}2)$ to get
\begin{align*}
\frac{\Gamma(\frac{3+3s}2)}{\Gamma(\frac{5-3s}2)} \frac 1{1-s} & \Big[ \Big(\frac 1{2\pi^2} \int_{\bS^3} \Big(1 +\frac{3(1-s)}2 f\Big)^{\frac{2}{1-s}} d\omega \Big)^{1-s}  - 1 \Big] \\
\leqslant &\frac 3{8\pi^2} \int_{\bB^4} (\Delta v)^2 dz   \\
& + \frac 3{8\pi^2}\sum_{k=1}^\infty \Big( 2\frac{\Gamma(k+\frac{3+3s}2)}{\Gamma(k+ \frac{3-3s}2)} -2(k+2)k^2  \Big)\int_{\bS^3}|Y_k|^2 d\omega .
\end{align*}
Note that 
\[
\lim_{s\to 1 } \Big( 2\frac{\Gamma(k+\frac{3+3s}2)}{\Gamma(k+ \frac{3-3s}2)} -2(k+2)k^2 \Big) =2k(k+2)
\]
for any $k \geqslant 1$ and $ \int_{\bS^3} |\widetilde{\nabla} f|^2 d\omega = \sum_{k=0}^\infty  k(k+2)  \int_{\mathbb S^n} |Y_k|^2 (\omega) d\omega$. Hence letting $s \nearrow 1$, we obtain
\[
2\log \Big(\frac1{2\pi^2} \int_{\bS^3} e^{3f} d\omega \Big) \leqslant \frac 3{8\pi^2}\int_{\bB^4} |\nabla v|^2 dx + \frac 3{4\pi^2} \int_{\bS^3} |\widetilde \nabla f|^2 d\omega 
\]
for any smooth function $f$ with $\int_{\bS^3} f d\omega =0$. For general function $f$, we apply the previous inequality for $f-(1/(2\pi^2)) \int_{\bS^3} f d\omega$ to get the following theorem.

\begin{theorem}[Lebedev--Milin inequality of order four; see \cite{ac2015}]\label{thmSMOrder4}
Let $f\in C^\infty(\bS^3)$ and suppose that $u$ is a smooth extension of $f$ to the unit ball $\bB^4$. If $v$ satisfies the Neumann boudanry condition 
\[
\frac{\partial v}{\partial \nu} = 0
\]
on $\bS^3$, then we have the following sharp trace inequality
\begin{equation}\label{eqSMorder4}
\log \lt(\frac1{2\pi^2} \int_{\bS^3} e^{3f} d\omega \rt) \leqslant \frac 3{16\pi^2} \int_{\bB^4} (\Delta v)^2 dx + \frac 3{8\pi^2} \int_{\bS^3} |\widetilde\nabla f|^2 d\omega + \frac 3{2\pi^2} \int_{\bS^3} f d\omega .
\end{equation}
Moreover, equality in \eqref{eqSMorder4} holds if, and only if, $v$ is a biharmonic extension of a function of the form
\[
c - \log |1 -  \langle z_0, \xi \rangle|,
\]
where $c>0$ is a constant, $\xi \in \bS^3$, $z_0$ is some fixed point in the interior of $\mathbb B^4$, and $v$ fulfills the boundary condition $\partial_\nu v = 0$.
\end{theorem}

Clearly, Inequality \eqref{eqSMorder4} can be rewritten as follows
\[
2\log \Big(\frac1{2\pi^2} \int_{\bS^3} e^{3 (f - \overline f)} d\omega \Big) \leqslant \frac 3{8\pi^2}\int_{\bB^4} |\nabla v|^2 dx + \frac 3{4\pi^2} \int_{\bS^3} |\widetilde \nabla f|^2 d\omega 
\]
where $\overline f$ is the average of $f$ which is $(1/(2\pi^2)) \int_{\bS^3} f d\omega$.


\section{Sobolev trace inequality of order six}
\label{sec-Order6}

\subsection{Sobolev trace inequality of order six on $\R_+^{n+1}$: Proof of Theorem \ref{thmTraceSpaceOrder6}}

This subsection is devoted to a proof of Theorem \ref{thmTraceSpaceOrder6}. To proceed, we first have the following observation.

\begin{proposition}\label{apx-propExtension6}
Any function $U \in W^{3,2}(\R_+^{n+1})$ satisfying
\begin{equation}\label{apd-6Equation} 
\Delta^3 U(x,y) = 0
\end{equation}
on the upper half space $\R_+^{n+1}$ and the boundary conditions
\begin{equation}\label{apd-6Boundary} 
\partial_y U(x,0) = 0, \quad \partial_y^2 U(x,0) = \lambda \Delta_x U(x,0)
\end{equation}
enjoys the following identity
\[
\int_{\R^{n+1}_+} |\nabla \Delta U(x,y)|^2 dx dy = (3\lambda^2-2\lambda +3) \int_{\R^{n}} U(x,0) (- \Delta)^{5/2} U(x,0) dx.
\]
\end{proposition}

\begin{proof}
By taking the Fourier transform in the $x$ variable on \eqref{apd-6Equation} we arrive at
\begin{equation}\label{eq:6Fourier} 
\begin{aligned}
0 =& \widehat{\Delta^3 U} (\xi,y) = \Big( -|\xi|^2 \, \text{Id} + \frac{\partial^2}{\partial y^2}\Big)^3 \widehat U(\xi, y)\\
=&   -|\xi|^6 \widehat U(\xi, y) + 3 |\xi |^4  \widehat U_{yy} (\xi, y) - 3 |\xi|^2 \widehat U_{yyyy}(\xi, y) + \widehat U_{yyyyyy}(\xi, y).
\end{aligned}
\end{equation}
Thus, we obtain an ordinary differential equation of order six for each value of $\xi$. Again from the form of \eqref{eq:6Fourier} we now consider the ODE
\begin{equation}\label{eq:6ODE} 
 \phi^{(6)}- 3 \phi^{(4)}  +3 \phi''  - \phi =0
\end{equation}
with $\phi \in W^{3, 2}([0, +\infty))$. It is an easy computation to verify that any solution $\phi$ to \eqref{eq:6ODE} satisfying the initial conditions $\phi (0) = 1$, $\phi'(0) = 0$, and $\phi''(0) = -\lambda$ must be of the form
\[\begin{split}
\phi (y) = &\Big[ 1 - C_2 - (2C_2 + C_3  - 1)y - \big (2C_2 + 2C_3 + C_4 -\frac {-\lambda+1}2 \big)y^2 \Big] e^{-y}\\
& + (C_1 + C_2 y + C_3 y^2 )e^{y}
\end{split}\]
for some constants $C_1$, $C_2$, and $C_3$. If, in addition, we assume that $\phi$ is bounded, then we find that $C_1 = C_2=C_3 = 0$, which then implies that
\[
\phi (y) = \Big(1+y + \frac {-\lambda+1}2y^2 \Big)e^{-y}.
\]
Hence we have just shown that there is a unique bounded solution $\phi$ to \eqref{eq:6ODE} satisfying $\phi (0) = 1$, $\phi'(0) = 0$, and $\phi''(0) = -\lambda$. Furthermore, by direct computation, we get
\[
 \int_0^{+\infty} \Big[ \big( -\phi  + \phi''\big)^2 + \big( -\phi' + \phi^{(3)} \big)^2 \Big] dy = 3\lambda^2-2\lambda +3.
\]
Now from \eqref{eq:6Fourier}, it is easy to verify that
\[
\widehat U(\xi, y) = \widehat u(\xi) \phi (|\xi| y)
\]
with $\partial_y \widehat U(\xi, y) = |\xi| \widehat u(\xi) \phi' (|\xi| y)$ and $\partial^2_y \widehat U(\xi, y) = |\xi|^2 \widehat u(\xi) \phi'' (|\xi| y)$. This is because by taking the Fourier transform the boundary $\partial_y U(x,0) = 0$ becomes $\partial_y \widehat U(\xi,0) = 0$ and the boundary $\partial_y^2 U(x,0) = \lambda \Delta_x U(x,0)$ becomes $\partial_y^2 \widehat U(\xi,0) = -\lambda |\xi|^2 \widehat U(\xi,0)$.

We now compute $\int_{\R^{n+1}_+} | \nabla \Delta U(x,y)|^2 dx dy$. We notice that 
\[
| \nabla \Delta U(x,y)|^2 =\big | \nabla_x \big[\Delta_x U(x,y) + \partial_y^2 U(x,y) \big] \big|^2 + (\partial_y \big[\Delta_x U(x,y) + \partial_y^2 U(x,y) \big)^2.
\]
Thus, by the Plancherel theorem, we obtain
\[\begin{aligned}
\int_{\R^{n+1}_+} | \nabla \Delta U(x,y)|^2 dx dy =& \frac 1{(2\pi)^{n}}       \int_0^{+\infty}
  \int_{\R^{n}}   |\xi|^2 \big( -|\xi|^2 \widehat U(\xi, y) +\partial_y^2 \widehat U(\xi, y) \big)^2 dy d\xi  \\
&+\frac 1{(2\pi)^{n}}       \int_0^{+\infty}
  \int_{\R^{n}} (\partial_y \big[-|\xi|^2 \widehat U(\xi, y) + \partial_y^2  \widehat U(\xi, y) \big] \big)^2    dy d\xi  \\
= & \frac 1{(2\pi)^{n}}       \int_0^{+\infty}
  \int_{\R^{n}}   |\xi|^6 \widehat u(\xi)^2 \big( -   \phi (|\xi| y)  +  \phi'' (|\xi| y) \big)^2 dy d\xi  \\
&+\frac 1{(2\pi)^{n}}       \int_0^{+\infty}
  \int_{\R^{n}}    |\xi|^6 \widehat u(\xi)^2 \big(-   \phi' (|\xi| y)  +  \phi^{(3)} (|\xi| y) \big)^2  d\xi  \\
=& \frac {3\lambda^2+2\lambda +3}{(2\pi)^{n}}\int_{\R^{n}}   |\xi|^5  \widehat u(\xi)^2  dy d\xi \\
= & (3\lambda^2+2\lambda +3) \int_{\R^n} U(x,0) (- \Delta)^{5/2} U(x,0) dx.
\end{aligned}\]
The proof is complete.
\end{proof}

We now use Proposition \ref{apx-propExtension6} to prove Theorem \ref{thmTraceSpaceOrder6}, namely, the following inequality holds
\[
(3\lambda^2-2\lambda +3) \frac{\Gamma(\frac{n+5}2)}{\Gamma(\frac{n-5}2)} \omega_n^{5/n} \Big(\int_{\R^n} |U(x,0)|^{\frac{2n}{n-5}} dx\Big)^{\frac{n-5}n} \leqslant \int_{\R^{n+1}_+} |\nabla \Delta U(x,y)|^2 dx dy
\]
for any function $U$ satisfying the boundary conditions in \eqref{apd-6Boundary}. Indeed, let us consider the following minimizing problem
\begin{equation}\label{eq:VarProblemPre6}
\inf_{w} \int_{\R^{n+1}_+} |\nabla \Delta w|^2 dx ,
\end{equation}
over all $w$ satisfying \eqref{apd-6Boundary}. It is well-known that that Problem \eqref{eq:VarProblemPre6} is attained by a function $W$ in $\R^{n+1}_+$, which satisfies all assumptions in Proposition \ref{apx-propExtension6}. Therefore, we obtain from Proposition \ref{apx-propExtension6} the identity
\[
\int_{\R^{n+1}_+} |\nabla \Delta W(x,y)|^2 dx dy = (3\lambda^2-2\lambda +3) \int_{\R^{n}} W(x,0)(- \Delta)^{5/2} W(x,0) dx.
\]
Making use of the fractional Sobolev inequality \eqref{eq:SobolevSpaceOrderS/2} to get
\begin{equation}\label{eq:6SobolevSpaceOrder}
\int_{\R^n} W(x,0) (- \Delta)^{5/2} W(x,0) dx \geqslant \frac{\Gamma(\frac{n+5}2)}{\Gamma(\frac{n-5}2)} \omega_n^{5/n} \Big(\int_{\R^n} | W(x,0) |^{\frac{2n}{n-5}} dx\Big)^{\frac{n-5}n} .
\end{equation}
Hence we have just shown that
\[
(3\lambda^2-2\lambda +3) \frac{\Gamma(\frac{n+5}2)}{\Gamma(\frac{n-5}2)} \omega_n^{5/n} \Big(\int_{\R^n} | W(x,0) |^{\frac{2n}{n-5}} dx\Big)^{\frac{n-5}n} \leqslant \int_{\R^{n+1}_+} |\nabla \Delta W(x,y)|^2 dx dy ,
\]
which yields the desired inequality since $W$ is the optimizer for the problem \eqref{eq:VarProblemPre6}. Clearly, equality in \eqref{eq:TraceSpaceOrder6}$_\lambda$ holds if, and only if, equality in \eqref{eq:6SobolevSpaceOrder} occurs, which implies that $U$ must be a triharmonic extension of a function of the form
\[
c \big(\mu +  |\xi - z_0|^2 \big)^{-(n-5)/2},
\]
where $c$ and $\mu$ are positive constants, $\xi \in \R^n$, $z_0 \in \R^n$, and $U$ also fulfills the boundary condition \eqref{eq:6SpaceNeumannCondition}$_\lambda$.


\subsection{Neumann boundary condition for extensions}

As in the Sobolev trace inequality of order four established in Theorem \ref{thmTraceOrder4}, to obtain a correct Sobolev trace inequality of order six, we need to take care of Neumann boundary conditions. The way to find correct boundary conditions is to look at the Sobolev trace inequality of order six on $\R^n$. Following this strategy, let us recall from \eqref{eq:TraceSpaceOrder6}$_{1/3}$ the trace inequality
\[
\frac 83 \frac{\Gamma(\frac{n+5}2)}{\Gamma(\frac{n-5}2)} \omega_n^{5/n} \lt(\int_{\R^n} |V(x,0)|^{\frac{2n}{n-5}}d\omega \rt)^{\frac{n-5}n} \leqslant \int_{\R^{n+1}_+} |\nabla \Delta V(x,y)|^2 dx dy
\]
satisfied by any function $V$ satisfying the following Neumann boundary conditions
\begin{equation}\label{eq:Sobolevtrace6BoundaryConditions}
\partial_y V(x,0) = 0,\qquad \partial_{y}^2 V(x,0) = \frac13 \Delta_{x} V(x,0).
\end{equation}
It is worth noticing that the boundary condition \eqref{eq:Sobolevtrace6BoundaryConditions} is slightly different from that of \cite[Theorem 3.3]{rayyang2013}. This is because following the calculation in \cite{rayyang2013}, it would be $\partial_{y}^2 V(x,0) = (1/5) \Delta_{x} V(x,0)$.

As in the proof of Theorem \ref{thmTraceOrder4}, given $f \in C^\infty (\bS^n)$, we consider the minimizing problem
\begin{equation}\label{eq:varproblem6}
\inf_{w \in \mathscr W_f} \int_{\bB^{n+1}} |\nabla \Delta w|^2 dx ,
\end{equation}
where
\begin{equation}\label{eq:varproblem6Set}
\mathscr W_f = \Big\{w : \, w \big{|}_{\bS^n} = f,\, \partial_\nu w|_{\bS^n} = -\frac{n-5}2 f,\, \partial^2_\nu w|_{\bS^n} =\frac13 \widetilde{\Delta} f +\frac{(n-5)(n-6)}{6} f \Big\}.
\end{equation}
It is well-known that that Problem \eqref{eq:varproblem6} is attained by a function $v$ in $\mathbb B ^{n+1}$. In other words, $v$ is a smooth extension of $f$ on $\bB^{n+1}$ satisfying
\begin{equation}\label{eq:varproblem6PDE}
\left\{
\begin{aligned}
\Delta^3 v &= 0 &\text{in } &\bB^{n+1},\\
\partial_\nu v  &= -\frac{n-5}2 f &\text{on }& \bS^n,\\
\partial^2_\nu v  &= \frac13 \widetilde{\Delta} f +\frac{(n-5)(n-6)}{6} f & \text{on } & \bS^n.
\end{aligned}
\right.
\end{equation}
We shall soon see that our choice for the boundary conditions in \eqref{eq:varproblem6Set} is correct. Keep following the idea used in the proof of Theorem \ref{thmTraceOrder4}, we define
\[
V(x,y) = (v \circ B)(x,y) \Phi^{\frac{n-5}2}
\]
Applying Proposition \ref{thmIDENTITY} for $k=3$ gives
\[
\Delta^3 V(x,y) = (\Delta^3 v)(B(x,y)) \Phi^{\frac{n+7}2}.
\]
By a similar computation leading us to \eqref{eqPartial_yU(x,y)}, we deduce that
\[
\partial_y V(x,0) = - \Big[ (\partial_\nu v) (B(x,0)) + \frac{n-5}2 v(B(x,0)) \Big] \Big( \frac2{1+|x|^2} \Big)^{\frac{n-3}2}.
\]
Consequently, the first boundary condition in \eqref{eq:boundarycond6}, namely
\[
\partial_\nu v \big|_{\mathbb S^n}= - \frac{n-5}2 f
\]
is equivalent to the following boundary condition
\[
\partial_y V(x,0) = 0,
\]
which is coincident with part of \eqref{eq:Sobolevtrace6BoundaryConditions}. Next we compute the second order derivative $\partial^2_{y} V$. Indeed,
\begin{align*}
\partial^2_{y} V(x,y) =&\partial_{y}^2 (v \circ B ) \Phi^{\frac{n-5}2} + 2 \partial_y (v \circ B )\partial_y \Phi^{\frac{n-5}2} + ( v \circ B) \partial_{y}^2 \Phi^{\frac{n-5}2}\\
=&\big[ \partial_{i,j}^2 v (B ) \partial_y B^i \partial_y B^j +   \partial_k v (B ) \partial_y^2 B^k  \big] \Phi^{\frac{n-5}2} \\
&-(n-5)  \partial_i v (B ) \partial_y B^i (1+y) \Phi^{\frac{n-3}2}\\
& + \frac{(n-5)(n-3)}4  v ( B ) (1+y)^2 \Phi^{\frac{n-1}2} -\frac{n-5}2 v ( B ) \Phi^{\frac{n-3}2}.
\end{align*}
Keep in mind that $\partial_y B^i(x,0) = -(2/(1+|x|^2)) B^i(x,0)$ for all $1 \leqslant i \leqslant n+1$, that
\[
\partial_{y}^2 B^i(x,0) = -\frac{2}{1+|x|^2}B^i(x,0)+ \frac8{(1+ |x|^2)^2}B^i(x,0)
\]
for $i =1,2,\ldots,n$, and that
\[\begin{aligned}
\partial_{y}^2 B^{n+1}(x,0) = &\frac4{(1+ |x|^2)^2} + \frac8{(1+ |x|^2)^2}B^{n+1}(x,0) \\
= & \frac 2{1+ |x|^2} -\frac{2}{1+|x|^2}B^{n+1}(x,0) + \frac8{(1+ |x|^2)^2}B^{n+1}(x,0).
\end{aligned}\]
Then we can verify
\[\begin{aligned}
\partial_{i,j}^2 v (B ) \partial_y B^i \partial_y B^j +   \partial_k v (B ) \partial_y^2 B^k =& \big[ \partial^2_\nu v (B) + 2\partial_\nu v (B)\big] \Big(\frac2{1+ |x|^2} \Big)^2 \\
&   + \big[  \partial_{n+1} v (B )  - \partial_\nu v (B) \big] \frac2{1+ |x|^2}
\end{aligned}\]
and
\[
\partial_i v (B ) \partial_y B^i (1+y)=-\partial_\nu v (B)  \frac2{1+ |x|^2}
\]
at $(x,0)$. Hence, with the fact that $\big[  \partial_\nu v (B )  - \partial_{n+1} v (B) \big] (1+|x|^2)/2 = \langle \nabla v(B) , (x,-1) \rangle$ at $(x,0)$, we obtain
\begin{align*}
\partial_{y}^2 V(x,0)&= 
\left(
\begin{aligned}
&\partial^2_\nu v (B(x,0))+ (n-3) \partial_\nu v(B(x,0))\\
& - \langle \nabla v(B(x,0)) , (x,-1) \rangle  \\
&+ \frac{(n-5)(n-3)}4 v(B(x,0))\\
& -\frac{n-5}4 v(B(x,0)) (1+|x|^2)
\end{aligned}
\right)
\Big(\frac2{1+|x|^2}\Big)^{\frac{n-1}2}.
\end{align*}
We now compute $\Delta V(x,y)$ in terms of $v$. Clearly,
\[
\Delta V = \Delta (v \circ B) \Phi^{\frac{n-5}2} + 2 \nabla (v \circ B ) \cdot \nabla \Phi^{\frac{n-5}2} + v( B ) \Delta \Phi^{\frac{n-5}2}.
\]
Again by Lemma \ref{lem:NablaDelta}, there holds
\[
\nabla \Phi^{\frac{n-5}2} =-\frac{n-5}2 ( x, 1+y ) \Phi^\frac{n-3}2 ,\quad \Delta \Phi^{\frac{n-5}2} =-2(n-5) \Phi^\frac{n-3}2.
\]
Thus, as in \eqref{eqO4-DeltaU}, we have just computed
\begin{align}\label{eqO6-DeltaU}
\begin{aligned}
\Delta V = &\big[ \Delta v(B ) - (n-1) \langle \nabla v(B ), (x,-1-y) \rangle \big] \Phi^{\frac{n-1}2}\\
&+ (n-5) \langle \nabla v(B ), (x, -1 -y ) \rangle \Phi^{\frac{n-1}2}   - 2(n-5) v (B ) \Phi^{\frac{n-3}2}\\
=& \Delta v(B ) \Phi^{\frac{n-1}2} - 4 \langle \nabla v (B ) , (x,-1-y) \rangle \Phi^{\frac{n-1}2} -2(n-5) v (B ) \Phi^{\frac{n-3}2}.
\end{aligned}
\end{align}
In particular, we obtain from \eqref{eqO6-DeltaU} the following
\[
\frac 14 \Delta V(x,0) = 
\left(
\begin{aligned}
& \frac 14  \Delta v(B(x,0) )  -  \langle \nabla v (B (x,0)) , (x,-1 ) \rangle \\
&   - \frac{n-5}4 v(B(x,0) ) (1+|x|^2)
\end{aligned}
\right) \Big(\frac2{1+|x|^2}\Big)^{\frac{n-1}2}.
\]
Thus
\begin{align*}
\partial_{y}^2 V(x,0) &= \frac14\Delta V(x,0) -
\left(
\begin{aligned}
&\frac14\Delta v(B(x,0))- \partial^2_\nu v (B(x,0))\\
&- (n-3) \partial_\nu v (B(x,0))\\
& -\frac{(n-5)(n-3)}4 v(B(x,0))
\end{aligned}
\right)
\Big(\frac2{1+|x|^2}\Big)^{\frac{n-1}2}.
\end{align*}
Note that $\Delta$ is the Euclidean Laplacian in $\R^{n+1}$, therefore we obtain
\[
\Delta v(B(x,0)) = \partial^2_\nu v (B(x,0)) + n \partial_\nu v (B(x,0)) + \widetilde \Delta v(B(x,0)).
\]
Hence, if $v$ satisfies the boundary condition \eqref{eq:boundarycond6}, namely
\[
\partial_\nu v\bigl|_{\bS^n} = -\frac{n-5}2 f , \quad \partial^2_\nu v\bigl|_{\bS^n} = \frac13 \widetilde{\Delta} f +\frac{(n-5)(n-6)}{6} f,
\]
and because $B(x,0) \in \mathbb S^n$, then we immediately have
\[
\Delta v(B(x,0)) = 4 \partial^2_\nu v (B(x,0)) - (n-3)(n-5) v(B(x,0)).
\]
We now plug in the preceding formula for $\Delta v$ into the formula for $\partial_{y}^2 V$ to get
\[
\partial_{y}^2 V(x,0) = \frac14 \Delta V(x,0) = \frac 14 \Delta_x V(x,0) + \frac 14 \partial_y^2 V(x,0),
\]
which, again, is coincident with the remaining part of \eqref{eq:Sobolevtrace6BoundaryConditions}. 

From this finding and in view of \eqref{eq:varproblem6PDE}, to obtain Sobolev trace inequality of order six, we could apply the trace inequality \eqref{eq:TraceSpaceOrder6}$_{1/3}$ for $V$. In other words, the desired trace inequality on $\mathbb B^{n+1}$ could be obtained from the transformed trace inequality on $\R _+^{n+1}$ as before. However, it does seem to us that the analysis of this approach is rather involved. Inspired by the second approach based on spherical harmonics for proving the trace inequality of order four on $\mathbb S^n$, we adopt this approach to prove the trace inequality of order six on $\mathbb S^n$ and this is the content of the next subsection.


\subsection{Sharp Beckner type inequality of order six on $\bB^{n+1}$}
\label{subsec-6Beckner}

Let $f \in C^\infty(\bS^n)$ and let $u$ be a triharmonic extension of $f$ in $\bB^{n+1}$ satisfying certain boundary conditions as in \eqref{eq:4NeumannCondition}, namely
\begin{equation}\label{eq:6BecknerRequire}
\partial_\nu u =  -\frac{n-5}2 f, \quad \partial^2_\nu u =  \frac13 \widetilde{\Delta} f +\frac{(n-5)(n-6)}{6} f,
\end{equation}
on $\mathbb S^n$. As in the previous section, we shall work with the spherical harmonic decomposition
\[
f (\omega)= \sum_{k=0}^\infty Y_k(\omega ),
\]
where $\omega = x/|x|$. As always we decompose $u$ to get 
\[
u (z) = \sum_{k=0}^\infty f_k(r) Y_k(\omega ).
\] 
Hence, it follows from \eqref{eq:6BecknerRequire} that the coefficients $f_k$ satisfy
\[
L_k^3 f_k = 0
\]
for any $r\in [0,1)$ and definitely $f_k(1) =1$. Solving the above differential equation gives
\[
f_k (r) = c_1(k) r^k + c_2(k) r^{k+2} + c_3(k) r^{k+4} 
\]
for some constants $c_i(k)$ with $i=1,2,3$ to be determined. In fact, these constants can be computed explicitly by using the boundary conditions in \eqref{eq:6BecknerRequire} as we shall do. Indeed, the condition $u = f$ on $\mathbb S^n$ implies that
\[
c_1 (k) + c_2(k) +c_3(k) = 1
\]
for all $k \geqslant 1$. Now the condition $\partial_\nu u = -((n-5)/2)f$ on $\mathbb S^n$ tells us that
\[
k c_1 (k) + (k+2) c_2(k) + (k+4) c_3(k) =-\frac{n-5}2
\]
while the condition $\partial^2_\nu u = (1/3) \widetilde{\Delta} f +((n-5)(n-6)/6) f$ on $\mathbb S^n$ gives
\[\begin{split}
k (k-1) c_1 (k) + (k+2)(k+1) c_2(k) &+ (k+4)(k+3) c_3(k) \\
& =-\frac{k(n-1+k)}3 + \frac{(n-5)(n-6)}6.
\end{split}\]
Putting these facts together, we compute to get
\[
\left\{
\begin{aligned}
c_1 (k) =&\frac{(n+1+2k)(n+3+2k)}{48} , \\
c_2(k) =& -\frac{(n-5+2k)(n+3+2k)}{24} ,\\
c_3 (k) =& \frac{(n-5+2k)(n-3+2k)}{48} .
\end{aligned}
\right.
\]
Now we use integration by parts and the triharmonicity of $u$ to get
\[
\begin{aligned}
\int_{\bB^{n+1}} |\nabla \Delta u|^2 dz =& - \int_{\bB^{n+1}} \Delta u \Delta^2u dz  + \int_{\mathbb S^n}  \Delta u \partial_\nu (\Delta u)\\
= &-\int_{\bB^{n+1}} u \Delta^3u dz - \int_{\mathbb S^n}  \Delta ^2 u \partial_\nu u\\
& + \int_{\mathbb S^n}  u \partial_\nu (\Delta^2 u) 
+ \int_{\mathbb S^n}  \Delta u \partial_\nu (\Delta u)\\
= & - \int_{\mathbb S^n}  \Delta ^2 u \partial_\nu u + \int_{\mathbb S^n}  u \partial_\nu (\Delta^2 u) + \int_{\mathbb S^n}  \Delta u \partial_\nu (\Delta u) .
\end{aligned}
\]
Recall that
\[
u (z) = \sum_{k=0}^\infty \big[ c_1(k) r^k + c_2(k) r^{k+2} + c_3(k) r^{k+4}  \big] Y_k(\omega ),
\]
which implies that
\[
\partial_\nu u  \big|_{\mathbb S^n} = -\frac{n-5}2 \sum_{k=0}^\infty Y_k (\omega),
\]
that
\[
\Delta u (z) = 2 \sum_{k=0}^\infty \big[ (n+1+2k) c_2(k) + 2(n+3+2k)c_3(k) r^2 \big] r^kY_k(\omega ),
\]
that
\[
\partial_\nu \Delta u \big|_{\mathbb S^n}  = 2 \sum_{k=0}^\infty \big[ k(n+1+2k) c_2(k) + 2(k+2)(n+3+2k)c_3(k) \big] Y_k(\omega ),
\]
that
\[
\Delta^2 u (z) = 8 \sum_{k=0}^\infty (n+1+2k)(n+3+2k)c_3(k) r^kY_k(\omega ),
\]
and that
\[
\partial_\nu \Delta^2 u \big|_{\mathbb S^n}= 8 \sum_{k=0}^\infty k (n+1+2k)(n+3+2k)c_3(k) Y_k(\omega ).
\]
We are now in position to get
\begin{equation}\label{eq:6BecknerNablaDelta^2}
\begin{aligned}
\int_{\bB^{n+1}} &|\nabla \Delta u|^2 dz \\
 = &\frac 1{36} \sum_{k=0}^\infty \int_{\mathbb S^n} (n-5+2k)^2 (12k^2+8kn+n^2-6n+9) (n+3+2k) | Y_k|^2 d\omega.
\end{aligned}
\end{equation}
Now let $0< s < n/5$. We define the function $F$ on $\R^n$ by
\[
F(x) = f(\mathcal S(x)) J_{\mathcal S}(x)^{\frac{n-5s}{2n}}.
\]
Then we have
\[
\int_{\bS^n} |f|^{\frac{2n}{n-5s}} d\omega = \int_{\R^n} |F|^{\frac{2n}{n-5s}} dx
\]
and as in \eqref{eq:GN} we still have
\[
\int_{\R^n} F(x) (-\Delta)^{5s/2} F(x) dx = \sum_{k=0}^\infty\frac{\Gamma(k+n/2+5s/2)}{\Gamma(k+n/2-5s/2)} \int_{\bS^n} |Y_k|^2 ( \omega ) d\omega .
\]
We now use the fractional Sobolev inequality \eqref{eq:SobolevSpaceOrderS/2}, the preceding identity, and \eqref{eq:6BecknerNablaDelta^2} to get
\begin{equation}\label{eq:6BecknerTypeHarmonicNoWeight}
\begin{split}
\frac 83 \frac{\Gamma(\frac{n+5s}2)}{\Gamma(\frac{n-5s}2)} & \omega_n^{5s/n}  \Big(\int_{\bS^n} |f|^{\frac{2n}{n-5s}} d\omega \Big)^{\frac{n-5s}n}\\
\leqslant & \frac 83 \sum_{k=0}^\infty           \frac{\Gamma(k+n/2+5s/2)}{\Gamma(k+n/2-5s/2)}         \int_{\bS^n} |Y_k|^2  d\omega \\
=& \int_{\bB^{n+1}} |\nabla \Delta u|^2 dz + \sum_{k=0}^\infty 
\left( 
\begin{aligned}
&\frac 83 \frac{\Gamma(k+n/2+5s/2)}{\Gamma(k+n/2-5s/2)} \\
&-\frac 1{36} (n-5+2k)^2  (n+3+2k)\\
&\times (12k^2+8kn+n^2-6n+9)
\end{aligned}
\right)  \int_{\bS^n} |Y_k|^2  d\omega .
\end{split}
\end{equation}
Thus, we are in position to state the following sharp Beckner type inequality of order six on $\mathbb S^n$.

\begin{theorem}\label{thm6BecknerTypeNoWeight}
Let $n\geqslant 5$ and $0< s < n/5$. Let $f\in C^\infty(\bS^n)$ and $v$ be a smooth extension of $f$ to the unit ball $\bB^{n+1}$ satisfying the boundary conditions
\[
\left\{
\begin{split}
\partial_\nu v \big|_{\bS^n}  =& -\frac{n-5}2 f, \\
\partial^2_\nu v\big|_{\bS^n}  =&  \frac13 \widetilde{\Delta} f +\frac{(n-5)(n-6)}{6} f .
\end{split}
\right.
\] 
Suppose that $f$ has a decomposition on spherical harmonics as $f =\sum_{k=0}^\infty Y_k(\omega )$. Then the following inequality holds
\begin{equation}\label{eq:6BecknerTypeNoWeight}
\begin{split}
\frac 83 \frac{\Gamma(\frac{n+5s}2)}{\Gamma(\frac{n-5s}2)} \omega_n^{5s/n} &  \Big(\int_{\bS^n} |f|^{\frac{2n}{n-5s}} d\omega \Big)^{\frac{n-5s}n}
\leqslant \int_{\bB^{n+1}} |\nabla \Delta v|^2 dz\\
&  + \sum_{k=0}^\infty 
\left( 
\begin{aligned}
&\frac 83 \frac{\Gamma(k+n/2+5s/2)}{\Gamma(k+n/2-5s/2)} \\
&-\frac 1{36} (n-5+2k)^2  (n+3+2k)\\
&\times (12k^2+8kn+n^2-6n+9)
\end{aligned}
\right)  \int_{\bS^n} |Y_k|^2  d\omega .
\end{split}
\end{equation}
Moreover, equality in \eqref{eq:6BecknerTypeNoWeight} holds if, and only if, $v$ is a triharmonic extension of a function of the form
\[
c|1-\langle z_0, \xi \rangle |^{-(n-5s)/2},
\]
where $c>0$ is a constant, $\xi \in \mathbb S^n$, and $z_0$ is some fixed point in the interior of $\mathbb B^{n+1}$.
\end{theorem}

\begin{proof}
Let $f \in C^\infty(\bS^n)$, it is well known that the minimizing problem \eqref{eq:varproblem6}
\[
\inf_w \int_{\bB^{n+1}} |\nabla \Delta w|^2  dz
\]
under the constraint $w \in \mathscr W_f$ with $\mathscr W_f$ is given in \eqref{eq:varproblem6Set} is attained by some function $u$. Clearly, $\Delta^3 u(z) =0$ in $\bB^{n+1}$ and $u \equiv f$, $\partial_\nu u = -((n-5)/2)f$, and $\partial^2_\nu u=(1/3) \widetilde{\Delta} f +((n-5)(n-6)/6) f$ on $\bS^n$. The inequality \eqref{eq:6BecknerTypeNoWeight} then follows  from \eqref{eq:6BecknerTypeHarmonicNoWeight} and the fact that
\[
\int_{\bB^{n+1}}  |\nabla \Delta u|^2 dz  \leqslant \int_{\bB^{n+1}}  |\nabla \Delta v|^2 dz
\]
since $u$ is a minimizer. Finally, the equality case in \eqref{eq:6BecknerTypeNoWeight} can be obtained by following the argument used in the proof of Theorem \ref{thm2BecknerTypeWeight}.
\end{proof}

It is worth noticing our choice of the coefficient $(8/3) \Gamma(\frac{n+5s}2)/\Gamma(\frac{n-5s}2) \omega_n^{5s/n}$ appearing on the left hand side of \eqref{eq:6BecknerTypeNoWeight} comes from the similar coefficient of the left hand side of \eqref{eq:SobolevTraceSpace6Expected}.


\subsection{Sharp Sobolev trace inequality of order six on $\bB^{n+1}$: Proof of Theorem \ref{thmTraceOrder6}}
\label{subsec-6SobolevBall}

Now we use the Beckner type inequality of order six \eqref{eq:6BecknerTypeNoWeight} to derive the sharp Sobolev trace inequality of order six \eqref{eq:TraceOrder6} on $\mathbb S^n$. To this purpose, it is necessary to compute $\int_{\mathbb S^n}  (\widetilde \Delta f)^2 d\omega$ and $\int_{\mathbb S^n}  |\widetilde \nabla f|^2 d\omega$ in terms of spherical harmonics. Since $f (\omega)= \sum_{k=0}^\infty Y_k(\omega )$, we obtain
\[
\widetilde \Delta f (\omega)= - \sum_{k=0}^\infty k(n-1+k) Y_k(\omega ),
\]
which then implies that
\begin{equation}\label{eq:6Delta^2f}
\int_{\mathbb S^n}  (\widetilde \Delta f)^2 d\omega =  \sum_{k=0}^\infty k^2(n-1+k)^2       \int_{\bS^n} |Y_k|^2 ( \omega ) d\omega.
\end{equation}
In a similar way, we compute
\begin{equation}\label{eq:6Nabla^2f}
\int_{\mathbb S^n}  |\widetilde \nabla f|^2 d\omega = - \int_{\mathbb S^n}  f \widetilde \Delta f  d\omega =  \sum_{k=0}^\infty k(n-1+k)        \int_{\bS^n} |Y_k|^2 ( \omega ) d\omega.
\end{equation}
We now let $n \geqslant 6$ and select $s=1$ in \eqref{eq:6BecknerTypeNoWeight} to get
\begin{equation}\label{eq:6BecknerTypeNoWeightS=1}
\begin{split}
\frac 83 \frac{\Gamma(\frac{n+5}2)}{\Gamma(\frac{n-5}2)} & \omega_n^{5/n}  \Big(\int_{\bS^n} |f|^{\frac{2n}{n-5}} d\omega \Big)^{\frac{n-5}n}
\leqslant \int_{\bB^{n+1}} |\nabla \Delta v|^2 dz\\
&  + \sum_{k=0}^\infty 
\left( 
\begin{aligned}
& \frac 83 \frac{\Gamma(k+n/2+5/2)}{\Gamma(k+n/2-5/2)} \\
&-\frac 1{36} (n-5+2k)^2  (n+3+2k)\\
&\times (12k^2+8kn+n^2-6n+9)
\end{aligned}
\right)  \int_{\bS^n} |Y_k|^2 ( \omega ) d\omega .
\end{split}
\end{equation}
When transferring back the above Beckner type inequality into the correct sharp trace inequality \eqref{eq:TraceOrder6}, the right hand side of \eqref{eq:TraceOrder6} must contain all lower order terms $\int_{\bS^n} (\widetilde\Delta f)^2 d\omega$, $\int_{\bS^n} |\widetilde{\nabla} f|^2 d\omega$, and $\int_{\bS^n} |f|^2 d\omega $. Therefore, it is necessary to recast the coefficient of the term $\int_{\bS^n} |Y_k|^2 ( \omega ) d\omega$ in \eqref{eq:6BecknerTypeNoWeightS=1} in such a way that it only consists of the term $k(n-1+k)$. Without using any computer software, tedious computation shows that
\[
\begin{aligned}
\frac 83\frac{\Gamma(k+n/2+5/2)}{\Gamma(k+n/2-5/2)} -\frac 1{36}& (n-5+2k)^2  (n+3+2k) (12k^2+8kn+n^2-6n+9)\\
=&\frac 1{18} \big[4(n+3) k(n-1+k) + (n-3)(n^2+4n-9) \big]\\
& \times \big[4 k(n-1+k) + (n+3)(n-5)\big].
\end{aligned}
\]
This, \eqref{eq:6Delta^2f}, and \eqref{eq:6Nabla^2f} give us the desired inequality \eqref{eq:TraceOrder6}, namely
\[
\begin{aligned}
\frac 83 \frac{\Gamma(\frac{n+5}2)}{\Gamma(\frac{n-5}2)} \omega_n^{5/n} &\Big(\int_{\bS^n} |f|^{\frac{2n}{n-5}}d\omega \Big)^{\frac{n-5}n}\\
 \leqslant &\int_{\bB^{n+1}} |\nabla \Delta v|^2 dx + \frac{8(n+3)}9\int_{\bS^n} (\widetilde\Delta f)^2 d\omega\\
&+ \frac{4(n^3+n^2-21n-9)}9 \int_{\bS^n} |\widetilde{\nabla} f|^2 d\omega + c_n \int_{\bS^n} |f|^2 d\omega 
\end{aligned}
\]
with $c_n=(n-5)(n-3)(n+3)(n^2+4n-9)/18$. Clearly, equality in \eqref{eq:TraceOrder6} holds if, and only if, equality in \eqref{eq:6BecknerTypeNoWeight} with $s = 1$ holds, namely, $v$ is a triharmonic extension of a function of the form
\[
c|1-\langle z_0, \xi \rangle |^{-(n-5)/2},
\]
where $c$ is a constant, $\xi \in \mathbb S^n$, and $z_0$ is some fixed point in the interior of $\mathbb B^{n+1}$, which also satisfies the Neumann boundary conditions. The proof is complete.


\subsection{A Ledebev--Milin type inequality of order six on $\bB^6$}

In the last part of this section, we treat the limiting case $n = 5$. Our aim is to derive a Ledebev--Milin type inequality of order six. To this purpose, we follow the strategy used to obtain Theorem \ref{thmSMOrder4}.

Suppose that $f \in C^\infty(\bS^5)$ with $\int_{\bS^5} f d\omega =0$ and let $v$ be a smooth extension of $f$ in $\bB^6$. As before, we decompose $f$ in terms of spherical harmonics to get
\[
f = \sum_{k=1}^\infty Y_k(\omega )
\]
Note that the function $ 1 + \frac 5 2(1-s) v$ is also a smooth extension of $1 +  \frac 5 2(1-s) f$ in $\mathbb B^6$ and $\omega_5 =  \pi^3$. Therefore, we can apply Theorem \ref{thm6BecknerTypeNoWeight} to get
\[
\begin{split}
\frac 83 \frac{\Gamma(\frac{5+5s}2)}{\Gamma(\frac{5-5s}2)} \Big[ \pi^{3s} &  \Big(\int_{\bS^5} \Big|           1+ \frac 5 2(1-s) f        \Big|^{\frac{2}{1- s}} d\omega \Big)^{1-s} - 1 \Big]\\
\leqslant & \Big( \frac 5 2(1-s) \Big)^2  \int_{\bB^6} |\nabla \Delta v|^2 dz\\
&  +  \Big( \frac 5 2(1-s) \Big)^2 \sum_{k=1}^\infty 
\left( 
\begin{aligned}
& \frac 83 \frac{\Gamma(k+(5+5s)/2)}{\Gamma(k+(5-5s)/2)} \\
&-\frac 8{9} k^2  (k+4)  (3k^2+10k+1)
\end{aligned}
\right)  
\int_{\bS^5} |Y_k|^2  d\omega 
\end{split}
\]
which implies
\begin{align*}
\frac 25  \frac{\Gamma(\frac{5+5s}2)}{\Gamma(\frac{5-5s}2)}& \Big[ \Big(\frac 1{\pi^3} \int_{\bS^5} \Big(1 + \frac{5(1-s)}2  f\Big)^{\frac{2}{1-s}} d\omega \Big)^{1-s}  - 1 \Big] \\
 \leqslant & \frac{ 15  (1-s)^2}{16 \pi^3}
 \left[ \int_{\bB^6}|\nabla \Delta v|^2 dz  
  +  \sum_{k=1}^\infty 
\left( 
\begin{aligned}
& \frac 83 \frac{\Gamma(k+(5+5s)/2)}{\Gamma(k+(5-5s)/2)} \\
&-\frac 8{9} k^2  (k+4)  (3k^2+10k+1)
\end{aligned}
\right)  
\int_{\bS^5}|Y_k|^2 d\omega 
\right].
\end{align*}
Dividing both sides by $(1-s)^2$ and making use of $\Gamma(\frac{7-5s}2) = (5/2)(1-s)\Gamma(\frac{5-5s}2)$ to get
\begin{align*}
 \frac{\Gamma(\frac{5+5s}2)}{\Gamma(\frac{7-5s}2)} \frac 1{1-s} & \Big[ \Big(\frac 1{\pi^3} \int_{\bS^5} \Big(1 +\frac{5(1-s)}2 f\Big)^{\frac{2}{1-s}} d\omega \Big)^{1-s}  - 1 \Big] \\
\leqslant &\frac {15}{16\pi^3}
\left[ \int_{\bB^6}|\nabla \Delta v|^2 dz  
  +  \sum_{k=1}^\infty 
\left( 
\begin{aligned}
& \frac 83 \frac{\Gamma(k+(5+5s)/2)}{\Gamma(k+(5-5s)/2)} \\
&-\frac 8{9} k^2  (k+4)  (3k^2+10k+1)
\end{aligned}
\right)  
\int_{\bS^5}|Y_k|^2 d\omega 
\right].
\end{align*}
Note that 
\[\begin{split}
\lim_{s \nearrow 1 } \Big( \frac 83 \frac{\Gamma(k+\frac{5+5s}2)}{\Gamma(k+ \frac{5-5s}2)} -& \frac 8{9} k^2  (k+4)  (3k^2+10k+1)  \Big) \\
= &\frac 1{18} \big[32 k(k+4) + 72 \big] \big[4 k(k+4)\big].
\end{split}\]
for any $k \geqslant 1$. Hence letting $s \nearrow 1$, we obtain
\[\begin{split}
24 \log \Big(\frac1{\pi^3} \int_{\bS^5} e^{5f} d\omega \Big) \leqslant & \frac {15}{16\pi^3}  \int_{\bB^6} |\nabla \Delta v|^2 dx  \\
&+ 
\frac {15}{16\pi^3} \Big[\frac{64}9 \int_{\bS^5} |\widetilde\Delta f|^2 d\omega   +16 \int_{\bS^5} |\widetilde \nabla f|^2 d\omega  \Big]
\end{split}\]
for any smooth function $f$ with $\int_{\bS^5} f d\omega =0$. For general function $f$, we apply the previous inequality for $f- \pi^{-3} \int_{\bS^5} f d\omega$ to get the following theorem.

\begin{theorem}[Lebedev--Milin inequality of order six]\label{thmSMOrder6}
Let $f\in C^\infty(\bS^5)$ and suppose that $v$ is a smooth extension of $f$ to the unit ball $\bB^6$. If $v$ satisfies the boundary conditions
\[
\left\{
\begin{split}
\partial_\nu v \big|_{\bS^5} =& 0,\\
\partial_\nu^2 v  \big|_{\bS^5} = &\frac 13 \widetilde \Delta f,
 \end{split}
\right.
\]
then we have the following sharp trace inequality
\begin{equation}\label{eqSMorder6}
\begin{split}
\log \lt(\frac1{\pi^3} \int_{\bS^5} e^{5 f} d\omega \rt) \leqslant & \frac 5{128\pi^3} \int_{\bB^6} |\nabla \Delta v|^2 dx + \frac 5{18\pi^3} \int_{\bS^5} |\widetilde\Delta f|^2 d\omega  \\
&+ \frac 5{8\pi^3} \int_{\bS^5} |\widetilde\nabla f|^2 d\omega + \frac 5{\pi^3} \int_{\bS^5} f d\omega .
\end{split}
\end{equation}
Moreover, equality in \eqref{eqSMorder6} holds if, and only if, $v$ is a biharmonic extension of a function of the form
\[
c - \log |1 -  \langle z_0, \xi \rangle|,
\]
where $c>0$ is a constant, $\xi \in \bS^5$, $z_0$ is some fixed point in the interior of $\mathbb B^6$, and $v$ fulfills the boundary conditions $\partial_\nu v = 0$ and $\partial^2_\nu v = (1/3)\widetilde \Delta f$.
\end{theorem}

We note that the coefficient of the two terms in the middle of the right hand side of \eqref{eqSMorder6} is a multiple of $c_5^{(1)}$ and $c_5^{(2)}$ given in \eqref{eq:CoefficientC_n}, respectively. Clearly, Inequality \eqref{eqSMorder6} can be rewritten as follows
\[\begin{split}
\log \Big(\frac1{\pi^3} \int_{\bS^5} e^{5 (f - \overline f)} d\omega \Big) \leqslant & \frac 5{128\pi^3} \int_{\bB^6} |\nabla \Delta v|^2 dx\\
& + \frac 5{18\pi^3} \int_{\bS^5} |\widetilde\Delta f|^2 d\omega  + \frac 5{8\pi^3} \int_{\bS^5} |\widetilde\nabla f|^2 d\omega 
\end{split}\]
where $\overline f$ is the average of $f$, which is $\pi^{-3} \int_{\bS^5} f d\omega$.


\section{Sobolev trace inequality of order eight and beyond}
\label{sec-Order8+}

In the final part of the paper, we would like to emphasize that sharp Sobolev trace inequalities of lower order can be easily derived using our approach. As demonstrated in Section \ref{sec-Order6} for the trace inequality of order six, we present in this section sharp trace inequalities of order eight on $\bB^{n+1}$ and on $\R_+^{n+1}$.

The strategy is as follows. At the beginning, we have to look for a sharp trace inequality of order eight on $\R_+^{n+1}$. In view of the boundary conditions in \eqref{eq:6SpaceNeumannCondition}$_\lambda$, there is an extra boundary condition involving the third order partial derivative $\partial_y^3 U(x,0)$. Our choice for such a boundary condition again comes from \cite[Theorem 3.3]{rayyang2013}. Our sharp trace inequality on $\R_+^{n+1}$ reads as follows.

\begin{theorem}[Sobolev trace inequality of order eight on $\R_+^{n+1}$]\label{thmTraceSpaceOrder8}
Let $U\in W^{4,2}(\overline{\R_+^{n+1}})$ be satisfied the Neumann boundary condition
\begin{subequations}\label{eq:8SpaceNeumannCondition}
\begin{align}
\partial_y U(x,0)=0, \quad \partial^2_y U(x,0)=\lambda \Delta_x U(x,0), \quad \partial^3_y U(x,0)=0.
\tag*{\eqref{eq:8SpaceNeumannCondition}$_\lambda$}
\end{align}
\end{subequations}
Then we have the sharp trace inequality
\begin{subequations}\label{eq:TraceSpaceOrder8}
\begin{align}
(20\lambda^2-8\lambda +4) \frac{\Gamma(\frac{n+7}2)}{\Gamma(\frac{n-7}2)} \omega_n^{7/n} \Big(\int_{\R^n} |U(x,0)|^{\frac{2n}{n-7}} dx\Big)^{\frac{n-7}n} \leqslant \int_{\R^{n+1}_+} |\nabla \Delta U(x,y)|^2 dx dy.
\tag*{\eqref{eq:TraceSpaceOrder8}$_\lambda$}
\end{align}
\end{subequations}
Moreover, equality in \eqref{eq:TraceSpaceOrder8}$_\lambda$ holds if, and only if, $U$ is a quadharmonic extension of a function of the form
\[
c \big( 1 +  |x - x_0|^2 \big)^{-(n-7)/2},
\]
where $c>0$ is a constant, $x \in \R^n$, $x_0$ is some fixed point in $\R^n$, and $U$ fulfills the boundary condition \eqref{eq:8SpaceNeumannCondition}$_\lambda$.
\end{theorem}

The proof of Theorem \ref{thmTraceSpaceOrder8} is almost similar to that of Theorem \ref{thmTraceSpaceOrder6}; hence we omit it. Next we want to determine $\lambda$. The way we look for $\lambda$ is to solve the equation $20\lambda^2-8\lambda +4=16/5$. The constant $16/5$ comes from the constant $c_4$ where $c_\alpha$ is already given in \eqref{eq:CAlpha}. Via the conformal transformation $B$, we need to determine appropriate boundary conditions from \eqref{eq:8SpaceNeumannCondition}$_{1/5}$.

Our sharp trace inequality on $\bB^{n+1}$ reads as follows.

\begin{theorem}[Sobolev trace inequality of order eight]\label{thmTraceOrder8}
Let $f\in C^\infty(\bS^n)$ with $n> 7$ and suppose $v$ is a smooth extension of $f$ in the unit ball $\bB^{n+1}$, which also satisfies the boundary conditions
\begin{equation}\label{eq:boundarycond8}
\left\{
\begin{split}
\partial_\nu v \big|_{\bS^n} =  &-\frac{n-7}2 f,\\ 
\partial^2_\nu v\big|_{\bS^n} =& \frac15 \widetilde \Delta f + \frac{(n-7)(2n-15)}{10} f, \\
\partial^3_\nu v \big|_{\bS^n} =& -\frac {3(n-5)}{10} \widetilde \Delta f - \frac{(n-5)(n-7)(n-15)}{20} f.
\end{split}
\right.
\end{equation}
Then the following inequality holds
\begin{equation}\label{eq:TraceOrder8}
\begin{aligned}
\frac {16}5 \frac{\Gamma(\frac{n+7}2)}{\Gamma(\frac{n-7}2)} \omega_n^{7/n} &\Big(\int_{\bS^n} |f|^{\frac{2n}{n-7}}d\omega \Big)^{\frac{n-7}n}\\
 \leqslant &\int_{\bB^{n+1}} | \Delta^2 v|^2 dx 
 + d_n^{(1)} \int_{\bS^n} |\widetilde\nabla\widetilde\Delta f|^2 d\omega 
 + d_n^{(2)} \int_{\bS^n} (\widetilde\Delta f)^2 d\omega\\
&+d_n^{(3)} \int_{\bS^n} |\widetilde{\nabla} f|^2 d\omega + d_n^{(4)} \int_{\bS^n} |f|^2 d\omega 
\end{aligned}
\end{equation}
with
\begin{equation}\label{eq:CoefficientD_n}
\left\{
\begin{split}
d_n^{(1)}=  &\frac{8}{25}(5n + 1),\\ 
d_n^{(2)} =& \frac{1}{25}(30n^3-54n^2-542n-490), \\
d_n^{(3)} =& \frac 1{50} (15n^5-57n^4-482n^3+582n^2+4325n+10485) , \\
d_n^{(4)}=& \frac 1{200} (n -7)(n +5)(5{n^5} - 19{n^4} - 74{n^3} - 26{n^2} + 615n + 3135).
\end{split}
\right.
\end{equation}
Moreover, equality in \eqref{eq:TraceOrder8} holds if, and only if, $v$ is a quadharmonic extension of a function of the form
\[
f_{z_0} (\xi) = c|1-\langle z_0, \xi \rangle |^{-(n-7)/2},
\]
where $c>0$ is a constant, $\xi \in \mathbb S^n$, $z_0$ is some fixed point in the interior of $\mathbb B^{n+1}$, and $v$ fulfills the boundary condition \eqref{eq:boundarycond8}.
\end{theorem}

To prove Theorem \ref{thmTraceOrder8} we first establish a Beckner type inequality in the same fashion of Theorem \ref{thm6BecknerTypeNoWeight}. To achieve that goal, we note that there is an extra work to consider the term $\int_{\bS^n} |\widetilde\nabla\widetilde\Delta f|^2 d\omega$. Using integration by parts, there holds
\[
\int_{\bS^n} |\widetilde\nabla\widetilde\Delta f|^2 d\omega = -\int_{\bS^n}  (\widetilde\Delta f ) \, (\widetilde\Delta^2 f) d\omega.
\]
While the spherical harmonic expansion involving $\widetilde\Delta f$ is already computed in \eqref{eq:SHDDeltaf}, the spherical harmonic expansion involving $\widetilde\Delta^2 f$ is nothing but
\[
\widetilde \Delta^2 Y_k = k^2(n+k-1)^2Y_k.
\]
This is because $\widetilde \Delta Y_k = - k(n+k-1)Y_k$. Thus,
\[
\int_{\bS^n} |\widetilde\nabla\widetilde\Delta f|^2 d\omega =  \sum_{k=0}^\infty k^3 (n-1+k)^3       \int_{\bS^n} |Y_k|^2 ( \omega ) d\omega.
\]
Putting all these information together, we eventually obtain an estimate similar to \eqref{eq:6BecknerTypeNoWeightS=1}, however, there are terms with higher order derivatives. Furthermore, the coefficient of $\int_{\bS^n} |Y_k|^2 d\omega$ becomes
\[
\begin{aligned}
\frac {16}5 \frac{\Gamma(k+n/2+7/2)}{\Gamma(k+n/2-7/2)} &- \frac 1{100}(n+5+2k)(n-7+2k)\\
&\times
\left(
\begin{aligned}
&80k^5+160k^4n+120k^3n^2+40k^2n^3+5kn^4\\
&-208k^4-336k^3n-192k^2n^2-44kn^3-3n^4\\
&-184k^3-136k^2n+2kn^2+12n^3+912k^2\\
&+732kn+138n^2-375k-285n-1680\\
\end{aligned}
\right).
\end{aligned}
\]
Finally, it remains to recast the above coefficient in terms of powers of $k(n+k-1)$. A detailed proof of Theorem \ref{thmTraceOrder8} is left for interested readers. 

Up to this position, one can ask if there is a sharp Sobolev trace inequality of any order on $\bB^{n+1}$. In principle, our approach is easy to implement, but boundary conditions of higher orders, like \eqref{eq:8SpaceNeumannCondition}$_\lambda$ and \eqref{eq:boundarycond8} for order eight, are not easy to derive. Jeffrey Case suggests us to compute boundary conditions from the paper of Graham and Zworski \cite{GZ}.

Finally, we put here a Lebedev--Milin inequality of order eight, whose proof is also left for interested readers.

\begin{theorem}[Lebedev--Milin inequality of order eight]\label{thmSMOrder8}
Let $f\in C^\infty(\bS^7)$ and suppose that $v$ is a smooth extension of $f$ to the unit ball $\bB^8$. If $v$ satisfies the boundary conditions
\[
\left\{
\begin{split}
\partial_\nu v \big|_{\bS^7} =  &0,\\ 
\partial^2_\nu v\big|_{\bS^7} =& \frac15 \widetilde \Delta f , \\
\partial^3_\nu v \big|_{\bS^7} =& -\frac {3(n-5)}{10} \widetilde \Delta f .
\end{split}
\right.
\]
Then we have the following sharp trace inequality
\begin{equation}\label{eqSMorder8}
\begin{split}
\log \lt(\frac 3{\pi^4} \int_{\bS^7} e^{7 f} d\omega \rt)
 \leqslant & \frac 7{1728\pi^4} \int_{\bB^8} | \Delta^2 v|^2 dx 
+\frac {7}{150\pi^4}  \int_{\bS^7} |\widetilde\nabla\widetilde\Delta f|^2 d\omega\\
& +\frac {49}{90\pi^4} \int_{\bS^7} |\widetilde\Delta f|^2 d\omega 
+  \frac {49}{90\pi^4} \int_{\bS^7} |\widetilde\nabla f|^2 d\omega + \frac {21}{\pi^4} \int_{\bS^7} f d\omega .
\end{split}
\end{equation}
Moreover, equality in \eqref{eqSMorder8} holds if, and only if, $v$ is a quadharmonic extension of a function of the form
\[
c - \log |1 -  \langle z_0, \xi \rangle|,
\]
where $c>0$ is a constant, $\xi \in \bS^7$, $z_0$ is some fixed point in the interior of $\mathbb B^8$, and $v$ fulfills the boundary conditions $\partial_\nu v = 0$, $\partial^2_\nu v = (1/5)\widetilde \Delta f$, and $\partial^3_\nu v = -(3(n-5)/10)\widetilde \Delta f$.
\end{theorem}

We note that, and as always, the coefficient of the three terms in the middle of the right hand side of \eqref{eqSMorder8} is a multiple of $d_7^{(1)}$, $d_7^{(2)}$, and $d_7^{(3)}$ given in \eqref{eq:CoefficientD_n}, respectively. Clearly, Inequality \eqref{eqSMorder8} can also be rewritten as follows
\[
\begin{split}
\log \Big(\frac 3{\pi^4} \int_{\bS^7} e^{7  (f - \overline f  )} d\omega \Big)
 \leqslant & \frac 7{1728\pi^4} \int_{\bB^8} | \Delta^2 v|^2 dx 
+\frac {7}{150\pi^4}  \int_{\bS^7} |\widetilde\nabla\widetilde\Delta f|^2 d\omega\\
& +\frac {49}{90\pi^4} \int_{\bS^7} |\widetilde\Delta f|^2 d\omega 
+  \frac {49}{90\pi^4} \int_{\bS^7} |\widetilde\nabla f|^2 d\omega   ,
\end{split}
\]
where $\overline f$ is the average of $f$ over $\bS^7$, which is $(3/\pi^4)\int_{\bS^7} f d\omega$.

\section*{Acknowledgments}

An important portion of the paper was done during the first author's visit to ICTP from June to August in 2018, where an excellent working condition is greatly acknowledged. Thanks also go to Quoc Hung Nguyen for useful discussions on fractional Laplacian during the first author's short visit to SNS Pisa in July. His advice definitely improved the paper style. The first author also benefited from the Vietnam Institute for Advanced Study in Mathematics (VIASM) during his visits in 2017 and 2018. The second author was supported by CIMI's postdoctoral research fellowship. Last but not least, the authors would like to thank Jeffrey Case and Jingang Xiong for valuable comments and suggestions which substantially improved the exposition of the article.

\appendix

\section{Proof of Proposition \ref{limits}}
\label{apd-Limits}

This appendix is devoted to a proof of Proposition \ref{limits}. To proceed, we first observe that $\alpha_k$ and $\beta_k$ are solutions to
\[
X^2 - \big( \frac{n-1}2 + k + b \big) X +\frac{kb}2 = 0;
\]
thanks to $k+b > 0$. If we denote by $f(X)$ the left hand side of the preceding equation, then it is not hard to verify that $f(-1)>0$ since $b>-1$ and $f(b-1)>0$ since $b<1$. This fact and $(n-1)/2+k+b > 0$ imply that
\[
\alpha_k+1, \beta_k +1, \alpha_k+ 1-b, \beta_k + 1 -b >0.
\]
Obviously, $A(0,k) =k$ for any $k$. In our proof below, we will use the Gaussian hypergeometric functions to describe the functions $f_k$; see \cite[Chapter $15$]{AS64}. Resolving \eqref{eq2f_k} for $f_k$ gives
\[
r^2(1-r^2) f_k''(r) + (n -(n+2b)r^2)r f_k'(r) - c_k (1-r^2) f_k(r) =0.
\]
Using the following variable change $f_k(r) = g_k(r^2)$, it is easy to verify that $g_k$ solves
\[
t^2(1-t) g_k''(t) +  \Big(\frac{n+1}2 -\frac{n+1+2b}2t\Big) t g_k'(t) - \frac{c_k}4 (1- t) g_k(t) = 0.
\]
We now further change $g_k(t) = t^{k/2} h_k(t)$ to get the following equation
\[
t(1-t) h_k''(t) + \Big(\frac{n+1+2k}2 -\frac{n+1+2k+2b}2t\Big) h_k'(t) -\frac{kb}2 h_k(t) =0.
\]
Recall that 
\[
\alpha_k + \beta_k = \frac{n+2k-1 + 2b}2,\quad \alpha_k \beta_k = \frac{kb}2.
\]
Denote 
\[
\gamma_k = \frac{n+2k+1}2 = \alpha_k+\beta_k + 1 -b.
\] 
Hence, solving the hypergeometric differential equation satisfied by $h_k$ gives
\[
h_k(t) = C_1F(\alpha_k, \beta_k; \gamma_k;t) + C_2 t^{1 -\gamma_k}F(\alpha_k-\gamma_k+1,\beta_k -\gamma_k +1; 2 -\gamma_k; t),
\]
for some constants $C_1$ and $C_2$. Note that in the preceding formula, $F$ is the Gaussian hypergeometric function; see \cite[Section 15.5]{AS64}. Since $b <1$, we deduce that $\gamma_k > \alpha_k + \beta_k$. Now, replacing $h_k(r^2)$ by $r^{-k} f_k (r)$ gives
\[
f_k(r) = C_1 r^kF(\alpha_k, \beta_k;\gamma_k;r^2) + C_2 r^{1-n-k} F(\alpha_k-\gamma_k +1,\beta_k -\gamma_k +1;2-\gamma_k; r^2).
\]
Keep in mind that $f_k(r) = O(r^k)$ when $r$ is near $0$. From this it is immediate to see that $C_2 =0$, which then implies
\[
f_k(r) = C_1 r^kF(\alpha_k, \beta_k;\gamma_k;r^2) . 
\]
Now the condition $f_k(1) =1$ tells us that $C_1 =1/{F(\alpha_k,\beta_k;\gamma_k;1)} $ and in terms of the Gamma function, we obtain
\[
C_1 =  \frac{\Gamma(\gamma_k -\alpha_k)\Gamma(\gamma_k-\beta_k)}{\Gamma(\gamma_k) \Gamma(\gamma_k -\alpha_k -\beta_k)};
\] 
see \cite[15.1.20]{AS64}. Using the differential formula \cite[15.2.1]{AS64} and , we get
\begin{equation}\label{eq:bduong}
\frac{f_k'(r)}{C_1} = k r^{k-1} F(\alpha_k, \beta_k;\gamma_k; r^2) + 2\frac{\alpha_k\beta_k}{\gamma_k} r^{k+1} F(\alpha_k+1,\beta_k+1;\gamma_k +1;r^2) .
\end{equation}
Depending on the size of $b$, we have the following three cases.

\medskip\noindent\textbf{The case $b>0$}. In this case, we apply the linear transformation formula \cite[15.3.3]{AS64} to further decompose $f_k'$ from \eqref{eq:bduong} as follows 
\[
\begin{aligned}
\frac{f_k'(r)}{C_1} =&k r^{k-1} F(\alpha_k, \beta_k, \gamma_k,r^2) \\
&+ 2\frac{\alpha_k\beta_k}{\gamma_k} r^{k+1} (1-r^2)^{-b} F(\gamma_k -\alpha_k, \gamma_k -\beta_k, \gamma_k +1,r^2),
\end{aligned}
\]
which then implies
\[
\begin{aligned}
 \Big(\frac{1-r^2}2\Big)^b \frac{f_k'(r)}{C_1} =&k  \Big(\frac{1-r^2}2\Big)^b  r^{k-1} F(\alpha_k, \beta_k; \gamma_k; r^2) \\
&+ 2^{-b}\frac{kb}{\gamma_k} r^{k+1}  F(\gamma_k -\alpha_k, \gamma_k -\beta_k; \gamma_k +1; r^2),
\end{aligned}
\]
Keep in mind that $\gamma_k+ 1- (\gamma_k -\alpha_k + \gamma_k -\beta_k) = b >0$, hence $F(\gamma_k-\alpha_k,\gamma_k-\beta_k; \gamma_k+1;1)$ exists; see \cite[15.1.1(a)]{AS64}. Therefore, we can send $r$ to $1$ to get
\begin{align*}
\lim_{r\to 1} \Big(\frac{1-r^2}2\Big)^b f_k'(r) &= 2^{-b} \frac{\Gamma(\gamma_k -\alpha_k)\Gamma(\gamma_k-\beta_k)}{\Gamma(\gamma_k) \Gamma(\gamma_k -\alpha_k -\beta_k)} 
 \frac{kb}{\gamma_k} \frac{\Gamma (\gamma_k + 1)\Gamma (\alpha_k + \beta_k +1 - \gamma_k)}{\Gamma (\alpha_k + 1)\Gamma (\beta_k+1)}\\
&=2^{-b}\frac{\Gamma(1+b)}{\Gamma(1-b)} \frac{\Gamma(\beta_k+1-b)\Gamma(\alpha_k+1-b)}{\Gamma(\alpha_k+1) \Gamma(\beta_k+1)}k =A(b,k)
\end{align*}
as claimed.

\medskip\noindent\textbf{The case $b<0$}. In this case, we decompose $f_k'$ from \eqref{eq:bduong} as follows 
\begin{align}\label{eq:bam}
\begin{aligned}
\frac{f_k'(r)}{C_1} =& r^{k-1} \Big[ k F(\alpha_k, \beta_k; \gamma_k;r^2) +  \frac{kb}{\gamma_k} F(\alpha_k+1,\beta_k+1;\gamma_k +1;r^2) \Big] \\
&-  \frac{kb}{\gamma_k} r^{k-1}(1-r^2) F(\alpha_k+1,\beta_k+1;\gamma_k +1;r^2).
\end{aligned}
\end{align}
In the sequel, we consider the behavior of the first term on the right hand side of \eqref{eq:bam} as $t \nearrow 1$. This is because after multiplying both sides by $(1-r^2)^b$ the second term is negligible as $t \nearrow 1$ due to the term $1-r^2$. We now apply the linear transformation formula \cite[15.3.6]{AS64} to get
\[
\begin{aligned}
k F(\alpha_k, & \beta_k,\gamma_k,t) +  \frac{kb}{\gamma_k} F(\alpha_k+1,\beta_k+1;\gamma_k +1;t) \\
=& k \frac{\Gamma(\gamma_k)\Gamma(1-b)}{\Gamma(\gamma_k -\alpha_k) \Gamma(\gamma_k-\beta_k)} F(\alpha_k, \beta_k; b; 1-t) \\
& +k (1-t)^{1-b}\frac{\Gamma(\gamma_k) \Gamma(b-1)}{\Gamma(\alpha_k)\Gamma(\beta_k)}F(\gamma_k-\alpha_k,\gamma_k-\beta_k; 2-b;1-t) \\
& + \frac{kb}{\gamma_k} \frac{\Gamma(\gamma_k+1)\Gamma(-b)}{\Gamma(\gamma_k -\alpha_k)\Gamma(\gamma_k -\beta_k)} F(\alpha_k+1,\beta_k +1; 1+b;1-t) \\
&  + \frac{kb}{\gamma_k} \frac{\Gamma(\gamma_k+1)\Gamma(b)}{\Gamma(\alpha_k+1)\Gamma(\beta_k+1)} (1-t)^{-b} F(\gamma_k-\alpha_k, \gamma_k -\beta_k; 1-b;1-t).
\end{aligned}
\]
Observe that
\[
 \frac{\Gamma(\gamma_k)\Gamma(1-b)}{\Gamma(\gamma_k -\alpha_k) \Gamma(\gamma_k-\beta_k)} +  \frac{b}{\gamma_k} \frac{\Gamma(\gamma_k+1)\Gamma(-b)}{\Gamma(\gamma_k -\alpha_k)\Gamma(\gamma_k -\beta_k)} =0.
\]
Hence by the definition of the hypergeometric series, we deduce that
\begin{align}\label{eq:hyperfor}
\begin{aligned}
k F(\alpha_k, & \beta_k,\gamma_k,t) +  \frac{kb}{\gamma_k} F(\alpha_k+1,\beta_k+1;\gamma_k +1;t) \\
=& k \frac{\Gamma(\gamma_k)\Gamma(1-b)}{\Gamma(\gamma_k -\alpha_k) \Gamma(\gamma_k-\beta_k)} \sum_{n \geqslant 1} \frac{(\alpha_k)_n (\beta_k)_n }{(b)_n} \frac{(1-t)^n}{n!} \\
& +k (1-t)^{1-b}\frac{\Gamma(\gamma_k) \Gamma(b-1)}{\Gamma(\alpha_k)\Gamma(\beta_k)}F(\gamma_k-\alpha_k,\gamma_k-\beta_k; 2-b;1-t) \\
& + \frac{kb}{\gamma_k} \frac{\Gamma(\gamma_k+1)\Gamma(-b)}{\Gamma(\gamma_k -\alpha_k)\Gamma(\gamma_k -\beta_k)}  \sum_{n \geqslant 1} \frac{(\alpha_k+1)_n (\beta_k +1)_n} {(1+b)_n} \frac{(1-t)^n}{n!} \\
&  + \frac{kb}{\gamma_k} \frac{\Gamma(\gamma_k+1)\Gamma(b)}{\Gamma(\alpha_k+1)\Gamma(\beta_k+1)} (1-t)^{-b} F(\gamma_k-\alpha_k, \gamma_k -\beta_k; 1-b;1-t).
\end{aligned}
\end{align}
From this, it is immediate to see that the first three terms on the right hand side of the preceding identity is of class $O(1-t)$. Hence, by \eqref{eq:hyperfor}, we obtain
\begin{align*}
k F(\alpha_k, &\beta_k;\gamma_k;t) +  \frac{kb}{\gamma_k} F(\alpha_k+1,\beta_k+1;\gamma_k +1;t) = O(1-t)\\
& +  \frac{kb}{\gamma_k} \frac{\Gamma(\gamma_k+1)\Gamma(b)}{\Gamma(\alpha_k+1)\Gamma(\beta_k+1)} (1-t)^{-b} \Big[ 1 + \sum_{n \geqslant 1} \frac{(\gamma_k-\alpha_k)_n (\gamma_k -\beta_k)_n}{(1-b)_n} \frac{(1-t)^n}{n!} \Big],
\end{align*}
which also implies that
\[
\begin{aligned}
k F(\alpha_k,  \beta_k; \gamma_k;t) +  &\frac{kb}{\gamma_k} F(\alpha_k+1,\beta_k+1;\gamma_k +1;t) \\
= & O(1-t) +  \frac{kb}{\gamma_k} \frac{\Gamma(\gamma_k+1)\Gamma(b)}{\Gamma(\alpha_k+1)\Gamma(\beta_k+1)} (1-t)^{-b},
\end{aligned}
\]
thanks to $-b>0$ and the fact that $F(\alpha_k+1,\beta_k+1,\gamma_k+1,0)$ exists because $\gamma_k+1 -(\beta_k+1 +\alpha_k+1) =-b >0$. This together with \eqref{eq:bam} yields
\begin{align*}
\lim_{r\to 1} \Big(\frac{1-r^2}2\Big)^b f_k'(r)&=2^{-b} \frac{kb}{\gamma_k} \frac{\Gamma(\gamma_k+1)\Gamma(b)}{\Gamma(\alpha_k+1)\Gamma(\beta_k+1)} C_1 = A(b,k).
\end{align*}

\medskip\noindent\textbf{The case $b=0$}. This case is trivial because in this scenario $u$ is simply a harmonic extension of $f$. Consequently, $f_k (r) = r^k$ and therefore
\[
\lim_{r \to 1} f'_k (r) = k = A(0,k)
\]
as claimed.

\section{Sharp Sobolev trace inequality of order four on $\R_+^{n+1}$}
\label{apd-SobolevTraceSpace4}

As discussed in the introduction, \eqref{eq:SobolevTraceSpaceOrder4} can be derived from a general result due to Case by considering the model case $(\R_+^{n+1},\R^n, y^{-2}( dx^2+dy^2))$; see \cite[Corollary 1.5]{Case2018}. In this appendix, we provide a new proof of \eqref{eq:SobolevTraceSpaceOrder4}. As usual,
\[
\widehat f (\xi) = \int_{\R^n} f(x) e^{- i \langle x, \xi \rangle }dx
\]
denotes the Fourier transform of $f$.

\begin{proposition}\label{apx-propExtension4}
Given a function $u \in W^{1,2}(\R^{n})$. Any function $U \in W^{2,2}(\R_+^{n+1})$ satisfying
\begin{equation}\label{apd-4Equation} 
\Delta^2 U(x,y) = 0
\end{equation}
on the upper half space $\R_+^{n+1}$ and the boundary condition
\begin{equation}\label{apd-4Boundary} 
U(x,0)=u(x), \quad \partial_y U(x,0) = 0
\end{equation}
enjoys the following identity
\[
\int_{\R^{n+1}_+} | \Delta U(x,y)|^2 dx dy =2  \int_{\R^n} u(x)(- \Delta)^{3/2} u(x) dx.
\]
\end{proposition}

\begin{proof}
The existence and uniqueness of $U$ solving \eqref{apd-4Equation} and \eqref{apd-4Boundary} is standard. Now by taking the Fourier transform in the $x$ variable on \eqref{apd-6Equation} we arrive at
\begin{equation}\label{eq:4Fourier} 
\begin{aligned}
0 = \widehat{\Delta^2 U} (\xi,y) =& \Big( -|\xi|^2 \, \text{Id} + \frac{\partial^2}{\partial y^2}\Big)^2 \widehat U(\xi, y)\\
=&   |\xi|^4 \widehat U(\xi, y) -2 |\xi |^2  \widehat U_{yy} (\xi, y)  + \widehat U_{yyyy}(\xi, y).
\end{aligned}
\end{equation}
Thus, we obtain an ordinary differential equation of order four for each value of $\xi$. Inspired by \eqref{eq:4Fourier}, let us now consider the ODE
\begin{equation}\label{eq:4ODE} 
\phi^{(4)}   - 2 \phi''  + \phi = 0
\end{equation}
in $(0, +\infty)$. From this, it is routine to verify that $\phi \in H^4 ((0, +\infty))$. In particular, all derivatives $\phi^{(i)}$ with $i=2, 3$ vanish at infinity. 

It is an easy computation to verify that any solution $\phi$ to \eqref{eq:4ODE} satisfying the initial conditions
\[
\phi (0) = 1, \quad \phi'(0) = 0.
\]
must be of the form
\[
\phi (y) = (C_1 + C_2 y)e^{-y} + \big[ (2C_1-C_2-1)y - (C_1-1) \big]e^{y}
\]
for some constants $C_1$ and $C_2$. If, in addition, we assume that $\phi$ is bounded, then we find that $C_1 = C_2=1$, which then implies that
\[
\phi (y) = (1+y)e^{-y}.
\]
Hence we have just shown that there is a unique bounded solution $\phi$ to \eqref{eq:4ODE} satisfying $\phi (0) = 1$, $\phi'(0) = 0$. Furthermore, by direct computation, we get
\[
\int_0^{+\infty} \big( -\phi  + \phi''\big)^2 dy = 2.
\]
Now from \eqref{eq:4Fourier}, it is easy to verify that
\[
\widehat U(\xi, y) = \widehat u(\xi) \phi (|\xi| y).
\]
We now compute $\int_{\R^{n+1}_+} | \Delta U(x,y)|^2 dx dy$. By the Plancherel theorem and the relation $\widehat{\Delta U} (\xi,y) = -|\xi|^2 \, \widehat U(\xi, y) + \widehat U_{yy}(\xi, y) $, we obtain
\[\begin{aligned}
\int_{\R^{n+1}_+} | \Delta U(x,y)|^2 dx dy =& \frac 1{(2\pi)^{n}} \int_{\R^{n}}\int_0^{+\infty} \big[- |\xi|^2  \widehat u(\xi)  \phi (|\xi| y)  + |\xi|^2 \widehat u(\xi) \phi'' (|\xi| y)  \big]^2 d\xi dy \\
=& \frac 1{(2\pi)^{n}}\int_{\R^{n}}   |\xi|^4  \widehat u(\xi)^2  \int_0^{+\infty} \big[-   \phi (|\xi| y)  +  \phi'' (|\xi| y)  \big]^2 dy d\xi \\
=& \frac {J(\phi)} {(2\pi)^{n}}  \int_{\R^{n}}   |\xi|^3  \widehat u(\xi)^2 d\xi \\
= & 2 \int_{\R^n} u(x)(- \Delta)^{3/2} u(x) dx.
\end{aligned}\]
The proof is complete.
\end{proof}

We now use Proposition \ref{apx-propExtension4} to prove \eqref{eq:SobolevTraceSpaceOrder4}, namely, the following inequality holds
\[
2 \frac{\Gamma(\frac{n+3}2)}{\Gamma(\frac{n-3}2)} \omega_n^{3/n} \Big(\int_{\R^n} |U(x,0)|^{\frac{2n}{n-3}} dx\Big)^{\frac{n-3}n} \leqslant \int_{\R^{n+1}_+} |\Delta U(x,y)|^2 dx dy
\]
for functions $U$ with $\partial_y U(x,0)=0$. Indeed, for simplicity, we set $u(x) = U(x,0)$. First we apply the fractional Sobolev inequality \eqref{eq:SobolevSpaceOrderS/2} to get
\begin{equation}\label{eq:4SobolevSpaceOrder}
\int_{\R^n} u(x)(- \Delta)^{3/2} u(x) dx \geqslant \frac{\Gamma(\frac{n+3}2)}{\Gamma(\frac{n-3}2)} \omega_n^{3/n} \Big(\int_{\R^n} |u(x)|^{\frac{2n}{n-3}} dx\Big)^{\frac{n-3}n}.
\end{equation}
Then we combine the preceding inequality and Proposition \ref{apx-propExtension4} to obtain the desired inequality. Clearly, equality in \eqref{eq:SobolevTraceSpaceOrder4} holds if, and only if, equality in \eqref{eq:4SobolevSpaceOrder} occurs, which implies that $U$ must be a biharmonic extension of a function of the form
\[
c \big( 1 +  |\xi - z_0|^2 \big)^{-(n-3)/2},
\]
where $c>0$ is a constant, $\xi \in \R^n$, $z_0 \in \R^n$, and $U$ also fulfills the boundary condition $\partial_y U(x,0)=0$.



\begin{thebibliography}{999999}

\bibitem[AS64]{AS64}
\textsc{M. Abramowitz, I.A. Stegun}, 
\emph{Handbook of Mathematical Functions with Formulas, Graphs, and Mathematical Tables}, 
US Government Printing Office, Washington, DC, 1964.

\bibitem[AC15]{ac2015}
\textsc{A. G. Ache, S-Y. A. Chang},
Sobolev trace inequalities of order four, 
\textit{Duke Math. J.} \textbf{166} (2017) 2719--2748.

\bibitem[Bec92]{beckner1992}
\textsc{W. Beckner},
Sobolev inequalities, the Poisson semigroup, and analysis on the sphere $\bS^n$,
\textit{Proc. Nat. Acad. Sci. U.S.A.} \textbf{89} (1992) 4816--4819.
 
\bibitem[Bec93]{beckner1993}
\textsc{W. Beckner},
Sharp Sobolev inequalities on the sphere and the Moser--Trudinger inequality, 
\textit{Ann. of Math.} {\bf 138} (1993) 213--242.

\bibitem[Cas15a]{Case2017}
\textsc{J. Case},
Some energy inequalities involving fractional GJMS operators,
\textit{Anal. PDE} \textbf{10} (2017) 253--280.

\bibitem[Cas15b]{Case2018}
\textsc{J. Case},
Boundary operators associated with the Paneitz operator,
\textit{Indiana Univ. Math. J.} \textbf{67} (2018) 293--327.

\bibitem[CC14]{casechang2013}
\textsc{J. Case, S-Y. A. Chang},
On fractional GJMS operators, 
\textit{Commun. Pure Appl. Math.} {\bf 69} (2016) 1017--1061.

\bibitem[CL18]{Case}
\textsc{J. Case, W. Luo},
Boundary operators associated to the sixth-order GJMS operator,
preprint, arXiv:1810.08027.

\bibitem[CW17]{ChangWang2017}
\textsc{S-Y. A. Chang, F. Wang},
Limit of fractional power Sobolev inequalities, 
\textit{J. Funct. Anal.} {\bf 274} (2018) 1177--1201.

\bibitem[CY17]{rayyang2013}
\textsc{S-Y. A. Chang, R. Yang},
On a class of non-local operators in conformal geometry, 
\textit{Chin. Ann. Math. Ser. B} {\bf 38} (2017) 215--234.

\bibitem[CT04]{CT04}
\textsc{A. Cotsiolis, N.K. Tavoularis},
Best constants for Sobolev inequalities for higher order fractional derivatives,
\textit{J. Math. Anal. Appl.} \textbf{295} (2004) 225--236. 

\bibitem[EL12]{EL12}
\textsc{A. Einav, M. Loss}, 
Sharp trace inequalities for fractional Laplacians,
\textit{Proc. Amer. Math. Soc.} \textbf{140} (2012) 4209--4216.

\bibitem[Esc88]{Escobar88}
\textsc{J.F. Escobar}, 
Sharp constant in a Sobolev trace inequality, 
\textit{Indiana Univ. Math. J.} \textbf{37} (1988) 687--698.

\bibitem[GZ03]{GZ}
\textsc{C.R. Graham, M. Zworski},
Scattering matrix in conformal geometry,
\textit{Invent. Math.} \textbf{152} (2003) 89--118. 

\bibitem[Han07]{Hang}
\textsc{F.B. Hang}, 
On the higher order conformal covariant operators on the sphere,
\textit{Commun. Contemp. Math.} \textbf{9} (2007) 279--299.

\bibitem[JN14]{GN}
\textsc{G. Jankowiak, V. H. Nguyen},
Fractional Sobolev and Hardy--Littlewood--Sobolev inequalities, preprint, arXiv:1404.1028.

\bibitem[JX13]{JinXiong}
\textsc{T. Jin, J. Xiong},
Sharp constants in weighted trace inequalities on Riemannian manifolds, 
\textit{Calc. Var. Partial Differential Equations} \textbf{48} (2013) 555--585. 

\bibitem[LM51]{LM}
\textsc{N.A. Lebedev, I.M. Milin},
On the coefficients of certain classes of analytic functions,
\textit{Mat. Sbornik N.S.} \textbf{28} (1951) 359--400.

\bibitem[Lie83]{Lieb1983}
\textsc{E. Lieb},
Sharp constants in the Hardy--Littlewood--Sobolev and related inequalities,
\textit{Ann. Math.} \textbf{118} (1983) 349--374.

\bibitem[Lio85]{Lions}
\textsc{P.L. Lions}, 
The concentration--compactness principle in the calculus of variations. The locally compact case II, 
\textit{Rev. Mat. Iberoamericana} \textbf{1} (1985) 45--121.

\bibitem[OPS88]{OPS}
\textsc{B. Osgood, R. Phillips, P. Sarnak},
Extremals of determinants of Laplacians, 
\textit{J. Funct. Anal.} \textbf{80} (1988) 148--211.

\bibitem[Xio18]{Xiong2018}
\textsc{J. Xiong},
A derivation of the sharp Moser--Trudinger--Onofri inequalities from the fractional Sobolev inequalities
preprint, arXiv:1804.02807.

\end{thebibliography}
\end{document}